\documentclass[11pt]{amsart}
\usepackage{mathrsfs}
\usepackage{amsmath}
\usepackage{amsfonts}
\usepackage{amssymb}
\usepackage{amsxtra}
\usepackage{mathtools}

\usepackage{dsfont}

\usepackage{graphicx}
\usepackage[english,polish]{babel}

\usepackage[T1]{fontenc}
\usepackage{newtxmath}
\usepackage{newtxtext}
\usepackage[margin=2.5cm, centering]{geometry}

\usepackage[compress, sort]{cite}

\usepackage[colorlinks,citecolor=blue,urlcolor=blue,bookmarks=true]{hyperref}
\hypersetup{
pdfpagemode=UseNone,
pdfstartview=FitH,
pdfdisplaydoctitle=true,
pdfborder={0 0 0}, 
pdftitle={Asymptotic behavior of heat kernels and Green functions on affine buildings},
pdfauthor={Bartosz Trojan},
pdflang=en-US
}

\newcommand{\pl}[1]{\foreignlanguage{polish}{#1}}

\newcommand{\ZZ}{\mathbb{Z}}
\newcommand{\RR}{\mathbb{R}}
\newcommand{\NN}{\mathbb{N}}
\newcommand{\CC}{\mathbb{C}}

\newcommand{\bfc}{\mathbf c}
\newcommand{\bfb}{\mathbf b}

\newcommand{\calP}{\mathcal{P}}
\newcommand{\calQ}{\mathcal{Q}}
\newcommand{\calR}{\mathcal{R}}
\newcommand{\calO}{\mathcal{O}}
\newcommand{\calI}{\mathcal{I}}
\newcommand{\calV}{\mathcal{V}}
\newcommand{\calM}{\mathcal{M}}
\newcommand{\calF}{\mathcal{F}}

\newcommand{\scoef}[2]{\frac{c_{#2} e^{\langle #1, #2\rangle}}{\kappa(#1)}}
\newcommand{\sprod}[2]{{\langle {#1}, {#2}\rangle}}
\newcommand{\norm}[1]{{\left\lvert {#1} \right\rvert}}
\newcommand{\mnorm}[1]{{\left\lVert {#1} \right\rVert}}
\newcommand{\abs}[1]{{\lvert {#1} \rvert}}
\newcommand{\der}[1]{\partial^{#1}}

\DeclareMathOperator{\Log}{Log}

\newcommand{\brho}[1]{\big(B_{x_0}^{#1} \rho\big)}
\newcommand{\btheta}{\big\lvert B_{x_0}^{1/2} \theta \big\rvert}

\newcommand{\dth}{{\: \rm d}\theta}

\newcommand{\cl}{\operatorname{cl}}
\newcommand{\conv}{\operatorname{conv}}
\newcommand{\GL}{\operatorname{GL}}
\newcommand{\Aff}{\operatorname{Aff}}

\newcommand{\supp}{\operatornamewithlimits{supp}}
\newcommand{\cspan}{\operatorname{\mathbb{C}-span}}
\newcommand{\rspan}{\operatorname{\mathbb{R}-span}}
\newcommand{\zspan}{\operatorname{\mathbb{Z}-span}}
\newcommand{\dist}{\operatorname{dist}}
\newcommand{\bpi}{\boldsymbol{\pi}}

\renewcommand{\atop}[2]{\substack{{#1}\\{#2}}}

\newcounter{thm}

\newtheorem{theorem}{Theorem}
\newtheorem{proposition}{Proposition}[section]
\newtheorem{lemma}{Lemma}
\newtheorem{corollary}{Corollary}
\newtheorem{claim}{Claim}
\newtheorem*{theorem*}{Theorem}
\newtheorem{main_theorem}[thm]{Theorem}

\theoremstyle{definition}
\newtheorem{definition}{Definition}

\newtheorem{remark}{Remark}

\begin{document}
\selectlanguage{english}

\title[Heat kernels and Green functions]
{Asymptotic behavior of heat kernels and Green functions on affine buildings}

\author{Bartosz Trojan}
\thanks{The research was partial supported by the National Science Centre, Poland, Grant 2016/23/B/ST1/01665}

\address{
	\pl{
	Bartosz Trojan\\
	The Institute of Mathematics\\
	Polish Academy of Sciences\\
	ul. \'Sniadeckich 8\\
	00-696 Warszawa\\
	Poland}
}

\curraddr{
	\pl{
    Wydzia\l{} Matematyki\\
	Politechnika Wroc\l{}awska\\
	Wyb. Wyspia\'{n}skiego 27\\
	50-370 Wroc\l{}aw\\
	Poland}
}
\email{bartosz.trojan@gmail.com}


\begin{abstract}
	We obtain asymptotic formulas for the transition densities
	$p(n; x,y)$ of finite range isotropic random walks on affine buildings.
	We also describe the asymptotic behavior of the corresponding Green functions.
\end{abstract}



\maketitle

\section{Introduction}
\label{sec:1}
Solving the heat equation led to develop at least two fundamental tools in modern mathematics, namely
the Fourier transform and the heat kernel. Harmonic analysis, i.e. the mathematical study of the
Fourier transform, is one of the main tools of this paper, while the heat kernel is its main object
of study. The latter can be constructed in the context of Riemannian geometry, see e.g. \cite{BGV}, leading to a deep
interplay between the analytic behavior of the heat kernel and the geometric properties of
the considered manifold. The better understood the ambient manifold, the more precise the information
on the heat kernel is expected. This paper is dedicated to studying the kernels on some singular spaces,
called affine buildings, in tight connection with Lie theory, more precisely with non-Archimedean Lie groups
\cite{BruhatTits1972}. The probabilistic viewpoint is systematically considered since it is particularly well adapted
to these singular spaces with strong symmetry properties. Some related results on associated Green functions are also
derived. The latter has a deep connection with potential theory, see e.g. \cite{Doob2001}.

In order to better motivate our study, let us first consider the case of Riemannian symmetric spaces of non-compact
type, called symmetric spaces for short, since affine buildings are the non-Archimedean counterparts to 
the latter manifolds; both situations are complementary pieces at the heart of Lie theory \cite{Helgason1979, Tits1977}.
On a symmetric space the heat kernel $h_t$ for the heat semigroup 
$e^{t\Delta}$ where $\Delta$ is the Laplace--Beltrami operator is a basic and well-studied object.
Estimates as well as asymptotic of $h_t$ play a fundamental r\^ole in studying geometry of the underlying space. 
Initial studies of $h_t$ were carried out by Sawyer \cite{saw1, saw2, saw3} and Anker 
\cite{a2, a3}. In \cite{aj} Anker and Ji proved sharp estimates on the kernel $h_t(x)$ whenever 
$\norm{x}$ is smaller than some constant multiple of $1 + t$. Global estimates were subsequently 
found by Anker and Ostellari \cite{ao2, ao}. These results have important applications. One is
to determine the behavior of the Green function which is the analytic input needed to describe
the Martin boundary \cite{gjt, gu}. In \cite{gjt} Guivarc'h, Ji and Taylor used results obtained
in \cite{aj} to construct the Martin compactification of symmetric spaces.

Apart from strong topological differences, real and non-Archimedean simple Lie groups share many combinatorial
and geometric properties. From the geometric point of view they both act, with strong transitivity properties, on
contractible spaces carrying nice non-positively curved complete distances. In the real case, these are of course
the symmetric spaces. The corresponding spaces for groups over totally disconnected local fields are \emph{Bruhat--Tits
buildings}. Bruhat--Tits buildings, or more generally \emph{affine buildings} when no group is assumed to act transitively 
on them, are unions of Euclidean tilings, called apartments, playing the r\^ole of maximal flats in symmetric spaces. 
Apartments contain Weyl cones, also called sectors, in which, thanks to a polar decomposition of the group,
the behavior of the heat kernel is suitably described. 

To get a deeper and more precise understanding of symmetric spaces many authors have
studied the corresponding problems on appropriate graphs. In this context, Guivarc'h, Ji and 
Taylor emphasize the importance of extending all the compactification procedures to Bruhat--Tits
buildings associated with reductive groups over $p$-adic fields. The group-theoretic part of
this program has been carried out by Guivarc'h and R\'emy in \cite{gure}. A basic important problem
raised in \cite{gjt} is to describe the asymptotic behavior of the Green function of a finite
range isotropic random walk on a Bruhat--Tits building. One way to understand the Green function is to
obtain the asymptotic formula for the corresponding heat kernel which in this context is the $n$'th 
iteration $p(n; \cdot, \cdot)$ of the transition operator given by the transition density $p$
(see Section \ref{sec:4.1} for definitions).

In the present paper we obtain the uniform asymptotic formula for the heat kernel $p(n; \cdot, \cdot)$,
giving a definitive answer to the question posed in \cite{gjt}. The heart of the matter is a 
detailed description of the off-diagonal behavior of $p(n; x, y)$. We achieve this for all affine buildings
in particular those with small or possibly trivial automorphism group. There has been considerable work done giving
on-diagonal estimates, i.e for $p(n; x, x)$ (see \cite{carwoe,park1} for instance). However, let us emphasize that one
needs to understand the off-diagonal part of the heat kernel $p(n; x, y)$ for the Green function asymptotic. Moreover,
the off-diagonal estimates can not be deduced from the on-diagonal results. We are able to establish
the asymptotic formula for $p(n; x, y)$, uniformly in a region asymptotically approaching the building-theoretic
analog of Cram\'er's zone which we denote by $\calM$. The main result of the paper is Theorem \ref{thm:4}, see also
Corollary \ref{cor:1} for its weaker version which is good enough for most applications. The asymptotic behavior of
the Green function is described in Theorems \ref{thm:6} and \ref{thm:7}. The Martin compactification
of affine buildings is the subject of the forthcoming paper \cite{RemyTrojan}.

Random walks on affine buildings have been studied for over thirty years. In 1978 Sawyer 
\cite{saw} obtained the asymptotic of $p(n; x, x)$ for homogeneous trees, i.e. affine buildings 
of type $\tilde{A}_1$. This is called \emph{Local Limit Theorem}. The result was 
extended to $\widetilde{A}_r$ by Tolli \cite{tol}, Lindlbauer and Voit \cite{linv} and Cartwright
and Woess \cite{carwoe}. Lastly, Local Limit Theorems for all affine buildings were proved by
Parkinson \cite{park1}.

Local Limit Theorems describe the behavior of $p(n; x, y)$ for fixed $x$ and $y$. However, in many applications it is 
desired to know the uniform asymptotic behavior in a large spacetime regime. For affine buildings it 
was previously studied in two cases only. For homogeneous
trees, uniform asymptotic were found by Lalley \cite{lal1,lal2}. For affine buildings of higher rank, 
the first results were obtained by Anker, Schapira and the author in \cite{ascht} where for each building of
type $\tilde{A}_r$, a \emph{distinguished} averaging operator was studied. We obtain sharp upper and lower bounds
on $p(n; x, y)$. In this paper we treat all affine buildings.

To be more precise, we need to introduce some notation. Let $\Phi$ be the type of the building, that is $\Phi$
is the affine root system in $\mathfrak{a}$, where $\mathfrak{a}$ is the Euclidean space on which apartments are modeled,
and let $\Phi^{++}$ denote the set of indivisible positive roots in $\Phi$.
By $W_0$ we denote the corresponding (spherical) Weyl group. Given a transition function $p$ of the isotropic finite
range random walk on good vertices $V_P$ of the building, we define the corresponding averaging operator acting on
functions on $V_P$ as
\[
	A f(x) = \sum_{y \in V_P} p(x, y) f(y).
\]
Then the Gelfand--Fourier transform of $A$, denoted as $\hat{A}$, is a $W_0$-invariant exponential polynomial
expressed as a combination of Macdonald spherical function $P_\omega$. The Cram\'er's zone $\mathcal{M}$ is the interior of
the convex hull in $\mathfrak{a}$ of the support of $\hat{A}$. For $\delta \in \mathcal{M}$, we set
\[
	\phi(\delta) = \max\big\{\sprod{x}{\delta} - \log \kappa(x) :  x \in \mathfrak{a} \big\}
\]
where $\kappa = \varrho^{-1} \hat{A}$ and $\varrho$ is the spectral radius of the random walk. We also need a quadratic
form on $\mathfrak{a}$ given by $B_0(u, u) = D_u^2 \log \kappa(0)$. Let us recall that each apartment of affine building
contains as a discrete subspace the coweight lattice of the root system $\Phi_0$, so the statement below completely
describes the process in the building.
\begin{main_theorem}
	\label{thm:A}
	Let $(\omega_n : n \in \NN)$ be a sequence of co-weights such that the sphere centered at $O$ and 
	radius $\omega_n$ is contained in the support of $p(n; O, \cdot\,)$. Suppose that 
	\[
		\lim_{n \to \infty} \sprod{\delta_n}{\alpha} = 0, \quad\text{for all}\quad \alpha \in \Phi
	\]
	where $\delta_n = n^{-1} \omega_n$. Then for any sequence of good vertices $(x_n : n \in \NN)$ such that
	the Weyl distance between $O$ and $x_n$ equals $\omega_n$, we have
	\[
		p(n; O, x_n) = n^{-\frac{r}{2}-\abs{\Phi^{++}}} P_{\omega_n}(0) \varrho^n e^{-n\phi(\delta_n)}
		\Big(C_0 + \calO\big(\norm{\delta_n}\big) + \calO\big(n^{-1}\big) \Big).
	\]
	The constant $C_0$ is absolute.
\end{main_theorem}
Theorem \ref{thm:A} is the direct consequence of a more general Theorem \ref{thm:4} which contains the detailed 
description of the asymptotic behavior of $p(n; \, x , y)$. We emphasize that the main difficulties lay in the
asymptotic analysis along the walls of the Weyl chamber when $\delta_n$ approaches the boundary of $\calM$.
Using Theorem \ref{thm:4}, we study the asymptotic behavior of the Green function $G_{\zeta}$ for
$\zeta \in (0, \varrho^{-1}]$. In particular, at the bottom of the spectrum, we obtain the following asymptotic formula
for the Green function $G_{\varrho^{-1}}$, see Theorem \ref{thm:7}. For the detailed study of the Green function
above the bottom of the spectrum see Theorem \ref{thm:6}.
\begin{main_theorem}
	For all $x \in V_P$, so that the Weyl distance between $O$ and $x$ equals $\omega$, we have
	\[
		G_{\varrho^{-1}}(O, x) = 
		P_{\omega}(0) \big(B_0^{-1}(\omega, \omega)\big)^{-\frac{r}{2}-\abs{\Phi^{++}}+1}
		\big(D_0 + o(1) \big),
	\]
	as $\norm{\omega}$ tends to infinity. The constant $D_0$ is absolute.
\end{main_theorem}
This paper is analytic in its nature as far as the tools of the proofs are concerned. However there are strong connections 
with Lie combinatorics associated to parametrization of representations. This is a well-known phenomenon in the field,
illustrated for instance by the case of the Fourier transform on non-Archimedean Lie groups. The starting point in
spherical harmonic analysis is to exhibit a suitable Gelfand pair: this was done by Satake in the 60’s who also showed
a combinatorial parametrization of the spherical functions providing the desired Fourier transform \cite{Satake1963}. 
The exact computation of the latter functions was achieved by Macdonald, leading to an explicit description of the
involved Plancherel measure \cite{macdo0}. The situation is so well understood now that it can be made completely
geometric, i.e. without any use of group action. The importance of the geometric approach is important since not every
affine building corresponds to a group of $p$-adic type. For exotic buildings of type $\tilde{A}_2$, Cartwright and
\pl{Młotkowski} \cite{CarMlo} proposed a construction of the spherical Fourier transform using the geometric and
combinatorial properties of the building. This approach was extended by Cartwright \cite{car} to buildings of type
$\tilde{A}_r$ and by Parkinson \cite{park2} to all affine buildings.

Let us now give a brief sketch of the proof of Theorem \ref{thm:4}. As usual, an application of the spherical
Fourier transform results in an oscillatory integral. Thanks to some geometric properties of the
support of the spherical Fourier transform of $p$, see Theorem \ref{thm:2}, the integral can be localized to
$\{\theta \in \mathfrak{a} : \norm{\theta} \leq \epsilon \}$. Therefore, the proof reduces to establishing
the asymptotic behavior, as $n$ approaches infinity, of
\[
	F_n(x) = \int\limits_{\norm{\theta} \leq \epsilon} 
	e^{n \varphi(x, \theta)} \frac{{\: \rm d} \theta}{\bfc(x + i \theta)},
\]
uniformly with respect to $x \in \cl \mathfrak{a}_+$ where $\mathfrak{a}_+$ is the Weyl chamber of
the underlying root system, and $\bfc$ is the non-Archimedean counterpart of Harish-Chandra $\bfc$-function. 
The function $\varphi$ is related to $\kappa$, for the definition we refer to \eqref{eq:62}. This paper 
relies on the study of oscillatory integrals in a uniform manner, and its core is contained in Theorem \ref{thm:3}
where the asymptotic behavior of $F_n$ along the walls is investigated. We use a variant of the steepest descent method.
However, there is an interplay between the time $n$ and the distance of $x$ to the walls.
Therefore, to identify the leading terms we need to utilize combinatorics of subroot systems.
In fact, if $x$ lies on a certain wall of $\mathfrak{a}_+$ then the function 
$\varphi(x, \cdot)$ retain symmetries in the directions orthogonal to that wall. Close to the wall
we take advantage of this by expanding $F_n$ into power series and using combinatorial methods
we identify remaining cancellations. In \cite{ascht}, a key combinatorial formula available for a
distinguished averaging operator allowed to avoid the difficult analysis of cancellations.

\subsection{Organization of the paper}
\label{subsec:1.1}
In Section \ref{sec:2.1} we present the definition of a function $s$ which underpins all
estimates for $p(n; \cdot, \cdot)$. We next prove two auxiliary lemmas: one analytic and
one combinatorial. In Section \ref{sec:4} the definitions of root systems and
affine buildings are recalled, and a number of spherical-analytic facts used across the paper are
collected. The main theorem is stated and proved in Section \ref{sec:4.4}. As an application,
the asymptotic behavior of the corresponding Green function are found (Theorem \ref{thm:6} and Theorem \ref{thm:7}).

We use the convention that $C, C', c, c', \ldots$ stands for a generic positive constant whose value can change
from line to line.

\subsection*{Acknowledge}
The author expresses his gratitude to Jean-Philippe Anker, Jacek \pl{Dziubański}, Bertrand R\'emy, Tim Steger,
and Ryszard Szwarc for extensive discussions, comments and support.

\section{Combinatorial and analytic preliminaries}
\label{sec:3}
This section contains most of the preliminaries necessary to the technical arguments used in this paper. 
This explains why it is varied in nature. The first subsection is dedicated to convex combinations of exponentials
in Euclidean spaces; they appear naturally in the study of random walks in spaces governed by Lie-theoretic data.
The optimization problem leads to defining a function providing asymptotic directions of random walks.
The second subsection is dedicated to multiple derivation while the last one contains a variation of the marriage
lemma useful later to handle root system combinatorics.

\subsection{Convex combinations of exponentials and the function s}
\label{sec:2.1}
Let $\mathfrak{a}$ be a $r$-dimensional real vector space with an inner product
$\sprod{\cdot}{\cdot}$. By $\mathfrak{a}_{\mathbb{C}}$ we denote its complexification. 
We fix a finite set of vectors $\mathcal{V} \subset \mathfrak{a}$ and a set of positive constants
$\{c_v: v \in \mathcal{V}\}$ satisfying $\sum_{v \in \mathcal{V}} c_v = 1$. Let
$\kappa: \mathfrak{a}_{\mathbb{C}} \rightarrow \mathbb{C}$ be a function given by the formula
\[
	\kappa(z) = \sum_{v \in \mathcal{V}} c_v e^{\sprod{z}{v}}.
\]
The motivation to study the function $\kappa$ comes from random walks. It is ultimately connected to the
Gelfand--Fourier transform of the corresponding averaging operator, see Section \ref{sec:4.1} for details.

For $x \in \mathfrak{a}$, by $B_x$ we denote a quadratic form $B_x(u, u) = D_u^2 \log \kappa(x)$ where $D_u$ is
the derivative along a vector $u$, i.e.,
\[
	D_u f(x) = \left. \frac{{\rm d}}{{\rm d} t} f(x + t u) \right|_{t = 0}.
\]
Since
\[
	D_u \log \kappa(x) = \sum_{v \in \calV} \scoef{x}{v} \sprod{u}{v},
\]
and
\[
	D_u \bigg(\frac{c_v e^{\sprod{x}{v}}}{\kappa(x)} \bigg) =
	\scoef{x}{v}\sprod{u}{v} - \sum_{v' \in \calV} \scoef{x}{v} \cdot \scoef{x}{v'} \sprod{u}{v'},
\]
we may write
\begin{equation}
	\label{eq:1}
	B_x(u,u) = 
	\frac{1}{2} \sum_{v, v' \in \mathcal{V}} \scoef{x}{v} \cdot \scoef{x}{v'} \sprod{u}{v-v'}^2.
\end{equation}
Let $\mathcal{M}$ be the interior of the convex hull of $\mathcal{V}$. We assume that $\mathcal{M}$
is \emph{not empty}. For the sake of completeness we provide the proof of the following well-known 
theorem.
\begin{theorem}
	\label{thm:1}
	For every $\delta \in \mathcal{M}$ a function $f(\delta,\cdot): \mathfrak{a} \rightarrow
	\mathbb{R}$ defined by
	$$
	f(\delta, x) = \sprod{x}{\delta} - \log \kappa(x)
	$$
	attains its maximum at the unique point $s \in \mathfrak{a}$ satisfying
	$\nabla \log \kappa(s) = \delta$.
\end{theorem}
\begin{proof}
	Without loss of generality, we may assume that $\nabla \kappa(0)=0$. Indeed, otherwise we will consider
	\[
		\tilde{\kappa}(z) = e^{-\sprod{z}{v_0}} \kappa(z) = \sum_{v \in \tilde{\calV}} c_{v + v_0} e^{\sprod{z}{v}}
	\]
	where $v_0 = \nabla \kappa(0)$, and $\tilde{\calV} = \calV - v_0$. Then $\tilde{\calM}$, the interior of the
	convex hull of $\tilde{V}$, equals $\calM - v_0$. For $\tilde{\delta} = \delta - v_0$, we have
	\[
		\tilde{f}(\tilde{\delta}, x) = \sprod{x}{\delta - v_0} - \log \tilde{\kappa}(x) 
		= \sprod{x}{\delta} - \log \kappa(x) = f(\delta, x).
	\]
	We conclude that if $s$ is the unique maximum of $\mathfrak{a} \ni x \mapsto \tilde{f}(\tilde{\delta}, x)$,
	then it is also the unique maximum of $\mathfrak{a} \ni x \mapsto f(\delta, x)$. Because
	\[
		\nabla \log \tilde{\kappa}(x) = \nabla \log \kappa(x) - v_0,
	\]
	we get $\nabla \log \kappa(s) = \tilde{\delta} - v_0 = \delta$, proving the claim.

	Fix $\delta \in \calM$. Since $\nabla \kappa(0) = 0$, by Taylor's theorem we have
	\[
		f(\delta, x) =  \sprod{x}{\delta} + \mathcal{O}(\norm{x}^2)
	\]
	as $\norm{x}$ approaches zero. Moreover, for any $x, u \in \mathfrak{a}$,
	\[
		D_u^2 f(\delta, x) = -B_x(u, u),
	\]
	thus the function $\mathfrak{a} \ni x \mapsto f(\delta, x)$ is strictly concave.

	Let us observe that 
	\[
		0 = \nabla \kappa(0) = \sum_{v \in \calV} c_v \cdot v \in \cl \calM.
	\]
	Since $\calM$ is not empty, the set $\calV$ cannot be contained in an affine hyperplane, thus, $0 \in \calM$.
	
	Now, $\delta \in \calM$ implies that there are $v_1, \ldots, v_r \in \partial \calM \cap \calV$ such that $\delta$
	belongs to the convex hull of $\{0, v_1, \ldots, v_r\}$, i.e. there are $t_0, t_1, \ldots t_r \in [0, 1]$
	satisfying
	\[
		\delta = t_0 \cdot 0 + \sum_{j=1}^r t_j \cdot v_j = \sum_{j=1}^r t_j \cdot v_j.
	\]
	Because $\delta \not \in \partial \calM$ we must have $t_0 > 0$, thus $\sum_{j = 1}^r t_j < 1$. Hence,
	\[
		\sum_{j=1}^r t_j \log \kappa(x) 
		\geq \sum_{j=1}^r t_j \big(\log c_{v_j} + \sprod{x}{v_j}\big)
		=\sum_{j=1}^r t_j \log c_{v_j} + \sprod{x}{\delta},
	\]
	and we get
	\begin{equation}
		\label{eq:69}
		f(\delta, x) = 
		\sprod{x}{\delta} - \log \kappa(x)
		\leq \Big(\sum_{j=1}^r t_j - 1\Big) \log \kappa(x) - \sum_{j=1}^r t_j \log c_{v_j}.
	\end{equation}
	Because
	\[
		\lim_{\norm{x} \to \infty} \log \kappa(x) = +\infty,
	\]
	the estimate \eqref{eq:69} implies that
	\[
		\lim_{\norm{x} \to \infty} f(\delta, x) = -\infty,
	\]
	and the proof is finished.
\end{proof}

In this article, for a given $\delta \in \calM$, we denote by $s \in \mathfrak{a}$ the unique solution to
\begin{equation}
	\label{eq:34}
	\delta = \nabla \log \kappa(s) = \sum_{v \in \calV} \frac{c_v e^{\sprod{s}{v}}}{\kappa(s)} \cdot v.
\end{equation}
Let $\phi: \mathcal{M} \rightarrow \mathbb{R}$ be defined by
\[
	\phi(\delta) = \max\big\{\sprod{x}{\delta} - \log \kappa(x) : x \in \mathfrak{a}\big\},
\]
thus, by Theorem \ref{thm:1},
\[
	\phi(\delta) = \sprod{\delta}{s} - \log \kappa(s).
\]
By \eqref{eq:34}, for any $u \in \mathfrak{a}$,
\[ 
	\sprod{\delta}{u} = D_u \log \kappa(s).
\]
Hence, for $u, u' \in \mathfrak{a}$,
\[
	\sprod{u}{u'} 
	=D_u \big(D_{u'} \log \kappa(s) \big)
	=\sum_{j = 1}^d D_j D_{u'} \log \kappa(s) D_u s_j
	= B_s(D_u s, u'),
\]
i.e., $D_u s = B_s^{-1} u$. Therefore, we can calculate
\[ 
	\nabla \phi(\delta) 
	= s + \sum_{j = 1}^d \delta_j \nabla s_j - \sum_{j =1}^d D_j \log \kappa(s) \nabla s_j
	= s,
\]
thus, 
\[
	D^2_u \phi(\delta) = D_u \big( \sprod{u}{s} \big) = B_s^{-1}(u,u).
\]
In particular, $\phi$ is a convex function on $\calM$. Let $\delta_0 = \nabla \log \kappa(0)$. By Taylor's theorem,
we have
\begin{equation}
	\label{eq:29}
	\phi(\delta) = \frac{1}{2} B_0^{-1}(\delta - \delta_0, \delta-\delta_0) +
	\mathcal{O}(\norm{\delta-\delta_0}^3)
\end{equation}
as $\delta$ approaches $\delta_0$. We claim that
\footnote{$A \asymp B$ means that $c B \leq A \leq C B$, for some constants $c,C>0$.}
for all $\delta \in \calM$,
\begin{equation}
	\label{eq:33}
	\phi(\delta) \asymp B_0^{-1}(\delta - \delta_0, \delta - \delta_0).
\end{equation}
Since $\phi$ is convex and satisfies \eqref{eq:29}, it is enough to show that $\phi$ is bounded from above.
Given $\delta \in \calM$, let $v_0 \in \calV$ be any vector satisfying
\[
	\sprod{s}{v_0} = \max\big\{\sprod{s}{v} : v \in \calV\big\}.
\]
Because
$$
\sprod{s}{\delta} - \sprod{s}{v_0} =
\sum_{v \in \mathcal{V}}
\frac{c_v e^{\sprod{s}{v}} }{\kappa(s)}\sprod{s}{v-v_0} \leq 0,
$$
we get 
\[
	\phi(\delta) = \sprod{s}{\delta} - \log \kappa(s) 
	\leq \sprod{s}{\delta} - \log \big(c_{v_0} e^{\sprod{s}{v_0}}\big) 
	\leq -\log c_{v_0},
\]
proving \eqref{eq:33}.

In general, there is no explicit formula for the function $\phi$. By the implicit function theorem, the function
$s$ is real-analytic on $\calM$. In particular, $s$ is bounded on any compact subset of $\calM$. From the other
side, $\norm{s}$ approaches infinity when $\delta$ tends to $\partial \calM$.
To see this, let us denote by $\calF$ a facet of $\calM$ such that $\delta$ approaches $\partial \calM \cap \calF$. Let
$u$ be an outward unit normal vector to $\calM$ at $\calF$. Then for each $v' \in \calF \cap \calV$ and 
$v'' \in \calV \setminus \calF$ we have
\begin{align*}
	\sprod{v' - \delta}{u} 
	& = \sum_{v \in \calV} \frac{c_v e^{\sprod{s}{v}}}{\kappa(s)} \sprod{v' - v}{u} \\
	& = \sum_{v \in \calV \setminus \calF} \frac{c_v e^{\sprod{s}{v}}}{\kappa(s)} \sprod{v' - v}{u}
	\geq \frac{c_{v''} e^{\sprod{s}{v''}}}{\kappa(s)} \sprod{v' - v''}{u}.
\end{align*}
Therefore, for any $v \in \calV \setminus \calF$,
\begin{equation}
	\label{eq:11}
	\lim_{\delta \to \partial \calM \cap \calF} \frac{e^{\sprod{s}{v}}}{\kappa(s)} = 0.
\end{equation}
The next theorem provides a control over the speed of convergence in \eqref{eq:11}. 
\begin{theorem}
	\label{thm:2}
	There are constants $\eta \geq 1$ and $C > 0$ such that for all $\delta \in \mathcal{M}$, and $v \in \mathcal{V}$
	we have
	\[
		\frac{e^{\sprod{s}{v}}}{\kappa(s)}
		\geq C \dist(\delta, \partial\mathcal{M})^\eta
	\]
	where $s=s(\delta)$ satisfies $\delta = \nabla \log \kappa(s)$. 
\end{theorem}
\begin{proof}
	We consider any enumeration of elements of $\calV = \{v_1, \ldots, v_N\}$. Define
	$$
	\Omega = \big\{\omega \in S^{r-1}:
		\sprod{\omega}{v_{i}} \geq \sprod{\omega}{v_{i+1}} \text{ for } i=1,\ldots,N-1 \big\}
	$$
	where $S^{r-1}$ is the unit sphere in $\mathfrak{a}$ centered at the origin. Since $\calV$ is finite, it is enough
	to prove that there are $C > 0$ and $\eta \geq 1$ such that for all $x \in \mathfrak{a}$, if
	$\frac{x}{\norm{x}} \in \Omega$ then for all $v \in \calV$,
	\[
		\frac{e^{\sprod{x}{v}}}{\kappa(x)}
		\geq
		C
		\dist(\delta, \partial \calM)^{\eta}
	\]
	where
	\[
		\delta = \sum_{v \in \calV} \frac{c_v e^{\sprod{x}{v}}}{\kappa(x)} \cdot v.
	\]
	Without loss of generality, we may assume that $\Omega \neq \emptyset$. Let $k$ be the smallest index such that
	points $\{v_1, \ldots, v_k\}$ do not lay on the same facet of $\calM$. Let us recall that a set $\calF$ is a facet of
	$\calM$ if there are $\lambda \in S^{r-1}$ and $c \in \RR$ such that for all $v \in \calV$,
	$\sprod{\lambda}{v} \leq c$, and
	\[
		\calF = \conv\{v \in \calV : \sprod{\lambda}{v} = c\}.
	\]
	Since $\{v_1, \ldots, v_k\}$ do not lay on the same facet of $\calM$ and $\Omega$ is a compact set, there is 
	$\epsilon > 0$ such that for all $\omega \in \Omega$ we have
	\begin{equation}
		\label{eq:8}
		\sprod{\omega}{v_1} \geq \sprod{\omega}{v_k} + \epsilon.
	\end{equation}
	Indeed, otherwise, there are $\omega_n \in \Omega$ such that
	\[
		\sprod{\omega_n}{v_k} \leq \sprod{\omega_n}{v_1} \leq \sprod{\omega_n}{v_k} + \frac{1}{n}.
	\]
	Since $\Omega$ is compact, there is $\omega_0 \in \Omega$ such that
	\[
		\sprod{\omega_0}{v_1} = \sprod{\omega_0}{v_k},
	\]
	and for each $i \in \{2, \ldots, N\}$,
	\[
		\sprod{\omega_0}{v_1} \geq \sprod{\omega_0}{v_i}.
	\]
	This contradicts that $\{v_1, \ldots, v_k\}$ do not lay on the same facet of $\calM$.

	Let $\calF$ be a facet containing $\{v_1,\ldots,v_{k-1}\}$ determined by $\lambda \in S^{r-1}$ and $c \in \RR$.
	Let us consider $x \in \RR^d$ such that $\frac{x}{\norm{x}} \in \Omega$ and
	$$
	\delta = \sum_{v\in\mathcal{V}} \frac{c_v e^{\sprod{x}{v}}}{\kappa(x)} \cdot v.
	$$
	The distance of $\delta$ to a plane containing the facet $\calF$ is not bigger than $c - \sprod{\lambda}{\delta}$, thus
	\begin{align*}
		\dist(\delta, \partial \mathcal{M}) 
		\leq 
		c - \sprod{\lambda}{\delta}
		& = 
		\sum_{v\in\mathcal{V}\setminus\mathcal{F}} 
		\frac{c_v e^{\sprod{x}{v}}}{\kappa(x)}
		\sprod{\lambda}{v_1-v} \\
		& \leq 
		2 \max\{\norm{v} : v \in \calV\} \cdot \frac{e^{\sprod{x}{v_k}}}{\kappa(x)}.
	\end{align*}
	Since
	\[
		c_{v_1} e^{\sprod{x}{v_1}} \leq \kappa(x) \leq e^{\sprod{x}{v_1}},
	\]
	we obtain 
	\[
		e^{\sprod{x}{v_k - v_1}} \geq 
		C \dist(\delta, \partial \calM).
	\]
	In particular, for $1 \leq j \leq k$, we have
	\[
		\frac{e^{\sprod{x}{v_j}}}{\kappa(x)} \geq 
		C \dist(\delta, \partial \calM).
	\]
	If $j > k$, we can estimate
	\[
		\frac{e^{\sprod{x}{v_j}}}{\kappa(x)}
		\geq
		e^{\sprod{x}{v_j - v_1}} 
		 = 
		\Big(e^{\sprod{x}{v_k - v_1}}\Big)^{\sprod{x}{v_1 - v_j}/\sprod{x}{v_1 - v_k}} 
		\geq C \dist(\delta, \partial\mathcal{M})^{\sprod{x}{v_1-v_j}/\sprod{x}{v_1 - v_k}}
	\]
	which finishes the proof since, by \eqref{eq:8},
	\[
		1 \leq \frac{\sprod{x}{v_1 - v_j}}{\sprod{x}{v_1 - v_k}} \leq \epsilon^{-1} \norm{v_1 - v_j},
	\]
	thus it is enough to take
	\[
		\eta = \epsilon^{-1} \cdot \max\{\norm{v_1 - v} : v \in \calV\}.
		\qedhere
	\]	
\end{proof}

\subsection{Analytic lemmas about multiple derivation}
\label{subsec:2.2}
For a multi-index $\sigma \in \mathbb{N}^r$ we denote by $X_\sigma$ a multi-set containing
$\sigma(i)$ copies of $i$. Let $\Pi_\sigma$ be a set of all partitions of $X_\sigma$ and 
let $\{u_1,\ldots,u_r\}$ be a basis of $\mathfrak{a}$.  For the convenience of the reader
we recall

\begin{lemma}[Fa\`a di Bruno's formula]
	\label{lem:1}
	There are positive constants $c_\pi$, $\pi \in \Pi_\sigma$, such that for sufficiently smooth
	functions $f: S \rightarrow T$, $F: T \rightarrow \mathbb{R}$, $T \subset \mathbb{R}$,
	$S \subset \mathbb{R}^r$, we have
	$$
	\der{\sigma} F(f(s)) = \sum_{\pi \in \Pi_\sigma} c_\pi \left.
	\frac{{\rm d}^m }{{\rm d} t^m} \right|_{t = f(s)} F(t)
	\prod_{j=1}^{m} \der{B_j} f(s)
	$$
	where $\pi = \{B_1, \ldots, B_m\}$.
\end{lemma}

Let us observe that for
\[
	F(t) = \frac{1}{2-t}, \qquad\text{and}\qquad f(s) = \prod_{j=1}^r \frac{1}{1 - s_j},
\]
the function $F(f(s))$ is real-analytic in some neighborhood of $s = 0$, thus, there is $C > 0$
such that for every $\sigma \in \mathbb{N}^r$,
\begin{equation}
	\label{eq:22}
	\sum_{\pi \in \Pi_\sigma} c_\pi m! \prod_{j=1}^m B_j! = \der{\sigma} F(f(0)) \leq C^{\abs{\sigma}+1} \sigma!
\end{equation}
where for a multi-set $B$ containing $\mu(i)$ copies of $i$ we have set 
\[
	B! = \prod_{i = 1}^r \mu(i)!
\]
Using Lemma \ref{lem:1} we can show
\begin{lemma}
	\label{lem:2}
	Let $\calV \subset \RR^d$ be a set of finite cardinality. Assume that for each $v \in \calV$, 
	we are given $a_v \in \mathbb{C}$, and $b_v > 0$. Then for $z = x+i\theta \in \CC^d$ such that
	\[
		\norm{\theta} \leq (2 \cdot \max\{\norm{v}: v \in \mathcal{V}\})^{-1},
	\]
	we have
	\begin{equation}
		\label{eq:3}
		\Big| \sum_{v \in \mathcal{V}} b_v e^{\sprod{z}{v}} \Big|
		\geq \frac{1}{\sqrt{2}} \sum_{v \in \mathcal{V}} b_v e^{\sprod{x}{v}}.
	\end{equation}
	Moreover, there is $C > 0$ such that for all $\sigma \in \mathbb{N}^d$,
	\begin{equation}
		\label{eq:48}
		\bigg| \der{\sigma} \bigg\{ \frac{ \sum_{v \in \mathcal{V}} a_v e^{\sprod{z}{v}} }
		{ \sum_{v \in \mathcal{V}} b_v e^{\sprod{z}{v}} } \bigg\} \bigg| 
		\leq C^{\abs{\sigma}} \sigma! \frac{ \sum_{v \in \mathcal{V}} \abs{a_v} e^{\sprod{x}{v}} }
		{ \sum_{v \in \mathcal{V}} b_v e^{\sprod{x}{v}} }.
	\end{equation}
\end{lemma}
\begin{proof}
	We start by proving \eqref{eq:3}. We have
	\begin{align*}
		\Big| \sum_{v \in \calV} b_v e^{\sprod{z}{v}} \Big|^2
		& =
		\sum_{v, v' \in \calV} b_v b_{v'} e^{\sprod{x}{v + v'}} \cos\sprod{\theta}{v-v'} \\
		& \geq
		\sum_{v, v' \in \calV} b_v b_{v'} e^{\sprod{x}{v + v'}} \left(1 - \frac{\sprod{\theta}{v-v'}^2}{2}\right) \\
		& \geq
		\frac{1}{2}
		\Big( \sum_{v \in \calV} b_v e^{\sprod{x}{v}} \Big)^2 
	\end{align*}
	because $\abs{\sprod{\theta}{v-v'}} \leq 1$.

	For the proof of \eqref{eq:48}, it is enough to show
	\begin{equation}
		\label{eq:50}
		\bigg|\partial^{\sigma} \bigg\{ \frac{1}{\sum_{v \in \calV} b_v e^{\sprod{z}{v}}}\bigg\} \bigg|
		\leq
		C^{\abs{\sigma} + 1} \sigma! \frac{1}{\sum_{v \in \calV} b_v e^{\sprod{x}{v}}}.
	\end{equation}
	Indeed, since
	\begin{equation}
		\label{eq:51}
		\Big|
		\partial^\alpha \Big\{\sum_{v \in \calV} a_v e^{\sprod{z}{v}} \Big\}
		\Big|
		\leq
		\sum_{v \in \calV} 
		\abs{a_v} \cdot \abs{v^\alpha} e^{\sprod{x}{v}} 
		\leq
		C^{\abs{\alpha}} 
		\sum_{v \in \calV} \abs{a_v} e^{\sprod{x}{v}},
	\end{equation}
	by \eqref{eq:50} and the Leibniz's rule we obtain \eqref{eq:48}.
	To show \eqref{eq:50}, we use Fa\`a di Bruno's formula with $F(t) = 1/t$. By Lemma \ref{lem:1} together with
	estimates \eqref{eq:3} and \eqref{eq:51}, we get
	\begin{align*}
		\bigg|
		\partial^{\sigma} 
		\bigg\{ \frac{1}{\sum_{v \in \calV} b_v e^{\sprod{z}{v}}} \bigg\}
		\bigg|
		&\leq
		\sum_{\pi \in \Pi_\sigma} c_\pi m! \Big(\sum_{v \in \calV} b_v e^{\sprod{x}{v}} \Big)^{-m-1}
		\prod_{j=1}^m \Big| \partial^{B_j}\Big\{\sum_{v \in \calV} b_v e^{\sprod{z}{v}} \Big\}\Big| \\
		&\leq
		C^{\abs{\sigma}}
		\frac{1}{\sum_{v \in \calV} b_v e^{\sprod{x}{v}}}
		\sum_{\pi \in \Pi_\sigma} c_\pi m! \prod_{j = 1}^m B_j!\\
		&\leq
		C^{\abs{\sigma} + 1}
		\frac{1}{\sum_{v \in \calV} b_v e^{\sprod{x}{v}}}
	\end{align*}
	where in the last inequality we have used \eqref{eq:22}.
\end{proof}

\subsection{Variation on the marriage lemma}
\label{subsec:1}
The following combinatorial lemma may be known but we include its proof for completeness and 
lack of reference. Let $(C_1, C_2, \ldots, C_r)$ be a fixed sequence of subsets of a finite set 
$X$. A multi-index $\gamma \in \mathbb{N}^r$ is called admissible if there is $(X_j : 1 \leq j \leq r)$
a partial partition of $X$ such that $X_j \subseteq C_j$ and $\abs{X_j} = \gamma(j)$. We set
$e_j$ to be a multi-index with $1$ on the $j$th position and $0$ elsewhere.
\begin{lemma}
	\label{lem:5}
	If $\gamma$ is admissible then for any partial partition $(X_j : 1 \leq j \leq r)$ corresponding
	to $\gamma$ we have
	\[
		\bigcup_{j \in J_\gamma} X_j = \bigcup_{j \in J_\gamma} C_j
	\]
	where $J_\gamma = \{j: \gamma+e_j \text{ is not admissible}\}$.
\end{lemma}
\begin{proof}
	Given $m \in J_\gamma$ we construct a sequence $(I_j : 0 \leq j)$ as follows:
	$I_0=\{m\}$ and for $i \geq 0$,
	$$
	I_{i+1} = \big\{j : X_j \cap C_k \neq \emptyset \text{ for some } k \in I_i \big\}.
	$$
	We notice that $I_i \subseteq I_{i+1}$. Let $I = \limsup_{i \geq 0} I_i$ and 
	$
	V = \bigcup_{j \in I} X_j.
	$
	We claim that
	\begin{equation}
		\label{eq:10}
		V = \bigcup_{j \in I} C_j.
	\end{equation}
	Suppose that, contrary to the claim, there is
	\[
		y \in \bigcup_{j \in I} C_j \cap V^c.
	\]
	We first observe that $y \not\in \bigcup_{j=1}^r X_j$. Indeed, $y \in C_j \cap X_{j'}$ for some $j \in I$
	implies that $j' \in I$. Also there are sequences
	$(j_i : 1 \leq i \leq n)$ and $(x_i : 0 \leq i \leq n)$ of distinct elements such that
	$j_1 = m$, $y \in C_{j_n}$, $x_0 \in X_{j_1}$, $x_n = y$ and $x_i \in C_{j_i} \cap X_{j_{i+1}}$
	for $i \in \{1, \ldots, n-1\}$. By setting
	$$
	Y_j = 
	\begin{cases}
		\big(X_{j_i} \cup \{x_i\} \big)\setminus \{x_{i-1}\} & \text{ if } j = j_i \text{ for }
		\in i \in \{1, \ldots, n\}, \\
		X_j & \text{ otherwise,}
	\end{cases}
	$$
	we obtain a partial partition of $X$ corresponding to $\gamma$ such that
	$$
	x_0 \in C_m \cap \Big(\bigcup_{j=1}^r Y_j \Big)^c
	$$
	which is not possible since $m \in J_\gamma$, proving the claim.

	As a consequence of the claim, we have
	\[
		\abs{V} = \sum_{j \in I} \gamma(j).
	\]
	We next show that $I = J_\gamma$. Suppose that, on the contrary, there is $k \in I \cap J_\gamma^c$. Then there
	exists $(Y_j : 0 \leq j \leq r)$ a partial partition corresponding to $\gamma$ such that
	\begin{equation}
		\label{eq:31}
		C_k \cap \Big(\bigcup_{j=1}^r Y_j \Big)^c \neq \emptyset.
	\end{equation}
	Since $Y_j \subseteq C_j$ and
	$$
	\sum_{j \in I} \abs{Y_j} = \sum_{j \in I} \gamma(j),
	$$
	we must have
	$$
	\bigcup_{j \in I} Y_j = \bigcup_{j \in I} C_j
	$$
	which contradicts \eqref{eq:31}. Therefore, $I \subseteq J_\gamma$ and the lemma follows.
\end{proof}

\section{Affine buildings}
\label{sec:4}
This section presents the singular, usually higher-dimensional, spaces in which we wish to study the behavior of
the heat kernel. These spaces are called affine buildings and are discrete analogues of Riemannian symmetric spaces.
They are union of Euclidean tilings in tight connection with the theory of root systems. They have strong symmetry
properties, so that they often have a very transitive automorphism group. Still, we prefer to use them in a purely
geometric way. The last subsection illustrates this choice by presenting the spherical harmonic analysis we need.
Indeed, harmonic analysis on buildings started in a group-theoretic context by exhibiting Gelfand pairs 
(see \cite{Satake1963}) and then by computing explicitly the corresponding spherical functions 
(see \cite{macdo0,matsumoto69}), but these fundamental works have now geometric generalizations avoiding group actions,
thus allowing to consider a few more cases in dimension $2$. 

\subsection{Root systems, weights and coweigths}
We start by recalling basic facts about root systems and Coxeter groups. A general reference
is \cite{Bourbaki2002}.

Let $\Phi$ be an irreducible but not necessarily reduced finite root system in $\mathfrak{a}$. 
Let $\{\alpha_i: i \in I_0\}$ where $I_0=\{1, \ldots, r\}$ be a fixed base of $\Phi$, and $\Phi^+$ the corresponding set
of all positive roots. Let $\mathfrak{a}_+$ be the positive Weyl chamber, i.e.
$$
\mathfrak{a}_+=\big\{x \in \mathfrak{a}: \sprod{\alpha}{x} > 0 \text{ for all } 
\alpha \in \Phi^+\big\}.
$$
By $\alpha_0$, we denote the highest root of $\Phi$, that is a root
$$
\alpha_0 = \sum_{i \in I_0} m_i \alpha_i,
$$
such that for any $\alpha \in \Phi$, $\alpha = \sum_{i \in I_0} n_i \alpha_i$ we have $n_i \leq m_i$.
We set $m_0 = 1$ and $I = I_0 \cup \{0\}$. Let 
$$
I_P=\{i \in I: m_i = 1\}.
$$
The dual basis to $\{\alpha_i: i \in I_0\}$ is denoted by $\{\lambda_i: i \in I_0\}$. The co-weight
lattice $P$ is the $\zspan$ of fundamental co-weights $\{\lambda_i: i \in I_0\}$. A co-weight
$\lambda \in P$ is called dominant if $\lambda = \sum_{i \in I_0} x_i \lambda_i$, where
$x_i \geq 0$ for all $i \in I_0$. Finally, the cone of all dominant co-weights is denoted by $P^+$.

Let $H_i=\{x \in \mathfrak{a}: \sprod{\alpha_i}{x}=0\}$ for each $i \in I_0$. We denote by $r_i$
the orthogonal reflection in $H_i$, i.e. $r_i(x) = x - \sprod{\alpha_i}{x} \alpha_i\spcheck$ for
$x \in \mathfrak{a}$ where for $\alpha \in \Phi$ we put
$$
\alpha\spcheck = \frac{2\alpha}{\sprod{\alpha}{\alpha}}.
$$
By $Q$ we denote the co-root lattice, that is $\zspan$ of the co-roots $\{\alpha\spcheck : \alpha \in \Phi\}$.
The subgroup $W_0$ of $\GL(\mathfrak{a})$ generated by $\{r_i: i \in I_0\}$ is the Weyl group of
$\Phi$. Let $r_0$ be the orthogonal reflection in the affine hyperplane $H_0=\{x \in \mathfrak{a}:
\sprod{\alpha_0}{x} = 1\}$. Then the affine Weyl group $W$ of $\Phi$ is the subgroup
of $\Aff(\mathfrak{a})$ generated by $\{r_i: i \in I\}$. Finally, the extended affine Weyl group
of $\Phi$ is $\widetilde{W}=W_0 \ltimes P$. We set
\[
	\rho = \sum_{j = 1}^r \lambda_j = \frac{1}{2} \sum_{\alpha \in \Phi^+} \alpha\spcheck
\]
Let $M=(m_{ij})_{i,j \in I}$ be a symmetric matrix with entries in $\mathbb{Z} \cup \{\infty\}$
such that for all $i,j \in I$,
$$
m_{ij} = \left\{
            \begin{array}{ll}
			 \geq 2 & \text{ if } i \neq j,\\
			 1      & \text{ if } i = j.
			\end{array}
         \right.
$$
The Coxeter group of type $M$ is the group $W$ given by the presentation
$$
\big\langle r_i: (r_i r_j)^{m_{ij}}=1 \text{ for all } i,j \in I \big\rangle.
$$
For a word $f=i_1 \cdots i_k$ in the free monoid $I$ we denote by $r_f$ an element of $W$ of the
form $r_f= r_{i_1} \cdots r_{i_k}$. The length of $w \in W$, denoted $\ell(w)$, is the smallest
integer $k$ such that there is a word $f=i_1 \cdots i_k$ and $w=r_f$.  We say $f$ is reduced if
$\ell(r_f) = k$.

\subsection{Building, thicknesses and (co)type}
\label{sec:4.2}
For the theory of affine buildings we refer the reader to \cite{ron}.

A set $\mathscr{X}$ equipped with a family of equivalence relations $\{\sim_i:i \in I\}$ is a
chamber system and the elements of $\mathscr{X}$ are called chambers. A gallery of type
$f = i_1 \cdots i_k$ in $\mathscr{X}$ is a sequence of chambers $(c_0, \ldots, c_k)$ such that for
all $1 \leq j \leq k$, $c_{j-1} \sim_{i_j} c_j$ and $c_{j-1} \neq c_j$. If $J \subseteq I$,
$J$-residue is a subset of $\mathscr{X}$ such that any two chambers can be joined by a gallery of
type $f = i_1\cdots i_k$ with $i_1, \ldots, i_k \in J$.

Let $W$ be a Coxeter group of type $M$. For each $i \in I$, we define an equivalence relation on
$W$ by declaring that $w \sim_i w'$ if and only if $w = w'$ or $w = w'r_i$. Then $W$ equipped with
$ \{\sim_i : i \in I\}$ is a chamber system called Coxeter complex of $W$.
\begin{definition}
	Let $W$ be a Coxeter group. A chamber system $\mathscr{X}$ is a building of type $W$ if
	\begin{enumerate}
			 \item for all $x \in \mathscr{X}$ and $i \in I$, 
				 $\abs{\{y \in \mathscr{X}: y \sim_i x\}} \geq 2$,
			 \item there is $W$-distance function
				 $\delta: \mathscr{X} \times \mathscr{X} \rightarrow W$ such that if $f$ is a
				 reduced word, then $\delta(x, y) = r_f$ if and only if $x$ and $y$ can be joined
				 by a gallery of type $f$.
	\end{enumerate}
	If $W$ is an affine Weyl group, the building $\mathscr{X}$ is called affine.
\end{definition}
Notice that if we define $\delta_W: W \times W \rightarrow W$ by $\delta_{W}(w, w')=w^{-1}w'$
then $\delta_W$ is $W$-distance function. Thus a Coxeter complex of $W$ is a building of type $W$.

A subset $\mathcal{A} \subset \mathscr{X}$ is called an apartment if there is a mapping
$\psi: W \rightarrow \mathscr{X}$ such that $\mathcal{A} = \psi(W)$ and for all
$w, w' \in W$, $\delta(\psi(w), \psi(w')) = \delta_W(w, w')$.

A building $\mathscr{X}$ has a geometric realization as a simplicial complex $\Sigma(\mathscr{X})$
where a residue of type $J$ corresponds to a simplex of dimension $\abs{I}-\abs{J}-1$. Let
$V(\mathscr{X})$ denote the set of vertices of $\Sigma(\mathscr{X})$. Define a mapping
$\tau: V(\mathscr{X}) \rightarrow I$ by declaring $\tau(x)=i$ if $x$ corresponds to a residue of
type $I \setminus \{i\}$. 

For $x \in \mathscr{X}$ and $i \in I$, let $q_i(x)$ be equal to
\[
	q_i(x) = \abs{\{y \in \mathscr{X}: y \sim_i x\}} - 1.
\]
We assume that the building is regular that is $q_i(x)$ is independent of $x$. Denote the common value by $q_i$, 
and assume local finiteness: $q_i < \infty$.  

To any irreducible locally finite affine building we associate an irreducible, but not
necessary reduced, finite root system $\Phi$ (see \cite{park2}) such that the affine Weyl group
corresponding to $\Phi$ is isomorphic to $W$, and $q_{\tau(v)} = q_{\tau(v+\lambda)}$ for all
$\lambda \in P$ and $v \in \Sigma(W)$. Then the set of \emph{good} vertices is defined by
$$
V_P = \{v \in V(\mathscr{X}): \tau(v) \in I_P\}.
$$

\subsection{Spherical harmonic analysis}
\label{subsec:4.3}
In this subsection we summarize spherical harmonic analysis on affine buildings
(see \cite{macdo0, park2}).

Let $\mathscr{X}$ be an irreducible locally finite regular affine building. Given
$x \in V_P$ and $\lambda \in P^+$, let $V_\lambda(x)$ denote the set of all $y \in V_P$ such that
there are: an apartment $\mathcal{A}$ containing $x$ and $y$, a type-preserving isomorphism 
$\psi: \mathcal{A} \rightarrow \Sigma(W)$ and $w \in \widetilde{W}$ such that $\psi(x) = 0$ and
$\psi(y) = w \lambda$. It may be shown that $\abs{V_\lambda(x)}$ is independent of $x$. Let $N_\lambda$ denote
its common value.

For each $\lambda \in P^+$, we define an operator $A_\lambda$ acting on $f \in \ell^2(V_P)$ by
$$
A_\lambda f(x) = \frac{1}{N_\lambda} \sum_{y \in V_\lambda(x)} f(y).
$$
Then $\mathscr{A}_0=\cspan\{A_\lambda : \lambda \in P^+\}$ is a commutative $\star$-subalgebra of
the algebra of bounded linear operators on $\ell^2(V_P)$, see \cite[Theorem 5.24]{park3}. The multiplicative functionals on 
$\mathscr{A}_0$ can be described in terms of Macdonald spherical functions $P_\lambda$, $\lambda \in P^+$, see
\cite[Section 6.3]{park3}. Namely, each multiplicative functional $h_z$, $z \in \mathfrak{a}_\mathbb{C}$, is a linear map on
$\mathscr{A}_0$ such that
\[
	h_z(A_\lambda) = P_\lambda(z)
\]
for all $\lambda \in P^+$. Before we recall the definition of Macdonald spherical functions, 
let us introduce some notation. Let $\Phi^{++}$ be the set of roots $\alpha \in \Phi^+$ so that 
$\frac{1}{2}\alpha \notin \Phi^+$. If $\alpha \in \Phi^{++}$ then $q_\alpha = q_i$ provided that $\alpha \in W_0\cdot\alpha_i$ 
for some $i \in I$. We define
\[
	\tau_\alpha = 
	\begin{cases}
		1 & \text{if } \alpha \notin \Phi, \\
		q_\alpha & \text{if } \alpha \in \Phi, \text{ but } \frac{1}{2}\alpha, 2 \alpha \notin \Phi, \\
		q_{\alpha_0} & \text{if } \alpha, \frac{1}{2} \alpha \in \Phi\\
		q_\alpha q_{\alpha_0}^{-1} &\text{if } \alpha, 2\alpha \in \Phi.
	\end{cases}
\]
Let $\chi_0$ denote the fundamental character that is a multiplicative function on $P$,
\[
	\chi_0(\lambda) = \prod_{\alpha \in \Phi^+} \tau_\alpha^{\sprod{\lambda}{\alpha}}.
\]
If $w \in W_0$ has a reduced expression $w = r_{i_1} r_{i_2} \cdots r_{i_k}$, then $q_w=q_{i_1} \cdots q_{i_k}$. 
If $\lambda \in P^+$, the Macdonald spherical function $P_\lambda$ is (see \cite{macdo0})
\[
	P_\lambda(z) = \frac{\chi_0(\lambda)^{-\frac{1}{2}} }{W_0(q^{-1})} 
	\sum_{w \in W_0} \bfc(w \cdot z) e^{\sprod{w \cdot z}{\lambda}}
\]
where
\begin{align*}
	\bfc(z) &= \prod_{\alpha \in \Phi^+} \frac{1 - \tau_\alpha^{-1} \tau_{\alpha/2}^{-1/2} e^{-\sprod{z}{\alpha\spcheck}}}
	{1-\tau_{\alpha/2}^{-1/2} e^{-\sprod{z}{\alpha\spcheck}}} \\
	&=
	\prod_{\alpha \in \Phi^{++}}
	\frac{\Big(1 - \tau_{2\alpha}^{-1} \tau_{\alpha}^{-\frac{1}{2}} e^{-\frac{1}{2}\sprod{z}{\alpha\spcheck}}\Big)
	\Big(1 + \tau_{\alpha}^{-\frac{1}{2}} e^{-\frac{1}{2} \sprod{z}{\alpha\spcheck}}\Big)}
	{1 - e^{-{\sprod{z}{\alpha\spcheck}}}},
\end{align*}
and 
$$
W_0(q^{-1})=\sum_{w \in W_0} q_w^{-1}.
$$
Values of $P_\lambda$ where the denominator of the $\bfc$-function equals zero is obtained by taking proper limits.

By $\mathscr{A}_2$ we denote the closure of $\mathscr{A}_0$ in the operator norm. Then $\mathscr{A}_2$ is
$C^\star$-algebra. To describe the Gelfand transform as well as the Planchrel's measure we need to distinguish two
cases:

\vspace*{1ex}
\noindent \emph{The standard case.} Assume that $\tau_\alpha \geq 1$ for all $\alpha \in \Phi$. Then for each
$\theta \in U_0$, where
\[
	U_0=\left\{\theta \in \mathfrak{a} : \sprod{\theta}{\alpha\spcheck} \leq \pi \text{ for all } \alpha \in \Phi\right\},
\]
the multiplicative functional $h_{i\theta}$ extends to $\mathscr{A}_2$ in a continuous way. Moreover, for each
$A \in \mathscr{A}_0$, $x \in V_P$, and $y \in V_\lambda(x)$, we have
\begin{equation}
	\label{eq:9}
	(A \delta_x)(y) = \bigg(\frac{1}{2\pi}\bigg)^r \frac{W_0(q^{-1})}{|W_0|}
	\int_{U_0} h_{i\theta}(A) \overline{P_\lambda(i\theta)} \frac{{\rm d} \theta}{|\bfc(i\theta)|^2}
\end{equation}
where $\delta_x(y)$ is Dirac's delta at $x$, see \cite[Theorem 5.2 \& Corollary 5.5]{park2}.

\vspace*{1ex}
\noindent \emph{The exceptional case.} Suppose that $\tau_\alpha < 1$ for some $\alpha \in \Phi$. It is only possible 
when $\Phi$ is $\text{BC}_r$ root system and $q_r < q_0$, namely
\[
	\Phi = \left\{ \pm e_i, \pm 2e_i, \pm e_j \pm e_k : 1 \leq i \leq r, 1 \leq j < k \leq r
	\right\}
\]
where $\{e_1, e_2, \ldots, e_r\}$ is the standard basis of $\mathfrak{a}$. We set $a = \sqrt{q_r q_0}$ and
$b = \sqrt{q_r/q_0}$. Then
\[
	\bfc(z) = \bigg( \prod_{j = 1}^r \frac{(1-a^{-1} e^{-z_j}) (1 + b^{-1} e^{-z_j})}{1-e^{-2z_j}} \bigg)
	\bigg(
	\prod_{1 \leq j < k \leq r} \frac{(1-q_1^{-1} e^{-z_j - z_k})(1 - q_1^{-1}e^{-z_j+z_k})}
	{(1 - e^{-z_j-z_k})(1-e^{-z_j+z_k})}\bigg).
\]
Let $v = \log b - i \pi$. For $j =1, \ldots, r$, we set
\[
	U_j = \left\{
	\theta \in \left[-\tfrac{1}{2} \pi, \tfrac{3}{2} \pi \right]^r  : \theta_j = -v \right\}.
\]
and $U_0 = [-\pi/2, 3 \pi/2]^r$. For $\theta \in U_1$, we define
\[
	\phi_1(i\theta) = \lim_{t \to 0} \frac{|\bfc(i\theta + t e_j)|^2}
	{1-e^t}.
\]
Then for each $\theta \in U_0 \sqcup U_1$, the multiplicative functional $h_{i\theta}$ extends to $\mathscr{A}_2$
in a continuous way. Moreover, for $A \in \mathscr{A}_0$, $x \in V_P$ and $y \in V_\lambda(x)$, we have
\begin{equation}
	\label{eq:81}
	\begin{aligned}
	(A \delta_x)(y) 
	&= \bigg(\frac{1}{2\pi}\bigg)^r \frac{W_0(q^{-1})}{|W_0|}
	\int_{U_0} h_{i\theta}(A) \overline{P_\lambda(i\theta)} \frac{{\rm d} \theta}{|\bfc(i\theta)|^2} \\
	&\phantom{=}+
	\bigg(\frac{1}{2\pi} \bigg)^{r-1} \frac{W_0(q^{-1})}{|W_0'|}
	\int_{U_1} h_{i\theta}(A) \overline{P_\lambda(i\theta)} \frac{{\rm d} \theta}{\phi_1(i\theta)}
	\end{aligned}
\end{equation}
where $W_0'$ is the Coxeter group $C_{r-1}$ and the measure ${\rm d}\theta$ on $U_j$ equals
\[
	{\rm d}\theta = \prod_{\stackrel{k = 1}{k \neq j}}^r {\rm d}\theta_k
\]
for $j = 0, 1, \ldots, r$, see \cite[Theorem 5.7 \& Corollary 5.8]{park2}

\section{Asymptotics}
In this section, we prove the main result of the paper, on the asymptotic behavior of the heat kernel on affine buildings 
(see Theorem \ref{thm:4}). This requires to recall some facts on random walks in Section \ref{sec:4.1}, where we also
explain the relationship with Section \ref{sec:2.1} on convex combinations of exponentials. The longest Section
\ref{sec:4.4} deals with the proof of Theorem \ref{thm:4}: it is analytic in nature but requires some combinatorial
arguments with Lie-theoretic ingredients. It starts with an application of the spherical Fourier transform and the
contour deformation which results in an oscillatory integral studied by the steepest descent method. The rest of the
preliminary Section \ref{sec:3}, i.e. analytic lemmas on multiple derivations and a combinatorial one elaborating on the
marriage lemma, is used here. Considerations of root systems are used to determine the correct leading terms in the
desired asymptotics. At last, Section \ref{subsec:4.6} is dedicated to asymptotics for the Green functions.

\subsection{Random walks}
\label{sec:4.1}
In this paper we are interested in asymptotic behavior of \emph{isotropic} random walks on good vertices $V_P$, i.e.
random walks with the transition probabilities $p(x,y)$ constant on
\[
	\big\{(x,y) \in V_P \times V_P: y \in V_\lambda(x)\big\}
\]
for every $\lambda \in P^+$. Let $A$ denote the corresponding operator acting on $\ell^2(V_P)$, namely 
for $f \in \ell^2(V_P)$,
$$
A f (x) = \sum_{y \in V_P} p(x, y) f(y).
$$
Then $A$ belongs to the algebra $\mathscr{A}_2$ and may be expressed as
$$
A = \sum_{\mu\in P^+} a_\mu A_\mu
$$
where $a_\mu \geq 0$ and $\sum_{\mu \in P^+} a_\mu = 1$. We say that the random walk has a \emph{finite range,} if 
$a_\mu > 0$ for finitely many $\mu \in P^+$. We set $p(1; x, y) = p(x, y)$, and for $n \geq 2$,
\[
	p(n; x, y) = \sum_{z \in V_p} p(n-1; x, z) p(z, y).
\]
If $O$ is a fixed good vertex, we write $p(n; x) = p(n; O, x)$.

The random walk is \emph{irreducible,} if for any $x, y \in V_P$ there is $n \in \NN$ such that
$p(n; x, y) > 0$. Lastly, the walk is called \emph{aperiodic} if for every $x \in V_P$,
\[
	\gcd\big\{n \in \NN : p(n; x, x) > 0\big\} = 1.
\]
We shall be concern with irreducible and aperiodic random walks having a finite range. Then there are a finite set
$\mathcal{V} \subset P$, and positive real numbers $\{c_v:v \in \mathcal{V}\}$, such that
\[
	\kappa(z) = \sum_{v \in \mathcal{V}} c_v e^{\sprod{z}{v}}
\]
where we have set
\[
	\kappa(z) = \varrho^{-1} h_z(A)
\]
and $\varrho = h_0(A)$. We can use the results of Section \ref{sec:2.1}. Recall that $\calM \subset \mathfrak{a}$ is
the interior of the convex hull of $\calV$. The set $\calM$ is not empty as it contains the convex hull of
\[
	\bigg\{\frac{\lambda_1}{m}, -\frac{\lambda_1}{m}, \ldots, \frac{\lambda_r}{m}, 
	-\frac{\lambda_r}{m} \bigg\}
\]
where $m$ is such that $V_{\lambda_j}(O) \subseteq p(m; \cdot)$ for all $j \in I_0$. Because $\kappa$ is $W_0$-invariant, 
we have $\nabla \kappa(0) = 0$. If $\delta \in \mathcal{M}$ and $w \in W_0$, we can write
\[
	w \cdot \delta = w \cdot \nabla \log \kappa(s) = \nabla \log \kappa(w \cdot s)
\]
where $s = s(\delta)$. Hence, Theorem \ref{thm:1} implies that $w \cdot s(\delta) = s(w \cdot \delta)$. For
$\alpha \in \Phi$, we set
\[
	r_\alpha (x) = x - \sprod{\alpha\spcheck}{x} \alpha.
\]
Since
\[
	0 \leq \sprod{s}{\delta}- \log \kappa(s) - \sprod{r_\alpha s}{\delta} + \log \kappa(r_\alpha s) 
	=\sprod{s}{\alpha\spcheck} \sprod{\alpha}{\delta},
\]
by the implicit function theorem, the mapping $s: \mathcal{M} \rightarrow \mathfrak{a}$ is real-analytic and
$s(\mathcal{M} \cap \cl \mathfrak{a}_+) = \cl \mathfrak{a}_+$. In what follows, $\eta \geq 1$ is the number
determined in Theorem \ref{thm:2}.

\subsection{Heat kernels}
\label{sec:4.4}
Before stating the asymptotic formula for $p(n; v)$, we need to introduce some notation. Given
$\emptyset \neq J \subsetneq I_0$, by $\Psi$ denote the set consisting of $\alpha \in \Phi$ so that
$\sprod{\alpha}{\lambda_j} = 0$ for all $j \in I_0 \setminus J$. Then $\Psi$ is a root system in
$\mathfrak{a}_\Psi = \rspan \Psi$. By $T_\Psi : \mathfrak{a} \rightarrow \mathfrak{a}$ we denote the orthogonal projection
along $\mathfrak{a}_\Psi$. Let $\Psi^+ = \Psi \cap \Phi^+$. For $\omega \in P^+$ and $x \in \mathfrak{a}$, we set
\[
	\calP_\Psi(\omega) = 
	\frac{\chi_0(\omega)^{-\frac{1}{2}}}{\abs{\bfb_\Psi(0)}^2}
	\cdot
	\lim_{\theta \to 0}
	\frac{1}{\abs{W_0(\Psi)}}
	\sum_{w \in W_0(\Psi)} e^{-\sprod{w \cdot \theta}{\omega}}
    \bfc_\Psi(-w \cdot \theta), 
\]
and
\[
	\calQ_\Psi(x) =
	\bigg(\frac{1}{2\pi}\bigg)^r
	\int_{\mathfrak{a}}
	e^{-\frac{1}{2} B_x(u, u)} \abs{\bpi_\Psi(u)}^2 {\: \rm d}u
	\cdot
	\bigg(
	\prod_{\alpha \in \Phi^+ \setminus \Psi^+}
	\frac{1 - \tau_{\alpha/2}^{-1/2} e^{-\sprod{x}{\alpha\spcheck}}}
	{1 - \tau_\alpha^{-1} \tau_{\alpha/2}^{-1/2} e^{-\sprod{x}{\alpha\spcheck}}}
	\bigg)
\]
where
\begin{equation}
	\label{eq:21}
	\bpi_\Psi(x) = \prod_{\alpha \in \Psi^{++}} \sprod{x}{\alpha\spcheck},
\end{equation}
and
\begin{align*}
    \bfc_\Psi(x) &= \prod_{\alpha \in \Psi^+}
    \frac{1 - \tau_\alpha^{-1} \tau_{\alpha/2}^{-1/2} e^{-\sprod{x}{\alpha\spcheck}}}
	{1 - \tau_{\alpha/2}^{-1/2} e^{-\sprod{x}{\alpha\spcheck}}}, \\
    \bfb_\Psi(x) &= \prod_{\alpha \in \Psi^{++}} \Big(1 - \tau_{2\alpha}^{-1} \tau_\alpha^{-1/2}
	e^{-\frac{1}{2} \sprod{x}{\alpha\spcheck}}\Big)
	\Big(1 + \tau_\alpha^{-1/2} e^{-\frac{1}{2} \sprod{x}{\alpha\spcheck}} \Big). 
\end{align*}
If $J = \emptyset$, then $\Psi = \emptyset$, and
\[
	\calP_\Psi(\omega) = \chi_0(\omega)^{-\frac{1}{2}},
	\qquad\text{and}\qquad
	\calQ_\Psi(x) = 
	\bigg(\frac{1}{2\pi}\bigg)^r
	\int_{\mathfrak{a}} e^{-\frac{1}{2} B_x(u, u)} {\: \rm d}u \cdot \frac{1}{\bfc(x)}.
\]
\begin{theorem}
	\label{thm:4}
	Let $J \subsetneq I_0$. Suppose that $(\omega_n : n \in \NN)$ is a sequence of co-weights
	such that $V_{\omega_n}(O)$ is contained in the support of $p(n; \,\cdot\,)$. We assume that
	$\delta_n = n^{-1} \omega_n$ satisfies
	\begin{subequations}
		\begin{equation}
		\label{eq:70}
		\lim_{n \to \infty} n^{-1} \dist(\delta_n, \partial \calM)^{-2\eta} =0,\\
		\end{equation}
		\begin{equation}
		\label{eq:70a}
		\lim_{n \to \infty} \sprod{\delta_n}{\alpha} \dist(\delta_n, \partial \calM)^{-2\eta} =0,
		\quad\text{for all}\quad \alpha \in \Psi^+,
		\end{equation}
		\begin{equation}
		\label{eq:70b}
		\sprod{\delta_n}{\alpha} \geq \xi, \quad\text{for all}\quad \alpha \in \Phi^+\setminus\Psi^+,
		\end{equation}
	\end{subequations}
	for some $\xi > 0$. Then for any sequence of good vertices $(v_n : n \in \NN)$ such that $v_n \in V_{\omega_n}(O)$,
	\[
		p(n; v_n) = 
		n^{-\frac{r}{2}-\abs{\Psi^{++}}}
		\varrho^n e^{-n\phi(\delta_n)}
		\calP_\Psi(\omega_n) \calQ_\Psi(t_n) \big(1 + E_n(\delta_n)\big)
	\]
	with
	\[
		\abs{E_n(\delta_n)} \leq C \sum_{\alpha \in \Psi^+} \big(\sprod{\delta_n}{\alpha} + n^{-1}\big)
		\dist(\delta_n, \partial \calM)^{-2\eta}
	\]
	where $t_n = (I - T_\Psi) s_n$, $s_n = \nabla \phi(\delta_n)$, and
	\[
		\phi(\delta) = \max\big\{\sprod{u}{\delta} - \log \kappa(u) : u \in \mathfrak{a}\big\}.
	\]
\end{theorem}
\begin{proof}
	We consider the standard case. The necessary changes in the exceptional case are explained in Appendix \ref{app:a}.
	Let us recall that $\tau_\alpha \geq 1$ for all $\alpha \in \Phi$. By the inversion formula \eqref{eq:9}, we can write
	\[
		p(n; v_n) = \bigg(\frac{1}{2 \pi} \bigg)^r \frac{W_0(q^{-1})}{\abs{W_0}} \int_{U_0}
		\big(h_{i\theta}(A)\big)^n
		\overline{P_{\omega_n}(i\theta)} \frac{{\: \rm d} \theta}{\abs{\bfc(i\theta)}^2}.
	\]
	Using the definition of $P_\omega$ and $W_0$-invariance of the integrand, we get
	\begin{align*}
		p(n; v_n) 
		=
		\bigg(\frac{1}{2 \pi} \bigg)^r
		\chi_0(\omega_n)^{-\frac{1}{2}}
		\varrho^n
		\calF_n(\omega_n)
	\end{align*}
	where
	\[
		\calF_n(\omega) = \int_{U_0}
		\kappa(i\theta)^n e^{-i \sprod{\theta}{\omega}} \frac{{\rm d}\theta}{\bfc(i\theta)}.
	\]
	Suppose now that $\theta_0 \in U_0$ is such that $\kappa(i \theta_0) = e^{i t}$ for some $t \in [-\pi, \pi)$.
	Since $\kappa(i\theta_0)$ is a convex combination of complex numbers from the unit circle,
	$\kappa(i\theta_0) = e^{i t}$ if and only if $e^{i \sprod{\theta_0}{v}} = e^{i t}$ for all $v \in \calV$. Therefore,
	if $p(n; x) > 0$ for some $x \in V_\omega(O)$, then $e^{i \sprod{\theta_0}{\omega}} = e^{i n t}$. Since the random
	walk is irreducible and aperiodic, for all sufficiently large $n$ we have $p(n; x) > 0$, thus
	\[
		e^{i n t} =	e^{i \sprod{\theta_0}{\omega}} = e^{i (n+1) t}
	\]
	which implies that $t = 0$. Therefore, $e^{i \sprod{\theta_0}{\omega}} = 1$ for all $\omega \in P^+$, which entails 
	that $\theta_0 = 0$.

	Next, we observe that we can shift the integrand. In fact, we have the following claim.
	\begin{claim}
		\label{clm:1}
		For any $u \in \mathfrak{b}$ where
		\[
			\mathfrak{b} 
			= \left\{x \in \mathfrak{a} : \sprod{x}{\alpha\spcheck} > 
			-\log \tau_{\alpha} - \tfrac{1}{2} \log \tau_{\frac{\alpha}{2}}
			\text{ for all } \alpha \in \Phi^+ \right\},
		\] 
		we have
		\[
			\calF_n(\omega) 
			=
			\int_{U_0}
	    	\kappa(u+i\theta)^n e^{-\sprod{u + i\theta}{\omega}} \frac{{\rm d}\theta}{\bfc(u + i\theta)}.
		\]
	\end{claim}
	Let us first observe that the integrand is $2\pi Q$-periodic. Hence, the value of the integral stays unchanged if we
	replace $U_0$ by any other fundamental domain for the action of $2\pi Q$ on $\mathfrak{a}$. It will be more convenient
	to replace $U_0$ with
	\[
		V = \big\{\theta = \theta_1 \alpha_1 + \ldots + \theta_r \alpha_r : \theta_j \in [-\pi, \pi] \big\}.
	\]
	Now, it is easy to see that for any $\lambda \in P$ we have
	\begin{equation}
		\label{eq:49}
		\int_V e^{i \sprod{\theta}{\lambda}} e^{-i\sprod{\theta}{\omega}} {\: \rm d}\theta
		=
		\int_V e^{\sprod{u + i \theta}{\lambda}} e^{-\sprod{u+i\theta}{\omega}} {\: \rm d} \theta.
	\end{equation}
	Since $e^{-\sprod{u}{\alpha\spcheck}} \tau_\alpha^{-1} \tau_{\alpha/2}^{-1/2} < 1$, we can write
	\begin{align*}
		\frac{1-\tau_{\alpha/2}^{-1/2} e^{-\sprod{u+i\theta}{\alpha\spcheck}}}
		{1 - \tau_{\alpha}^{-1} \tau_{\alpha/2}^{-1/2} e^{-\sprod{u+i\theta}{\alpha\spcheck}}}
		&=
		\sum_{n \geq 0} \tau_\alpha^{-n} \tau_{\alpha/2}^{-n/2} e^{-n \sprod{u+i\theta}{\alpha\spcheck}}
		-
		\tau_\alpha
		\sum_{n \geq 1} \tau_\alpha^{-n} \tau_{\alpha/2}^{-n/2} e^{-n \sprod{u+i\theta}{\alpha\spcheck}}
		\\
		&=
		1 + (1-\tau_\alpha) \sum_{n_\alpha \geq 1} \tau_\alpha^{-n_\alpha} \tau_{\alpha/2}^{-n_\alpha/2} 
		e^{-n_\alpha \sprod{u+i\theta}{\alpha\spcheck}}
	\end{align*}
	where the series is uniformly and absolutely convergent. Hence, 
	\begin{align*}
		\frac{1}{\bfc(u + i \theta)} &=
		\prod_{\alpha \in \Phi_+} 
		\bigg(
		1 + (1 - \tau_\alpha) \sum_{n_\alpha \geq 1} \tau_\alpha^{-n_\alpha} \tau_{\alpha/2}^{-n_\alpha/2}
		    e^{-n_\alpha \sprod{u+i\theta}{\alpha\spcheck}}
		\bigg)\\
		&= 
		\sum_{\alpha\spcheck \in Q\spcheck} c(\alpha\spcheck, q) e^{\sprod{u+i\theta}{\alpha\spcheck}}.
	\end{align*}
	Thus, by the identity \eqref{eq:49}, we obtain
	\begin{align*}
		&\int_V \kappa(u+i\theta)^n e^{-\sprod{u+i\theta}{\omega}} \frac{{\rm d}\theta}{\bfc(u+i\theta)} \\
		&\qquad\qquad=
		\sum_{\alpha\spcheck \in Q\spcheck} c(\alpha\spcheck, q) 
		\sum_{v_1, \ldots, v_n \in \calV}
		\prod_{j = 1}^n c_{v_j}
		\int_V
		e^{\sprod{u+i\theta}{\sum_{j = 1}^n v_j}} e^{-\sprod{u+i\theta}{\omega}} e^{\sprod{u+i\theta}{\alpha\spcheck}}
		{\: \rm d}\theta\\
		&\qquad\qquad=
		\sum_{\alpha\spcheck \in Q\spcheck} c(\alpha\spcheck, q) 
		\sum_{v_1, \ldots, v_n \in \calV}
		\prod_{j = 1}^n c_{v_j}
		\int_V
		e^{\sprod{i\theta}{\sum_{j = 1}^n v_j}} e^{-i\sprod{\theta}{\omega}} e^{i \sprod{\theta}{\alpha\spcheck}}
		{\: \rm d}\theta\\
		&\qquad\qquad=
		\int_V \kappa(i\theta)^n e^{-i\sprod{\theta}{\omega}} \frac{{\rm d}\theta}{\bfc(i\theta)},
	\end{align*}
	proving the claim.

	Thanks to Claim \ref{clm:1}, we can choose the shift $u \in \mathfrak{b}$ depending on $\omega_n$ in such a way
	that the critical point of the phase function is at $\theta = 0$.
	
Let us notice that, if $p(n; v_n) > 0$ then $\delta_n = n^{-1} \omega_n \in \cl \calM$. Since
$\dist(\delta_n, \partial \calM) > 0$, by Theorem \ref{thm:1}, there is the unique $s_n = s(\delta_n)$ such that
$\nabla \log \kappa(s_n) = \delta_n$. Hence, by Claim \ref{clm:1}, we can write
\begin{align*}
	\calF_n(\omega_n)  
	= 
	e^{-n \phi(\delta_n)}
	\int_{U_0} \bigg(\frac{\kappa(s_n + i\theta)}{\kappa(s_n)}\bigg)^n 
	e^{-i \sprod{\theta}{\omega_n}} \frac{{\rm d} \theta}{\bfc(s_n + i \theta)},
\end{align*}
where
\[
	\phi(\delta) = \sprod{\delta}{s} - \log \kappa(s).
\]
Let $\epsilon > 0$ be small enough to satisfy \eqref{eq:52} and \eqref{eq:38}. We set
\[
	U_\epsilon = \big\{\theta \in \mathfrak{a} : \sprod{\theta}{\alpha\spcheck} < \epsilon, \text{ for all }
	\alpha \in \Phi \big\}.
\]
With a help of Theorem \ref{thm:2}, we can show that the integral over $U_0 \setminus U_\epsilon$ is negligible. 
To see this, we write
\begin{align}
	\nonumber
	1- \bigg|\frac{\kappa(u+i\theta)}{\kappa(u)}\bigg|^2 
	&= 1 - \sum_{v, v' \in \calV} 
	\frac{c_v e^{\sprod{u+i\theta}{v}}}{\kappa(u)}
	\cdot
	\frac{c_{v'} e^{\sprod{u-i\theta}{v'}}}{\kappa(u)} \\
	\label{eq:20}
	&=2 \sum_{v, v' \in \calV} \scoef{u}{v} \cdot \scoef{u}{v'}
	\Big(\sin \Big\langle \frac{\theta}{2}, v-v'\Big\rangle\Big)^2.
\end{align}
We need to show that for each $\theta \in U_0 \setminus U_\epsilon$, there is
always at least one nonzero term in \eqref{eq:20}. In fact, we show the following statement.
\begin{claim}
	\label{clm:2}
	For every $v_0 \in \calV$, there is $\xi > 0$ such that for all $\theta \in U_0 \setminus U_\epsilon$ there is
	$v' \in \calV$ satisfying
	\[
		\Big|
		\sin \Big\langle \frac{\theta}{2}, v' - v_0 \Big\rangle
		\Big|
		\geq
		\xi.
	\]
\end{claim}
For the proof, we assume to contrary that for some $v_0 \in \calV$ and all $m \in \NN$ there is
$\theta_m \in U_0 \setminus U_\epsilon$ such that for all $v \in \calV$,
\[
	\Big|
    \sin \Big\langle \frac{\theta_m}{2}, v - v_0 \Big\rangle
    \Big|
    \leq
	\frac{1}{m}.
\]
By compactness of $U_0 \setminus U_\epsilon$, there is a subsequence $(\theta_{m_k} : k \in \NN)$ convergent to
$\theta' \in U_0 \setminus U_\epsilon$. Then for all $v \in \calV$,
\[
	\sin \Big\langle \frac{\theta'}{2}, v -v_0\Big\rangle = 0,
\]
and thus $\abs{\kappa(i\theta')} = 1$, which is impossible since $0 = \theta' \notin U_0 \setminus U_\epsilon$.

Before we apply Claim \ref{clm:2}, we select any $v_0 \in \calV$ satisfying
\[
	\sprod{s_n}{v_0} = \max\big\{\sprod{s_n}{v} : v \in \calV\big\},
\]
thus $e^{\sprod{s_n}{v_0}} \geq \kappa(s_n)$. By Claim \ref{clm:2} and \eqref{eq:20}, for each
$\theta \in U_0 \setminus U_\epsilon$ there is $v' \in \calV$ such that
\begin{align*}
	1 - \bigg|\frac{\kappa(s_n + i \theta)}{\kappa(s_n)}\bigg|^2 
	&\geq 2 c_{v_0} \scoef{s_n}{v'} \xi^2 \\
	&\geq 2\xi^2 \min\big\{c_{v}^2 : v \in \calV\big\} \cdot \frac{e^{\sprod{s_n}{v'}}}{\kappa(s_n)}.
\end{align*}
Although $v'$ may depend on $\theta$ and $n$, by Theorem \ref{thm:2}, there are $C> 0$ and $\eta \geq 1$ such that for all
$\theta \in U_0 \setminus U_\epsilon$ and all $n \in \NN$,
\[
	1 - \bigg|\frac{\kappa(s_n + i \theta)}{\kappa(s_n)}\bigg|^2 \geq C \dist(\delta_n, \partial \calM)^\eta.
\]
Hence,
\[
	\bigg|\frac{\kappa(s_n + i \theta)}{\kappa(s_n)}\bigg|^2 \leq
	1 - C \dist(\delta_n, \partial \calM)^\eta \leq e^{-C \dist(\delta_n, \partial \calM)^\eta}.
\]
Since
\[
	\Bigg|
	\frac
	{1 - \tau_{\alpha/2}^{-1/2} e^{-\sprod{s_n+i\theta}{\alpha\spcheck}}}
	{1 - \tau_\alpha^{-1} \tau_{\alpha/2}^{-1/2} e^{-\sprod{s_n + i \theta}{\alpha\spcheck}}}
	\Bigg|
	\leq
	\frac{2}{1-\tau_\alpha^{-1} \tau_{\alpha/2}^{-1/2}},
\]
we conclude that
\begin{equation}
	\label{eq:79}
	\bigg|
	\int_{U_0 \setminus U_\epsilon}
	\bigg(\frac{\kappa(s_n + i \theta)}{\kappa(s_n)}\bigg)^n e^{-i\sprod{\theta}{\omega_n}}
	\frac{{\rm d} \theta}{\bfc(s_n + i\theta)}
	\bigg|
	\leq
	C \exp\big\{-C' n \dist(\delta_n, \partial \calM)^{\eta}\big\}.
\end{equation}
The argument above reduced the problem to studying the integral over $U_\epsilon$. Observe that, by \eqref{eq:3}, 
the function $\Log \kappa$ is analytic in a strip $\mathfrak{a} + i B$ where $\Log$ denotes the principal value of
the complex logarithm and
\[
	B = \Big\{\theta \in \mathfrak{a} : \norm{\theta} < \big(2 \cdot \max\{\norm{v} : v \in \calV\}\big)^{-1} \Big\}.
\]
Let $F_n$ be a function on $\mathfrak{b}$ defined by
\begin{equation}
	\label{eq:85}
	F_n(x) = \int_{U_\epsilon} e^{n \varphi(x, \theta)} \frac{{\rm d} \theta}{\bfc(x+i\theta)}
\end{equation}
wherein
\begin{equation}
	\label{eq:62}
	\varphi(x, \theta) = \Log \kappa(x+i\theta) - \Log \kappa(x) - i \sprod{\theta}{\nabla \log \kappa(x)},
\end{equation}
provided that $\epsilon$ is sufficiently small to guarantee that
\begin{equation}
	\label{eq:52}
	U_\epsilon \subseteq B.
\end{equation}
Hence, by \eqref{eq:79},
\[
	\calF_n(\omega_n) = e^{-n\phi(\delta_n)} \big(F_n(s_n) + E_n(\delta_n)\big)
\]
where
\[
	\abs{E_n(\delta_n)} \leq C\exp\big\{-C' n \dist(\delta_n, \partial \calM)^{\eta} \big\}. 
\]
Therefore, our aim is to find the asymptotic behavior of $(F_n(s_n) : n \in \NN)$. We notice that $F_n(x)$ is
an oscillatory integral depending depending on $x \in \mathfrak{b}$, and its asymptotic behavior depends on stabilizer
subgroup of $W_0$ with respect to $x$.

We start by proving some estimates on $\varphi$. Since for any $u, u' \in \mathfrak{a}$ and $z \in \mathfrak{a} + i B$
we have
\[
	D_u D_{u'} \Log \kappa(z) = \frac{1}{2} \sum_{v, v' \in \calV} \scoef{z}{v} \cdot \scoef{z}{v'} \sprod{u}{v-v'}
	\sprod{u'}{v-v'},
\]
by Lemma \ref{lem:2}, there is $C > 0$ such that for all $\sigma \in \NN^r$,
\begin{equation}
	\label{eq:67}
	\Big|
	\der{\sigma}  \big(D_u D_{u'} \Log \kappa(z) \big)
	\Big|
	\leq
	C^{\abs{\sigma}} \sigma! \sqrt{B_x(u, u) B_x(u', u')}
\end{equation}
where $z = x + i \theta$. By using the integral form for the reminder, we can write
\[
	\psi(x, \theta) = \varphi(x, \theta) - \frac{1}{2}B_x(\theta, \theta) = 
	-\frac{i}{2} \int_0^1 (1-t)^2 D_\theta^3 \Log \kappa(x+i \theta t) {\: \rm d}t.
\]
In view of \eqref{eq:67}, there is $c > 0$ such that for all $x \in \mathfrak{a}$ and $\theta \in B$,
\begin{equation}
	\label{eq:42}
	\abs{\psi(x, \theta)} \leq c \norm{\theta} B_x(\theta, \theta).
\end{equation}
Therefore, by choosing
\begin{equation}
	\label{eq:38}
	\epsilon < 
	\bigg(4 \cdot 
	\sup\bigg\{\frac{\abs{\psi(a, b)}}{\norm{b} B_a(b, b)} : a \in \mathfrak{a}, b \in B 
	\bigg\}\bigg)^{-1},
\end{equation}
if $\norm{\theta} < \epsilon$, then we may estimate
\begin{equation}
	\label{eq:23}
	\abs{\psi(x, \theta)} \leq \tfrac{1}{4} B_x(\theta, \theta).
\end{equation}
Hence,
\begin{equation}
	\label{eq:53}
	\Re \varphi(x, \theta) \leq -\tfrac{1}{4} B_x(\theta, \theta). 
\end{equation}
We next observe that the function $F_n$ is real-analytic on $\mathfrak{b}$. To see this, let us choose in $\mathfrak{a}$
coordinates $x_j = \sprod{x}{\alpha_j}$. By Lemma \ref{lem:2}, there is $C > 0$ such that for all $\mu \in \NN^r$
and $x + i \theta \in \mathfrak{b} + iU_\epsilon$,
\[
	\bigg|\der{\mu}_x\bigg(\frac{1}{\bfc(x+i\theta)}\bigg) \bigg| \leq C^{\abs{\mu}+1} \mu!.
\]
For $\nu \in \NN^r$, by Lemma \ref{lem:1} together with estimates \eqref{eq:67} and \eqref{eq:53}, we have
\[
	\big| \der{\nu}_x  e^{n\varphi(x, \theta)} \big| 
	\leq 
	C^{\abs{\nu}+1} 
	\sum_{\pi \in \Pi_\nu} c_\pi e^{-\frac{n}{4} B_x(\theta, \theta)} 
	(nB_x(\theta, \theta))^m \prod_{j=1}^m B_j!.
\]
Since
\[
	e^{-\frac{n}{4} B_x(\theta, \theta)} (nB_x(\theta, \theta))^m \leq 8^m m! e^{-\frac{n}{8} B_x(\theta, \theta)},
\]
by \eqref{eq:22}, we obtain
\begin{align}
	\nonumber
	\big| \der{\nu}_x  e^{n\varphi(x, \theta)} \big| 
	&\leq
	C^{\abs{\nu}+1} e^{-\frac{n}{8}B_x(\theta, \theta)} 
	\sum_{\pi \in \Pi_\nu} 8^m c_\pi m! \prod_{j=1}^m B_j!
	\\
	\label{eq:54} 
	&\leq 
	C^{\abs{\nu}+1} \nu! e^{-\frac{n}{8} B_x(\theta, \theta)}.
\end{align}
Therefore,
\begin{align*}
	\big|\der{\sigma}_x F_n (x) \big| 
	&\leq 
	C^{\abs{\sigma}+1} \sigma! \int_{U_\epsilon} e^{-\frac{n}{8} B_x(\theta, \theta)} \dth\\
	&\leq
	C^{\abs{\sigma}+1} \sigma! n^{-\frac{r}{2}} \big(\det B_x \big)^{-\frac{1}{2}},
\end{align*}
which implies that $F_n$ is real-analytic.

We start with the case $J \neq \emptyset$. Our aim is to describe the asymptotic behavior of $F_n(x)$ close to walls. 
Let $x_0 \in \partial \mathfrak{a}_+$ be such that $\sprod{x_0}{\alpha_j\spcheck} = 0$ for all $j \in J$. By $\Gamma_\Psi$
we denote the set of all multi-indices $\gamma$ such that $\der{\gamma} \bpi_\Psi \neq 0$ where $\bpi_\Psi$
is defined in \eqref{eq:21}. The following theorem is our key tool.
\begin{theorem}
	\label{thm:3}
	There are $C, C', R > 0$ such that for all $h \in \mathfrak{a}_\Psi$, $\norm{h} \leq R$,
	\begin{equation}
		\label{eq:57}
		F_n(x_0+h) =  \big(\det B_{x_0} \big)^{-\frac{1}{2}}
	\bpi_\Psi \big( B_{x_0}^{-1} \rho \big) 
	\sum_{\gamma \in \Gamma_\Psi} 
	\big(B_{x_0} h\big)^\gamma 
	n^{-\frac{r}{2}-\abs{\Psi^{++}} + \abs{\gamma}} A^\gamma_n(x_0, h)
	+ E_n(x_0, h)
	\end{equation}
	where $A^\gamma_n(x_0, h) = a_\gamma(x_0) + g_\gamma(x_0, h) + E^\gamma_n(x_0, h)$, and
	\begin{equation}
		\label{eq:59}
		\begin{aligned}
		&|a_\gamma(x_0)| \leq C, \qquad\qquad 
		&&|E_n^{\gamma}(x_0, h)| \leq C n^{-1} \mnorm{B_{x_0}^{-1}},\\
		&|g_{\gamma}(x_0, h)|   \leq C \norm{h},
		&&|E_n(x_0, h)| \leq C \exp\big\{-C' n \mnorm{B_{x_0}^{-1}}^{-1}\big\}.
		\end{aligned}
	\end{equation}
	The constants $C$, $C'$ and $R$ are independent of $x_0$ and $n$.
\end{theorem}
\begin{proof}
	We start by changing coordinates in $\mathfrak{a}$, namely for $x \in \mathfrak{a}$ we write
	\[
		x_j = \begin{cases}
			\sprod{x}{\alpha_j} & \text{if } j \in J, \\
			\sprod{x}{(I-T_\Psi)\alpha_j} & \text{if } j \in I_0 \setminus J.
		\end{cases}
	\]
	Therefore,
	\[
		\partial_j = 
		\begin{cases}
			D_{T_\Psi \lambda_j} & \text{if } j \in J, \\
			D_{\lambda_j} & \text{if } j \in I_0 \setminus J.
		\end{cases}
	\]
	Observe that for $j \in I_0 \setminus J$, we have $T_\Psi \lambda_j = 0$, because for any $k \in J$,
	\[
		\sprod{T_\Psi \lambda_j}{\alpha_k} = \sprod{\lambda_j}{\alpha_k} = 0.
	\]
	Since for all $w \in W_0$ and $x \in \mathfrak{a}$,
	\[
		B_{w \cdot x}(w \cdot u, w \cdot u') 
		= D_{w \cdot u} D_{w \cdot u'} \log \kappa(w \cdot x) = D_u D_{u'} \log \kappa(x) = B_x(u, u),
	\]
	for any $\alpha \in \Psi$, we have
	\[
		B_{x_0}(r_\alpha u, r_\alpha u') = B_{x_0}(u, u').
	\]
	Thus, for $k \in I_0$, $j \in J$, $j \neq k$, we have
	\[
		B_{x_0}(\lambda_k, \alpha_j) = B_{x_0}(r_j\lambda_k, r_j \alpha_j)
		= - B_{x_0}(\lambda_k, \alpha_j) = 0.
	\]
	Therefore, by setting
	\[
		\rho = \frac{1}{2} \sum_{\alpha \in \Phi^{++}} \alpha\spcheck = \sum_{j = 1}^r \lambda_j,
	\]
	we have
	\[
		B_{x_0} \alpha_j = B_{x_0}(\lambda_j, \alpha_j) \alpha_j = B_{x_0}(\rho, \alpha_j) \alpha_j.
	\]
	Hence, 
	\begin{equation}
		\label{eq:64}
		\sprod{B_{x_0} \rho}{\alpha_j} = \sprod{B_{x_0} \lambda_j}{\alpha_j}
		=\sprod{B_{x_0} T_\Psi \lambda_j}{\alpha_j}.
	\end{equation}
	Moreover, we have
	\begin{align}
		\nonumber
		T_\Psi B_{x_0} T_\Psi \lambda_j 
		&= \sum_{k \in J} \sprod{T_\Psi B_{x_0} T_{\Psi} \lambda_j}{\alpha_k} T_\Psi \lambda_k \\
		\label{eq:65}
		&= \sprod{B_{x_0} \rho}{\alpha_j} T_\Psi \lambda_j.
	\end{align}
	Without loss of generality, we may replace $\epsilon$ by any $0 < \epsilon' < \epsilon$. Indeed, by \eqref{eq:53} we
	have
	\[
		\bigg|\int_{U_\epsilon \setminus U_{\epsilon'}}	e^{n \varphi(x, \theta)} \frac{{\rm d} \theta}{\bfc(x+i\theta)}
		\bigg|
		\leq
		C \int_{U_\epsilon \setminus U_{\epsilon'}} e^{-\frac{n}{4} B_x(\theta, \theta)} {\: \rm d}\theta.
	\]
	Since the mapping $\mathfrak{a} \ni x \mapsto B_x(\theta, \theta)$ is
	real-analytic, by \eqref{eq:67} we have
	\[
		B_x(\theta, \theta) \geq \big(1 - C \norm{h}\big) B_{x_0}(\theta, \theta) \geq 
		\frac{1}{2} B_{x_0}(\theta, \theta),
	\]
	provided that $\norm{h} < (2 C)^{-1}$. Hence,
	\[
		\bigg|\int_{U_\epsilon \setminus U_{\epsilon'}} e^{n \varphi(x, \theta)} \frac{{\rm d} \theta}{\bfc(x+i\theta)}
        \bigg|
        \leq
		C \exp\big\{-\tfrac{1}{8} n \mnorm{B_{x_0}^{-1}}^{-1} \big\}.
	\]
	We next define a function $f$ on $\mathfrak{b}+iU_\epsilon$ by the formula
	\begin{align*}
		f(z) = \frac{1}{\bfc(z) \bpi_\Psi(z)}
		= \frac{1}{\bfb(z)} 
		\bigg(\prod_{\alpha \in \Phi^{++}\setminus\Psi^{++}} 1 - e^{-\sprod{z}{\alpha\spcheck}}\bigg)
		\bigg(\prod_{\alpha \in \Psi^{++}} \frac{1 - e^{-\sprod{z}{\alpha\spcheck}}}{\sprod{z}{\alpha\spcheck}}\bigg).
	\end{align*}
	Observe that each factor is real-analytic on $\mathfrak{b} + i U_\epsilon$, thus there is $C > 0$ such that
	for all $\mu, \nu \in \NN^r$ and $x+i\theta \in \mathfrak{b} + i U_\epsilon$,
	\begin{equation}
		\label{eq:55}
		\big\lvert \der{\mu}_\theta \der{\nu}_x f(x+i\theta) \big\rvert
		\leq C^{\abs{\nu}+\abs{\mu}+1} \nu! \mu!.
	\end{equation}
	We are going to show that there are positive constants $C$ and $C'$ such that for
	any $\sigma \in \mathbb{N}^J$,
	\begin{equation}
		\label{eq:56}
		\der{\sigma} F_n(x_0) 
		=
		\big(\det B_{x_0} \big)^{-\frac{1}{2}}
		\bpi_\Psi 
		\brho{-1}
		\sum_{\atop{\gamma \in \Gamma_\Psi}{\gamma \preceq \sigma}}
		\brho{}^{\gamma} 
		n^{- \frac{r}{2} - \abs{\Psi^{++}} + \abs{\gamma}}
		A^{\gamma}_{n, \sigma}(x_0)
		+ E_{n, \sigma}(x_0)
	\end{equation}
	where $A_{n, \sigma}^{\gamma} (x_0) = a^\sigma_{\gamma}(x_0) + E^\gamma_{n, \sigma}(x_0)$ and
	\begin{equation}
		\label{eq:60}
		\begin{aligned}
		&|a_\gamma^\sigma(x_0)| \leq C^{\abs{\sigma}+1} \sigma!, \qquad 
		&&|E_{n, \sigma}^{\gamma}(x_0)| 
		\leq C^{\abs{\sigma}+1} \sigma! n^{-1} \mnorm{B_{x_0}^{-1}},\\
		&
		&&|E_{n, \sigma}(x_0)| \leq C^{\abs{\sigma}+1} \sigma! \exp\big\{-C' n \mnorm{B_{x_0}^{-1}}^{-1}\big\}.
		\end{aligned}
	\end{equation}
	Recall that for two multi-indices $\sigma, \gamma \in \NN^r$, we write $\gamma \preceq \sigma$ if and only if
	$\gamma(j) \leq \sigma(j)$ for all $j \in I_0$. Let us check that \eqref{eq:56} implies \eqref{eq:57}. Notice that, 
	by \eqref{eq:64}, for $\gamma \in \Gamma_\Psi$,
	\[
		\big(B_{x_0} \rho\big)^\gamma h^\gamma 
		= \prod_{j \in J} \big(\sprod{B_{x_0} \rho}{\alpha_j} h_j\big)^{\gamma(j)} 
		= \prod_{j \in J} \big(\sprod{B_{x_0} T_\Psi \lambda_j}{\alpha_j} h_j\big)^{\gamma(j)} 
		= \big(B_{x_0} h\big)^\gamma.
	\]
	Since $F_n$ is real-analytic, for $h \in \mathfrak{a}_\Psi$, $\norm{h} < C^{-1}$ we have
	\[
		F_n(x_0+h) 
		=\big(\det B_{x_0} \big)^{-\frac{1}{2}}
		\bpi_\Psi
        \brho{-1}
		\sum_{\gamma \in \Gamma_\Psi}
		\big(B_{x_0} h\big)^\gamma
		n^{-\frac{r}{2} - \abs{\Psi^{++}} + \abs{\gamma}}
		\sum_{\sigma \succeq \gamma}
		h^{\sigma-\gamma} A^\gamma_{n, \sigma}(x_0) + E_n(x_0, h).
	\]
	Thus
	\[
		a_\gamma(x_0) = \frac{1}{\gamma!} a_\gamma^\gamma(x_0), \qquad g_\gamma(x_0, h) 
		= \sum_{\sigma \succ \gamma} \frac{1}{\sigma!} a_\gamma^\sigma(x_0) h^{\sigma-\gamma},
	\]
	and
	\[
		E_n^\gamma(x_0, h) = \sum_{\sigma \succeq \gamma} \frac{1}{\sigma!} E^\gamma_{n, \sigma}(x_0) h^{\sigma-\gamma},
		\qquad
		E_n(x_0, h) = \sum_{\sigma \in \NN^J} \frac{1}{\sigma!}E_{n,\sigma}(x_0) h^\sigma.
	\]
	The estimates \eqref{eq:59} clearly follow from \eqref{eq:60}.
	
	For the proof of \eqref{eq:56}, in view of \eqref{eq:54} and \eqref{eq:55}, we can write
	\[
		\der{\sigma} F_n(x_0) = 
		\sum_{\atop{\mu + \nu = \sigma}{\mu \in \Gamma_\Psi}}
		\frac{\sigma!}{\nu!\mu!}
		\int_{U_\epsilon}
   		\der{\nu}_x
		\Big|_{x = x_0}
    	\Big( e^{n\varphi(x, \theta)} f(x + i\theta) \Big)
	    \der{\mu} \bpi_\Psi(i\theta) \dth.
	\]
	For $\mu + \nu = \sigma$, $\mu \in \Gamma_\Psi$, we set
	\begin{equation}
		\label{eq:61}
		I_n^{\mu\nu} = \int_{U_\epsilon}
		\der{\mu}_\theta \der{\nu}_x
		\Big|_{x = x_0} 
		\Big( e^{n\varphi(x, \theta)} f(x + i\theta) \Big) 
		\bpi_\Psi(\theta) \dth.
	\end{equation}
	Then by the integration by parts, one can show that
	\begin{align*}
 		\bigg| 
		\int_{U_\epsilon} 
		\der{\nu}_x	
		\Big|_{x = x_0} 
		\Big( e^{n \varphi(x, \theta)} f(x + i\theta) \Big)
		\der{\mu} \bpi_\Psi(\theta) \dth 
		- (-1)^{\abs{\mu}} I^{\mu\nu}_n \bigg|
		& \leq C^{\abs{\sigma} + 1} \nu! \mu! 
		\int_{\partial U_\epsilon} e^{-\frac{n}{8} B_{x_0}(\theta, \theta)} {\: \rm d}S(\theta) \\
		& \leq C^{\abs{\sigma} + 1} \nu! \mu! 
		\exp\big\{-C' n  \lVert B_{x_0}^{-1} \rVert^{-1}\big\},
	\end{align*}
	because for $\theta \in \mathfrak{a}$,
	\[
		\big\lVert B_{x_0}^{-1} \big\rVert^{-1} \sprod{\theta}{\theta} \leq B_{x_0}(\theta, \theta).
	\]
	In this way, we have reduced the matter to finding the asymptotic of $I^{\mu\nu}_n$. Let $\gamma$ denote a maximal
	multi-index belonging to $\Gamma_\Psi$ satisfying $\mu \preceq \gamma \preceq \sigma$.
	We claim that
	\begin{equation}
		\label{eq:36}
		I_n^{\mu\nu} = \big(\det B_{x_0} \big)^{-\frac{1}{2}}
		\bpi_\Psi \brho{-1} \brho{}^{\gamma} n^{- \frac{r}{2} - \abs{\Psi^{++}} + \abs{\gamma}} 
		A^{\mu\nu}_n(x_0)
	\end{equation}
	where $A^{\mu\nu}_n(x_0) = a_{\mu\nu}(x_0) + E_n^{\mu\nu}(x_0)$ and
	\begin{align*}
		&|a_{\mu\nu}(x_0)|    \leq C^{\abs{\sigma} + 1} \mu!\nu!, \qquad
		&&|E_n^{\mu\nu}(x_0)|  \leq C^{\abs{\sigma} + 1} \mu!\nu! n^{-1}
		\big\lVert B_{x_0}^{-1} \big\rVert.
	\end{align*}
	We emphasize that the degree of $\bpi_\Psi(\theta)$ is $\abs{\Psi^{++}}$, thus the main difficulty in showing
	\eqref{eq:36} lies in finding the remaining cancellations. To do so, by Leibniz's rule together with Lemma \ref{lem:1}
	we express the integrand in \eqref{eq:61} as a linear combination of terms of a form
	\[
		n^m
		e^{n\varphi(x_0, \theta)} 
		\bigg(
		\prod_{j = 1}^m
		\der{\mu_j}_\theta \der{\nu_j}_x \varphi(x_0, \theta)\bigg) \der{\mu_0}_\theta\der{\nu_0}_x f(x_0+i\theta)
	\]
	where $m \in \NN$, $(\nu_j : 0 \leq j \leq m)$ and $(\mu_j : 0 \leq j \leq m)$ are sequences of multi-indices such
	that $\abs{\mu_j}+\abs{\nu_j} \geq 1$ for $j \geq 1$ and
	\[
		\mu = \sum_{j = 0}^m \mu_j, \qquad \nu = \sum_{j = 0}^m \nu_j.
	\]
	Therefore, to prove \eqref{eq:36} it is enough to establish the asymptotic of
	\begin{equation}
		\label{eq:41}
		I_n = \int_{U_\epsilon} 
		e^{n \varphi(x_0, \theta)} 
		\Big( \prod_{j=0}^m g_j(\theta) \Big)
		\bpi_\Psi(\theta) \dth
	\end{equation}
	where
	\[
		g_0(\theta) = \frac{1}{\nu_0! \mu_0!} \der{\mu_0}_\theta\der{\nu_0}_x f(x_0+i\theta),
	\]
	and for $j \in \{1, \ldots, m\}$,
	\[
		g_j(\theta) = \frac{1}{\nu_j! \mu_j!} \der{\mu_j}_\theta \der{\nu_j}_x \varphi(x_0, \theta).
	\]
	We claim that
	\begin{equation}
 		\label{eq:19}
		I_n = \big(\det B_{x_0}\big)^{-\frac{1}{2}} 
		\bpi_\Psi \brho{-1} \brho{}^{\gamma}
		n^{-\frac{r}{2} - \abs{\Psi^{++}} - m + \abs{\gamma}} A_n(x_0)
	\end{equation}
	where $A_n(x_0) = a(x_0) + E_n(x_0)$ and
	\begin{align*}
		&|a(x_0)|    \leq C^{\abs{\sigma} + 1}, \qquad
		&&|E_n(x_0)|  \leq C^{\abs{\sigma} + 1} n^{-1} \big\lVert B_{x_0}^{-1} \big\rVert.
	\end{align*}
	For the proof, let $J_\gamma = \{j \in J : \gamma + e_j \in \Gamma_\Psi\}$. We 
	introduce an auxiliary root system
	\[
		\Upsilon = \big\{\alpha \in \Psi : \sprod{\alpha}{\lambda_j} = 0 \text{ if } j \notin J_\gamma\big\}.
	\]
	This is the main idea that the root subsystem $\Upsilon$ describes the remaining symmetries of the
	integrand in \eqref{eq:41}.

	For a multi-index $\beta \in \NN^r$, we set
	\[
		\beta'(j) = 
		\begin{cases}
			\beta(j) & \text{if } \alpha_j \in \Upsilon,\\
			0        & \text{otherwise,}
		\end{cases}
	\]
	and $\beta'' = \beta - \beta'$. Let 
	\[
		\Lambda_0 = \big\{1 \leq j \leq m: \abs{\nu_j'}+\abs{\mu_j'} = 0 \big\},
		\qquad\text{and}\qquad
		\Lambda_0^c = \big\{0 \leq j \leq m : j \notin \Lambda_0 \big\}.
	\]
	We construct a sequence of multi-indices $(\beta_j : 0 \leq j \leq m)$ as follows: if $j \in \Lambda_0$ then we take 
	$\beta_j \preceq \mu_j$, $\abs{\beta_j} = \min\{2, \abs{\mu_j}\}$, otherwise
	$\beta_j \preceq 2(\nu_j'+\mu_j')$, $\abs{\beta_j} = 2$ and $\abs{\beta_0} = 0$. Let 
	$\beta = \sum_{j=0}^m \beta_j$. By maximality of $\gamma$ we have $\mu'' \preceq\gamma''$ and $\gamma' = \nu' + \mu'$,
	thus
	\begin{equation}
		\label{eq:12}
		\beta \preceq \mu'' + 2 \gamma' \preceq \gamma + \gamma'.
	\end{equation}
	We set 
	\[
		K_0 = \abs{\Upsilon^{++}} + \sum_{j \in \Lambda_0} (2-\abs{\beta_j}) = \abs{\Upsilon^{++}} + 2m - \abs{\beta}.
	\]
	Let us notice that if $j \in \Lambda_0$, then the function $g_j$ is $W_0(\Upsilon)$-invariant. 
	Indeed, by \eqref{eq:62}, for all $\alpha \in \Psi$ and $\theta \in U_\epsilon$,
	\[
		\varphi(x_0, r_\alpha \theta) = \varphi(r_\alpha x_0, r_\alpha \theta) = \varphi(x_0, \theta).
	\]
	Since $j \in \Lambda_0$, if $\nu_j(k) + \mu_j(k) > 0$ then $\alpha_k \in \Psi \setminus \Upsilon$,
	thus $\sprod{\alpha}{\lambda_k} = 0$ for all $\alpha \in \Upsilon$. Hence, for all $\theta \in U_\epsilon$
	and $\alpha \in \Upsilon$,
	\[
		\der{\mu_j}_\theta \der{\nu_j}_x \varphi(x_0, r_\alpha \theta) 
		= \der{\mu_j}_\theta \der{\nu_j}_x \varphi(x_0, \theta).
	\]
	Therefore, we may write
	\begin{equation}
		\label{eq:45}
		I_n = \frac{1}{\abs{W_0(\Upsilon)}}
		\int_{U_\epsilon}
		e^{n\varphi(x_0, \theta)}
		G(\theta) \bpi_\Psi(\theta)\dth
	\end{equation}
	where
	\[
		G(\theta) = \bigg(\prod_{j \in \Lambda_0} g_j(\theta) \bigg) \sum_{w \in W_0(\Upsilon)} (-1)^{\ell(w)} 
		\prod_{j \in \Lambda_0^c} g_j(w \cdot \theta).
	\]
	To identify cancellations in $I_n$, we need the following two propositions.
	\begin{proposition}
		\label{prop:1}
		Let $\tau \in \mathbb{N}^r$, $\abs{\tau} \geq 2$. If $\tau(k) \geq 1$ for $k \in J$ then
		\[
			|\der{\tau} \log \kappa(x_0)|
			\leq
			C^{\abs{\tau}+1} \tau! \sprod{\alpha_k}{B_{x_0} \rho}.
		\]
	\end{proposition}
	\begin{proof}
		Let $h(x) = \der{\tau - e_k} \log \kappa(x)$. Suppose that for each $j \in I_0$ such that 
		$e_j \preceq \tau - e_k$, we have $\sprod{\lambda_j}{T_\Psi \lambda_k} = 0$. Then
		$\sprod{\lambda_j}{T_\Psi \lambda_k} \neq 0$ implies that $j \in J$ and $h(r_j x) = h(x)$. Hence,
		\[
			D_{\alpha_j} h(x_0) = - D_{r_j \alpha_j} h(x_0) = -D_{\alpha_j} h(x_0) = 0.
		\]
		Therefore,
		\[
			\der{\tau} \log \kappa(x_0) = 
			\der{}_k h(x_0) = \sum_{j \in J} \sprod{\lambda_j}{T_\Psi \lambda_k}
			D_{\alpha_j} h(x_0) = 0.
		\]
		Otherwise, there is $j \in I_0$ such that $e_j \preceq \tau - e_k$ and $\sprod{\lambda_j}{T_\Psi \lambda_k}
		\neq 0$. Since $j \in J$, by \eqref{eq:65},
		\[
			\sprod{B_{x_0}\rho}{\alpha_j}\sprod{T_\Psi \lambda_j}{\lambda_k}
			=\sprod{T_\Psi B_{x_0} T_\Psi \lambda_j}{ \lambda_k}
			=\sprod{\lambda_j}{T_\Psi B_{x_0} T_\Psi \lambda_k}
			=\sprod{B_{x_0}\rho}{\alpha_k}\sprod{\lambda_j}{T_\Psi \lambda_k}
		\]
		and so $\sprod{B_{x_0}\rho}{\alpha_j} = \sprod{B_{x_0}\rho}{\alpha_k}$. Therefore, by \eqref{eq:67},
		we have
		\begin{align*}
			\big|\der{\tau} \log \kappa(x_0) \big| 
			&\leq C^{\abs{\tau} + 1} \tau!
			\sqrt{B_{x_0}(T_\Psi \lambda_k, T_\Psi \lambda_k) B_{x_0}(T_\Psi \lambda_j, T_\Psi \lambda_j)} \\
			&=
			C^{\abs{\tau} + 1} \tau!
			\lvert T_\Psi \lambda_k \rvert \cdot \lvert T_\Psi \lambda_j \rvert
			\sqrt{ \sprod{B_{x_0}\rho}{\alpha_k}  \sprod{B_{x_0}\rho}{\alpha_j}},
		\end{align*}
		which finishes the proof.
	\end{proof}

	In the next proposition, we use our variant of marriage lemma, see Lemma \ref{lem:5}.
	\begin{proposition}
		\label{prop:3}
		We have
		\[
			\bpi_\Upsilon(\rho) \cdot \bpi_\Psi \big(B_{x_0} \rho\big)
			=
			\big(B_{x_0} \rho\big)^{\gamma''} 
			\bpi_\Psi(\rho) \cdot \bpi_\Upsilon \big(B_{x_0} \rho\big).
		\]
	\end{proposition}
	\begin{proof}
		Let $X = \Psi^{++}$ and $C_i = \{\alpha \in \Psi^{++}: \sprod{\alpha}{\lambda_i} > 0\}$.
		Then $\gamma \in \Gamma_\Psi$ is admissible (see Section \ref{subsec:1}). We choose any
		partial partition $(X_j : j \in J)$ corresponding to $\gamma$. 
		
		For any $u \in \mathfrak{a}$ and $\alpha \in \Psi^{++}$, we have
		\[
			\sprod{B_{x_0} \alpha}{u} = -\sprod{B_{x_0} r_\alpha \alpha}{u} = -\sprod{B_{x_0} \alpha}{u} + 
			\sprod{\alpha}{u} \sprod{B_{x_0} \alpha}{\alpha\spcheck},
		\]
		thus 
		\[
			\frac{\sprod{B_{x_0}\alpha}{u}}{\sprod{\alpha}{u}} 
			= \frac{\sprod{B_{x_0}\alpha}{\alpha}}{\sprod{\alpha}{\alpha}},
		\]
		provided that $\sprod{\alpha}{u} \neq 0$. Therefore, for $\alpha \in X_j$,
		\[
			\frac{\sprod{B_{x_0}\alpha}{\rho}}{\sprod{\alpha}{\rho}} 
			=
			\frac{\sprod{B_{x_0}\alpha}{\alpha}}{\sprod{\alpha}{\alpha}}
			=
			\frac{\sprod{B_{x_0}\alpha}{T_\Psi \lambda_j}}{\sprod{\alpha}{T_\Psi \lambda_j}},
		\]
		which together with \eqref{eq:65} implies that 
		\[
			\sprod{B_{x_0}\alpha}{\rho}
        	=
			\sprod{B_{x_0} \rho}{\alpha_j} \sprod{\alpha}{\rho}.
		\]
		Hence, by Lemma \ref{lem:5}, we obtain
		\begin{align*}
			\prod_{\alpha \in \Psi^{++} \setminus \Upsilon^{++}} \sprod{\alpha}{B_{x_0} \rho}
			&=\prod_{j: \alpha_j \in \Psi^{++} \setminus \Upsilon^{++}} 
			\prod_{\alpha \in X_j} \sprod{\alpha}{B_{x_0} \rho} \\
			&=\big(B_{x_0} \rho\big)^{\gamma''} \prod_{j: \alpha_j \in \Psi^{++} \setminus \Upsilon^{++}} 
			\prod_{\alpha \in X_j} 
			\sprod{\alpha}{\rho} 
			=\big(B_{x_0} \rho\big)^{\gamma''} \prod_{\alpha \in \Psi^{++} \setminus \Upsilon^{++}} 
			\sprod{\alpha}{\rho}.
			\qedhere
		\end{align*}
	\end{proof}
	We are now in the position to prove \eqref{eq:19}. Since the function $G$ is real-analytic, we can expand $G(\theta)$
	about $\theta = 0$ into convergent power series. We are going to estimate $D^k_\theta G(0)$ for $k \in \NN$.
	Let $(k_j : j \in \Lambda_0)$ and $k_0 \in \NN$ be such that
	\[
		k_0 + \sum_{j \in \Lambda_0} k_j = k.
	\]
	We first consider $j \in \Lambda_0$. Observe that $k_j + \abs{\mu_j} \geq 2$, otherwise, by \eqref{eq:62}, 
	$D_\theta^{k_j} g_j(0) = 0$. We claim that, 
	\begin{equation}
		\label{eq:14}
		\big
		\lvert
		D_\theta^{k_j} g_j (0)
		\big
		\rvert
		\leq
		C^{k_j + \abs{\nu_j}+\abs{\mu_j}+1} k_j!
		\brho{1/2}^{\beta_j} \btheta^{2-\abs{\beta_j}} 
		\norm{\theta}^{k_j - 2 + \abs{\beta_j}}.
	\end{equation}
	For the proof, we need to consider three cases:	
	
	\noindent 
	{case 1: $\abs{\mu_j} = 0$.} Then $\abs{\beta_j} = 0$ and $k_j \geq 2$. By \eqref{eq:67}, we get
	\begin{align*}
		\big\lvert
        D_\theta^{k_j} g_j (0)
        \big\rvert
		&\leq
		C^{k_j + \abs{\nu_j} + 1} k_j! B_{x_0}(\theta, \theta) \norm{\theta}^{k_j-2} \\
		&=
		C^{k_j + \abs{\nu_j} + 1} k_j! \btheta^2 \norm{\theta}^{k_j-2}.
	\end{align*}

	\noindent
	{case 2: $\abs{\mu_j} = 1$.} Let $\beta_j = e_p$ for $p \in J$. Since $k_j \geq 1$, by \eqref{eq:67} and
	\eqref{eq:65}, we obtain
	\begin{align*}
		\big\lvert
        D_\theta^{k_j} g_j (0)
        \big\rvert
        &\leq
        C^{k_j + \abs{\mu_j}+\abs{\nu_j} + 1} k_j! \sqrt{B_{x_0}(\theta, \theta) 
		B_{x_0}(T_\Psi \lambda_p, T_\Psi \lambda_p)}
		\norm{\theta}^{k_j-1} \\
		&\leq
		C^{k_j + \abs{\mu_j}+\abs{\nu_j} + 1} k_j! \btheta \sprod{B_{x_0}^{1/2} \rho}{\alpha_p} \norm{\theta}^{k_j-1},
	\end{align*}
	
	\noindent
	{case 3: $\abs{\mu_j} \geq 2$.} Let $\beta_j = e_p + e_q$ for $p, q \in J$. Then we conclude that 
	\begin{align*}
        \big\lvert
        D_\theta^{k_j} g_j (0)
        \big\rvert
        &\leq
        C^{k_j + \abs{\mu_j}+\abs{\nu_j} + 1} k_j! \sqrt{B_{x_0}(T_\Psi \lambda_p, T_\Psi \lambda_p) 
		B_{x_0}(T_\Psi \lambda_q, T_\Psi \lambda_q)}
        \norm{\theta}^{k_j} \\
        &\leq
        C^{k_j + \abs{\mu_j}+\abs{\nu_j} + 1} k_j! \sprod{B_{x_0}^{1/2} \rho}{\alpha_p} 
		\sprod{B_{x_0}^{1/2} \rho}{\alpha_q}
		\norm{\theta}^{k_j}.
    \end{align*}	
	Let us next consider a sequence $(\tau_j : j \in \Lambda_0^c)$ of multi-indices from $\NN^r$ such that
	\begin{equation}
		\label{eq:78}
		k_0 = \sum_{j \in \Lambda_0^c} \abs{\tau_j}.
	\end{equation}
	We may assume that
	\begin{equation}
		\label{eq:77}
		\sum_{w \in W_0(\Upsilon)} (-1)^{\ell(w)} \prod_{j \in \Lambda_0^c} 
		\big(w \cdot \theta \big)^{\tau_j} \neq 0,
	\end{equation}
	in particular, $k_0 \geq \abs{\Upsilon^{++}}$. Since the left-hand side of \eqref{eq:77} is anti-invariant under
	the action of $W_0(\Upsilon)$, it is divisible by $\bpi_\Upsilon(\theta)$. Moreover,
	\[
		\bpi_\Upsilon(\theta) = \bpi_\Upsilon\big(B_{x_0}^{-1/2} \rho\big) \bpi_\Upsilon\big(B_{x_0}^{1/2} \theta\big),
	\]
	thus we obtain
	\[
		\Big|\sum_{w \in W_0(\Upsilon)} (-1)^{\ell(w)} \prod_{j \in \Lambda_0^c}
        \big(w \cdot \theta \big)^{\tau_j}\Big|
        \leq
        C
		\bpi_\Upsilon \brho{-1/2}
        \btheta^{\abs{\Upsilon^{++}}}.
	\]
	If $j \in \Lambda_0^c$, $j \geq 1$, then by Proposition \ref{prop:1}, we have
	\begin{equation*}
		\big\lvert \der{\tau_j} g_j(0) \big\rvert
		\leq C^{\abs{\tau_j}+\abs{\nu_j}+\abs{\mu_j}+1} \tau_j! \brho{1/2}^{\beta_j},
	\end{equation*}
	which is also correct for $j = 0$ because $\abs{\beta_0} = 0$. Therefore,
	\begin{align*}
		&
		\bigg|
		\Big(
		\sum_{w \in W_0(\Upsilon)} (-1)^{\ell(w)} \prod_{j \in \Lambda_0^c}
        \big(w \cdot \theta \big)^{\tau_j}
		\Big)
		\prod_{j \in \Lambda_0^c}
		\partial^{\tau_j} g_j(0)
		\bigg|\\
		&\qquad\qquad\leq
		C
		\bpi_\Upsilon \brho{-1/2}
        \btheta^{\abs{\Upsilon^{++}}}
		\norm{\theta}^{k_0-\abs{\Upsilon^{++}}}
		\prod_{j \in \Lambda_0^c} C^{\abs{\tau_j}+\abs{\nu_j}+\abs{\mu_j}+1} \tau_j! \brho{1/2}^{\beta_j}.
	\end{align*}
	By summing up over $(\tau_j : j \in \Lambda_0^c)$ satisfying \eqref{eq:78}, we arrive at
	\begin{align*}
		&\bigg\lvert
		\sum_{w \in W_0(\Upsilon)} (-1)^{\ell(w)} D_{w \cdot \theta}^{k_0} 
		\Big(\prod_{j \in \Lambda_0^c} g_j (\theta) \Big)_{\theta = 0}
		\bigg\rvert\\
		&\qquad\qquad\leq
		C^{k_0} k_0!
		\bpi_\Upsilon \brho{-1/2}
		\btheta^{\abs{\Upsilon^{++}}}
		\norm{\theta}^{k_0 - \abs{\Upsilon^{++}}}
		\prod_{j \in \Lambda_0^c} C^{\abs{\nu_j}+\abs{\mu_j}+1}
		\brho{1/2}^{\beta_j}.
	\end{align*}
	Finally, the above estimate together with \eqref{eq:14} imply that for $k \geq K_0$ we have
	\begin{equation}
		\label{eq:17}
		\big\lvert
		D_\theta^k G(0)
		\big\rvert
		\leq 
		C^{\abs{\sigma}+k +1} k!
		\bpi_{\Upsilon} \brho{-1/2}
		\brho{1/2}^{\beta}
		\btheta^{K_0}
		\norm{\theta}^{k - K_0},
	\end{equation}
	and $D^k_\theta G(0) = 0$ if $k < K_0$. By taking $\epsilon < C^{-1}$, for $\norm{\theta} \leq \epsilon$
	and $K \geq K_0$, we conclude that 
	\begin{equation}
		\label{eq:18}
		\bigg\lvert 
		\sum_{k \geq K} \frac{D_\theta^k G(0)}{k!} 
		\bigg\rvert
		\leq C^{\abs{\sigma}+K+1} \bpi_\Upsilon \brho{-1/2} \brho{1/2}^\beta
		\big\lVert B_{x_0}^{-1} \big \rVert^{(K-K_0)/2}\\
		\btheta^{K}.
	\end{equation}
	In particular, for $\norm{\theta} \leq \epsilon$,
	\begin{equation}
		\label{eq:63}
		\abs{ G(\theta) } \leq
		C^{\abs{\sigma} + K_0  + 1} \bpi_\Upsilon \brho{-1/2} \brho{1/2}^\beta
		\btheta^{K_0}.
	\end{equation}
	We are now ready to prove \eqref{eq:19}. We first treat the case when $K_0 + \abs{\Psi^{++}} \in 2 \mathbb{Z}$. Let us
	write
	\begin{align*}
		e^{n\psi(x_0, \theta)} G(\theta) 
		&= \big(e^{n\psi(x_0, \theta)} - 1 - n\psi(x_0, \theta) \big) G(\theta)
		+ n \psi(x_0, \theta) \bigg( G(\theta) - 
		\frac{D_\theta^{K_0} G(0)}{K_0!} \bigg)\\
		&\phantom{=}
		+ \bigg( G(\theta) - \frac{D_\theta^{K_0} G(0)}{K_0!} 
		- \frac{D_\theta^{K_0+1} G(0)}{(K_0+1)!}\bigg)
		+  n \bigg(\psi(x_0, \theta) - \frac{D_\theta^3 \psi(x_0, 0)}{3!} \bigg)
		\frac{D_\theta^{K_0} G(0)}{K_0!}\\
		&\phantom{=}
		+  n \bigg(\frac{D_\theta^3 \psi(x_0, 0)}{3!} \cdot
		\frac{D_\theta^{K_0} G(0)}{K_0!}\bigg)
		+ \frac{D_\theta^{K_0+1} G(0)}{(K_0+1)!}
		+  \frac{D_\theta^{K_0} G(0)}{K_0!}, 
	\end{align*}
	and split $I_n$ into seven corresponding integrals denoted by $\calI_1, \ldots, \calI_7$, respectively. 
	
	Since for $a \in \CC$,
	\[
		\big|e^a - 1 - a \big| \leq \frac{\abs{a}^2}{2} e^{\abs{a}},
	\]
	by \eqref{eq:42} and \eqref{eq:23}, we can estimate
	\begin{align*}
		\Big|
		e^{n\psi(x_0, \theta)} - 1 - n\psi(x_0, \theta)
		\Big|
		&\leq
		\frac{1}{2}
		e^{\frac{n}4 B_x(\theta, \theta)} \big(n\psi(x_0, \theta)\big)^2 
		\\
		&\leq
		C
		e^{\frac{n}4 B_x(\theta, \theta)} 
		n^2 \big \lVert B_{x_0}^{-1} \big \rVert \cdot
		\btheta^{6}.
	\end{align*}
	In view of Proposition \ref{prop:3},
	\begin{align*}
		\bpi_\Upsilon \big(B_{x_0}^{-1/2} \rho \big) \cdot \bpi_\Psi (\theta) 
		&= 
		\bpi_\Upsilon \big(B_{x_0}^{-1/2} \rho \big) \cdot \bpi_\Psi\big(B_{x_0}^{-1/2} \rho\big) \cdot
		\bpi_\Psi\big(B_{x_0}^{1/2} \theta\big)\\
		&=
		\frac{\bpi_\Upsilon(\rho)}{\bpi_\Psi(\rho)} \big(B_{x_0}^{1/2} \rho\big)^{\gamma''}
		\bpi_\Psi\big(B_{x_0}^{-1} \rho\big) \cdot \bpi_\Psi\big(B^{1/2}_{x_0} \theta \big),
	\end{align*}
	therefore, by \eqref{eq:63} we get
	\begin{align*}
		\abs{\calI_1} 
		&\leq
		C^{\abs{\sigma} + K_0+ 1}
		\bpi_\Psi \brho{-1}
		\brho{1/2}^{\beta+\gamma''}
		n^2 \big \lVert B_{x_0}^{-1} \big \rVert
		\int_\mathfrak{a} e^{-\frac{n}{2} B_{x_0}(\theta, \theta)} 
		\btheta^{K_0+\abs{\Psi^{++}}+6}
		\dth\\
		&\leq
		C^{\abs{\sigma} + K_0 + 1}
		\big(\det B_{x_0}\big)^{-\frac{1}{2}}
		\bpi_\Psi \big( B_{x_0}^{-1} \rho \big) \brho{1/2}^{\beta+\gamma''}
		n^{-\frac{1}{2} (K_0 + \abs{\Psi^{++}}+r)}
		n^{-1} \big \lVert B_{x_0}^{-1} \big \rVert. 
	\end{align*}
	For the second integrand, we use \eqref{eq:42} and \eqref{eq:18} to estimate
	\[
		\bigg|n \psi(x_0, \theta) \bigg(G(\theta) - \frac{D_\theta^{K_0} G(0)}{K_0!}\bigg)\bigg|
		\leq
		C^{\abs{\sigma} + K_0+1} \bpi_\Upsilon\big(B_{x_0}^{-1/2} \rho\big) \big(B_{x_0}^{1/2} \rho\big)^\beta
		n \mnorm{B_{x_0}^{-1}}\cdot
		\big\lvert B_{x_0}^{1/2} \theta \big\rvert^{K_0+4},
	\]
	thus
	\begin{align*}
		\abs{\calI_2} 
		&\leq
		C^{\abs{\sigma}+K_0+1} \bpi_\Psi \brho{-1} \brho{1/2}^{\beta+\gamma''} n \big \lVert B_{x_0}^{-1} \big \rVert
		\int_\mathfrak{a} e^{-\frac{n}{2} B_{x_0}(\theta, \theta)}
        \btheta^{K_0+\abs{\Psi^{++}}+4}
        \dth \\
		&\leq
		C^{\abs{\sigma} + K_0 + 1}
		\big(\det B_{x_0}\big)^{-\frac{1}{2}}
		\bpi_\Psi \big( B_{x_0}^{-1} \rho \big) \brho{1/2}^{\beta+\gamma''}
		n^{-\frac{1}{2} (K_0 + \abs{\Psi^{++}}+r)}
		n^{-1} \big \lVert B_{x_0}^{-1} \big \rVert.
	\end{align*}
	By \eqref{eq:18}, the third integrand is bounded by
	\[
		\bigg|
		G(\theta) - \frac{D_\theta^{K_0} G(0)}{K_0!}
        - \frac{D_\theta^{K_0+1} G(0)}{(K_0+1)!}\bigg|
		\leq
		C^{\abs{\sigma}+K_0+1} \bpi_\Upsilon\big(B_{x_0}^{-1/2} \rho\big)
		\big(B_{x_0}^{1/2} \rho\big)^\beta
		\mnorm{B_{x_0}^{-1}}\cdot
		\big\lvert B_{x_0}^{1/2} \theta \big\rvert^{K_0+2},
	\]
	hence,
	\begin{align*}
		\abs{\calI_3}
		&\leq
		C^{\abs{\sigma}+K_0+1} 
		\bpi_\Psi \brho{-1} 
		\brho{1/2}^{\beta+\gamma''}
		\big \lVert B_{x_0}^{-1} \big \rVert 
        \int_\mathfrak{a} e^{-\frac{n}{2} B_{x_0}(\theta, \theta)}
        \btheta^{K_0+\abs{\Psi^{++}}+2}
        \dth \\
        &\leq
        C^{\abs{\sigma} + K_0 + 1}
        \big(\det B_{x_0}\big)^{-\frac{1}{2}}
        \bpi_\Psi \big( B_{x_0}^{-1} \rho \big) \brho{1/2}^{\beta+\gamma''}
        n^{-\frac{1}{2} (K_0 + \abs{\Psi^{++}}+r)}
        n^{-1} \big \lVert B_{x_0}^{-1} \big \rVert.
	\end{align*}
	For the fourth integrand, we observe that by \eqref{eq:67},
	\begin{align*}
		\bigg|\psi(x_0, \theta) - \frac{D_\theta^3 \psi(x_0, \theta)}{3!}\bigg|
		&=
		\frac{1}{3!} \bigg|\int_0^1 (1-t)^3 D_\theta^4 \Log \kappa(x_0+i \theta t ) {\: \rm d}t\bigg| \\
		&\leq
		C \norm{\theta}^2 B_{x_0}(\theta, \theta).
	\end{align*}
	Therefore, by \eqref{eq:17}, we obtain
	\begin{align*}
		\bigg| n \bigg(\psi(x_0, \theta) - \frac{D_\theta^3 \psi(x_0, 0)}{3!} \bigg)
                \frac{D_\theta^{K_0} G(0)}{K_0!}
		\bigg| 
		&\leq
		C^{\abs{\sigma}+K_0+1} \bpi_\Upsilon\big(B_{x_0}^{-1/2}\rho\big) \big(B_{x_0}^{1/2} \rho\big)^\beta
		n \big\lVert B_{x_0}^{-1} \big\rVert
		\cdot
		\btheta^{K_0+4},
	\end{align*}
	and the corresponding integral is again bounded by
	\[
		C^{\abs{\sigma} + K_0+1}
        \big(\det B_{x_0}\big)^{-\frac{1}{2}}
        \bpi_\Psi \big( B_{x_0}^{-1} \rho \big) \brho{1/2}^{\beta+\gamma''}
        n^{-\frac{1}{2} (K_0 + \abs{\Psi^{++}}+r)}
        n^{-1} \big \lVert B_{x_0}^{-1} \big \rVert.
	\]
	The fifth and sixth integral equal zero because the integrands are odd functions as $3 + K_0 + \abs{\Psi^{++}}$ and
	$1 + K_0 + \abs{\Psi^{++}}$ are odd integers. Finally, by \eqref{eq:17}, we have
	\begin{align*}
		&\bigg|
		\int_{\mathfrak{a}}
		e^{-\frac{n}{2} B_{x_0}(\theta, \theta)}
		\frac{D_\theta^{K_0} G(0)}{K_0!}
		\bpi_\Psi(\theta)
		\dth
		-
		\calI_7 
		\bigg|\\
		&\qquad\qquad\leq
		C^{\abs{\sigma} + K_0+1}
		\bpi_\Psi \big(B_{x_0}^{-1/2}\big) \big(B_{x_0}^{1/2} \rho\big)^{\beta+\gamma''}
		\int_{U_\epsilon^c}
		e^{-\frac{n}{2} B_{x_0}(\theta, \theta)}
		\btheta^{K_0+\abs{\Psi^{++}}}
		\dth \\
		&\qquad\qquad
		\leq 
		C^{\abs{\sigma} + K_0+1} \big(\det B_{x_0}\big)^{-\frac{1}{2}}
		\bpi_\Psi \big( B_{x_0}^{-1} \rho \big) \brho{1/2}^{\beta+\gamma''}
		n^{-\frac{1}{2} (K_0 + \abs{\Psi^{++}}+r)}
		\exp\big\{-C' n \lVert B_{x_0}^{-1} \rVert^{-1}\big\}.
	\end{align*}
	By the change of variables, we obtain
	\[
		\int_{\mathfrak{a}} e^{-\frac{n}{2} B_{x_0}(\theta, \theta)} \frac{D_\theta^{K_0} G(0)}{K_0!}
        \bpi_\Psi(\theta)
        \dth
		=
		n^{-\frac{1}{2}(K_0+\abs{\Psi^{++}}+r)} \int_{\mathfrak{a}} e^{-\frac{1}{2} B_{x_0}(u, u)}
		\frac{D_u^{K_0} G(0)}{K_0!}
        \bpi_\Psi(u)
        {\: \rm d} u,
	\]
	hence, by \eqref{eq:17}, we get
	\[
		\bigg|\int_{\mathfrak{a}} e^{-\frac{1}{2} B_{x_0}(u, u)}
        \frac{D_u^{K_0} G(0)}{K_0!}
        \bpi_\Psi(u)
        {\: \rm d} u
		\bigg|
		\leq
		C^{\abs{\sigma}+K_0+1} \big(\det B_{x_0}\big)^{-\frac{1}{2}}
        \bpi_\Psi \big( B_{x_0}^{-1} \rho \big) \brho{1/2}^{\beta+\gamma''}.
	\]
	Therefore, we conclude that
	\begin{equation}
		\label{eq:68}
		I_n = \big(\det B_{x_0}\big)^{-\frac{1}{2}}
        \bpi_\Psi \big( B_{x_0}^{-1} \rho \big) 
		\brho{1/2}^{\beta+\gamma''}
		n^{-\frac{1}{2}(K_0+\abs{\Psi^{++}}+r)} 
		A_n(x_0).
	\end{equation}
	What is left is to compare the exponents. In view of \eqref{eq:12}, $\beta + \gamma'' \preceq 2 \gamma$, and
	by Proposition \ref{prop:3}, $\abs{\gamma''} + \abs{\Upsilon^{++}} = \abs{\Psi^{++}}$. Therefore,
	\begin{align*}
		K_0 + \abs{\Psi^{++}} - 2m 
		&= \abs{\Upsilon^{++}} - \abs{\beta} + \abs{\Psi^{++}} \\
		&= 2 \abs{\Psi^{++}} - \abs{\beta} - \abs{\gamma''},
	\end{align*}
	giving $2 (\abs{\Psi^{++}} - \abs{\gamma})$ in the case when $\beta + \gamma'' = 2 \gamma$. 
	If $\beta + \gamma'' \prec 2\gamma$, then 
	\[
		\brho{1/2}^{\beta+\gamma''} 
		n^{-\frac{1}{2} (k_0 + \abs{\Psi^{++}}+r)}
		\leq \brho{}^{\gamma}
		n^{-\frac{r}{2} - \abs{\Psi^{++}} -m + \abs{\gamma}}
		\big(n^{-1} \mnorm{B_{x_0}^{-1}}\big)^{\frac{1}{2} (\abs{\beta} + \abs{\gamma''}) - \abs{\gamma}}.
	\]
	which concludes the proof of \eqref{eq:19} when $K_0 + \abs{\Psi^{++}} \in 2\mathbb{Z}$.

	If $K_0 + \abs{\Psi^{++}} \notin 2\mathbb{Z}$, we write
	\[
		e^{n\psi(x_0, \theta)} G(\theta) 
		= (e^{n\psi(x_0, \theta)} - 1)G(\theta)
		+ \bigg(G(\theta) - \frac{D_\theta^{K_0} G(0)}{K_0!}\bigg)
		+ \frac{D_\theta^{K_0} G(0)}{K_0!}.
	\]
	By a reasoning analogous to the previous case, one can show that
	\begin{equation}
		\label{eq:43}
		\abs{I_n} \leq C^{\abs{\sigma}+K_0+1} \big(\det B_{x_0}\big)^{-\frac{1}{2}} 
		\bpi_\Psi \brho{-1} \brho{1/2}^{\beta+\gamma''} n^{-\frac{1}{2} (K_0+\abs{\Psi^{++}}+r)}
		n^{-\frac{1}{2}} \big\lVert B_{x_0}^{-1} \big\rVert^{1/2}.
	\end{equation}
	Since $\abs{\beta} + \abs{\gamma''} = \abs{\Psi^{++}} + 2m - K_0 \not\in 2\mathbb{Z}$, by
	\eqref{eq:12} we get $\abs{\beta} + \abs{\gamma''} < 2 \abs{\gamma}$. Thus
	\[
		\brho{1/2}^{\beta+\gamma''} \leq \brho{}^{\gamma} \big\lVert B_{x_0}^{-1} \big\rVert^{1/2}.
	\]
	Finally, \eqref{eq:68} together with \eqref{eq:43} imply \eqref{eq:19} and the proof of Theorem \ref{thm:3} is
	completed.
\end{proof}

In the generic case, that is when $J = \emptyset$, to determine the asyptotic behaviour of $F_n$ we can use the same reasoning 
as we have applied in Theorem \ref{thm:3} to study $I_n$ for $m=0$, $\mu = \nu = 0$, and
\[
	g_0(\theta) = \frac{1}{\bfc(x_0+i\theta)},
\]
resulting in the following corollary. 
\begin{corollary}
	\label{cor:2}
	If $J = \emptyset$, then there is $C > 0$ such that
	\[
		F_n(x_0) = \big(\det B_{x_0} \big)^{-\frac{1}{2}} n^{-\frac{r}{2}} \big(a_0(x_0) + E_n(x_0)\big)
	\]
	where
	\[
		|a_0(x_0)| \leq C,
		\qquad
		|E_n(x_0)| \leq C n^{-1} \big\|B_{x_0}^{-1}\big\|.
	\]
	The constant $C$ is independent of $x_0$ and $n$.
\end{corollary}

Based on Theorem \ref{thm:3} and Corollary \ref{cor:2}, we can finish the proof the theorem. Indeed, by taking
$h_n = s_n - t_n$, we get
\begin{align}
	\label{eq:35}
	F_n(s_n) = \big(\det B_{t_n}\big)^{-\frac{1}{2}} \bpi_\Psi\big(B_{t_n}^{-1} \rho \big)
	\sum_{\gamma \in \Gamma_\Psi} \big(B_{t_n} h_n \big)^\gamma n^{-\frac{r}{2} - \abs{\Psi^{++}} + \abs{\gamma}}
	A_n^\gamma(t_n, h_n) + E_n(t_n, h_n)
\end{align}
where $A_n^\gamma(t_n, h_n) = a_\gamma(t_n) + g_\gamma(t_n, h_n) + E_n^\gamma(t_n, h_n)$, and
\begin{align*}
	&\abs{a_\gamma(t_n)} \leq C, \qquad\qquad &&\abs{g_\gamma(t_n, h_n)} \leq C \norm{h_n},\\
	&\abs{E^\gamma_n(t_n, h_n)} \leq C n^{-1} \mnorm{B^{-1}_{t_n}}, \qquad\qquad
	&&\abs{E_n(t_n, h_n)} \leq C \exp \big\{-C' n \mnorm{B_{t_n}^{-1}}^{-1} \big\}.
\end{align*}
Now, our task is to estimate $\norm{h_n}$ and $\mnorm{B_{t_n}^{-1}}$ in terms of $\delta_n$. By \eqref{eq:1} and Theorem
\ref{thm:2}, we have
\[
	B_0(u, u) \geq B_{t_n}(u, u) \geq C \dist(\delta_n, \partial \calM)^{2\eta} B_0(u, u).
\]
Hence, we get
\begin{equation}
	\label{eq:39}
	\mnorm{B_{t_n}^{-1}} = \Big(\min\big\{B_{t_n}(u, u) : \norm{u} = 1 \big\}\Big)^{-1}
	\leq C \dist(\delta_n, \partial \calM)^{-2\eta},
\end{equation}
and
\[
	r! \det B_0 \geq \det B_{t_n} \geq C \dist(\delta_n, \partial \calM)^{2r\eta}.
\]
To control $\norm{h_n}$ we estimate $\sprod{s_n}{\alpha\spcheck}$ for $\alpha \in \Psi^+$. By
$W_0$-invariance and Theorem \ref{thm:1}, we have
\[
	\sprod{s_n}{\alpha\spcheck}=-\sprod{r_\alpha s_n}{\alpha\spcheck} =
	-\big\langle s\big(\delta_n - \sprod{\delta_n}{\alpha}\alpha\spcheck\big), \alpha\spcheck\big\rangle.
\]
By the triangle inequality, for any $t \in [0, 1]$ we have
\begin{align*}
	\dist\big(\delta_n - t\sprod{\delta_n}{\alpha} \alpha\spcheck, \partial \calM\big)
	&\geq
	\dist(\delta_n, \partial \calM) - \norm{\alpha\spcheck} \sprod{\delta_n}{\alpha} \\
	&\geq
	\frac{1}{2} \dist(\delta_n, \partial \calM),
\end{align*}
provided that $n$ is large enough because, by \eqref{eq:70} and \eqref{eq:70a},
\[
	\lim_{n \to \infty} \sprod{\delta_n}{\alpha} \dist(\delta_n, \partial \calM)^{-1} = 0.
\]
Hence, by \eqref{eq:39}, we can estimate 
\begin{align*}
	2 \sprod{s_n}{\alpha\spcheck} 
	&= 
	\sprod{s_n}{\alpha\spcheck} 
	-\big\langle s\big(\delta_n - \sprod{\delta_n}{\alpha}\alpha\spcheck\big), \alpha\spcheck\big\rangle\\
	&\leq
	\sprod{\delta_n}{\alpha}
	\sup_{0 \leq t \leq 1} \Big\langle B^{-1}_{s\big(\delta_n - t \sprod{\delta_n}{\alpha}\alpha\spcheck\big)}
	\alpha\spcheck, \alpha\spcheck\Big\rangle\\
	&\leq
	C \sprod{\delta_n}{\alpha} \dist(\delta_n, \partial \calM)^{-2\eta},
\end{align*}
which gives
\begin{equation}
	\label{eq:16}
	\norm{h_n} \leq C \sum_{\alpha \in \Psi^+} \sprod{\delta_n}{\alpha} \dist(\delta_n, \partial \calM)^{-2\eta}.
\end{equation}
Lastly, since
\[
	n \dist(\delta_n, \partial \calM)^\eta 
	= n^{\frac{1}{2}} \big(n \dist(\delta_n, \partial \calM)^{2\eta} \big)^{\frac{1}{2}},
\]
we obtain
\begin{equation}
	\label{eq:28}
	\exp\big\{-C' n \dist(\delta_n, \partial \calM)^\eta \big\}
	\leq
	C'
	\big(\det B_{t_n} \big)^{-\frac{1}{2}} \bpi_\Psi\big(B_{t_n}^{-1} \rho\big) n^{-\frac{r}{2} - \abs{\Psi^{++}}-1}
	\dist(\delta_n, \partial \calM)^{-2\eta}. 
\end{equation}
The argument above allows us to control the approximation in \eqref{eq:35} in terms of $\delta_n$.

We next claim
\begin{claim}
	\label{clm:5}
	There are $R, C > 0$ such that for all $h \in \mathfrak{a}_\Psi$, $\norm{h} \leq R$ and $\alpha \in \Psi$,
	\[
		\big| D_\alpha \log \kappa(x_0+h) - B_{x_0} (\alpha, x_0 + h) \big|
		\leq C \lvert B_{x_0} (\alpha, x_0+h) \rvert \cdot \norm{h}.
	\]
\end{claim}
For the proof, let us observe that $\log \kappa$ is a real-analytic function on $\mathfrak{a}$.
By Lemma \ref{lem:2}, there is $C > 0$ such that for $j \geq 1$,
\[
	D_h^j \log \kappa(x_0) \leq C^{j+1} j! \norm{h}^j
\]
where $C > 0$ is independent of $x_0$. Hence, for $\norm{h} < C^{-1}$ and $\alpha \in \Psi$, we can write
\[
	D_\alpha \log \kappa(x_0+h) = \sum_{k \geq 0} \frac{1}{k!} D^k_h D_\alpha \log \kappa(x_0).
\]
Let us consider $k \geq 2$. Then
\begin{align*}
	D_h^k D_\alpha \log \kappa(x_0) 
	&=
	D_{r_\alpha h}^k D_{r_\alpha \alpha} \log \kappa(x_0) \\
	&=
	-\big(D_h - \sprod{h}{\alpha\spcheck} D_{\alpha}\big)^k D_\alpha \log \kappa(x_0)\\
	&=
	-D_h D_\alpha \log \kappa(x_0) - \sum_{j = 1}^k \frac{k!}{j!(k-j)!} (-1)^j \sprod{h}{\alpha\spcheck}^j 
D_h^{k-j} D_\alpha^{j+1} \log \kappa(x_0).
\end{align*}
For $j \in \{1, \ldots, k\}$, by \eqref{eq:67},
\[
	\big|D^{k-j}_h D^{j+1}_\alpha \log \kappa(x_0) \big| 
	\leq C^{k+1} (k-j)! j! B_{x_0}(\alpha, \alpha) \norm{h}^{k-j} \norm{\alpha}^j.
\]
Since 
\begin{align*}
	\sprod{B_{x_0} \alpha}{h} &= - \sprod{r_\alpha\big( B_{x_0} \alpha \big)}{h} \\
	&= -\sprod{B_{x_0} \alpha}{h} + B_{x_0}(\alpha, \alpha) \sprod{\alpha\spcheck}{h},
\end{align*}
we get
\[
	\big|D_h^k D_\alpha \log \kappa(x_0) \big|
	\leq
	C^{k+1} k! \lvert B_{x_0}(\alpha,h) \rvert \cdot \norm{h}^{k-1}.
\]
Hence,
\[
	\big\lvert D_\alpha \log \kappa(x_0+h) - D_h D_\alpha \log \kappa(x_0) \big\rvert 
	\leq C \lvert B_{x_0}(\alpha, h)\rvert \cdot \norm{h},
\]
proving the claim. 

With a help of Claim \ref{clm:5}, for all $\alpha \in \Psi$,
\[
	\big| \sprod{\delta_n}{\alpha} - \sprod{B_{t_n} h_n}{\alpha} \big|
	\leq
	C \abs{\sprod{B_{t_n} h_n}{\alpha}} \cdot \norm{h_n}.
\]
Therefore, by \eqref{eq:39}--\eqref{eq:28}, we can write
\begin{equation}
	\label{eq:72}
	\calF_n(\omega_n) = 
	\big(\det B_{t_n}\big)^{-\frac{1}{2}} \bpi_\Psi\big(B_{t_n}^{-1} \rho \big)
	n^{-\frac{r}{2}-\abs{\Psi^{++}}}
	e^{-n\phi(\delta_n)}
    \sum_{\gamma \in \Gamma_\Psi} \omega_n^\gamma A_n^\gamma(t_n, h_n)
\end{equation}
where $A_n^\gamma(t_n, h_n) = a_\gamma(t_n) + E_n^\gamma(t_n, h_n)$, and
\begin{align*}
	\abs{a_\gamma(t_n)} \leq C, \qquad\qquad 
	\abs{E_n^\gamma(t_n, h_n)} \leq C \sum_{\alpha \in \Psi^{++}} \big(\sprod{\delta_n}{\alpha} + n^{-1}\big)
	\dist(\delta_n, \partial \calM)^{-2 \eta}.
\end{align*}
Notice that
\[
	\big(\det B_{t_n}\big)^{-\frac{1}{2}} \bpi_\Psi\big(B_{t_n}^{-1} \rho \big)
	c_\Psi
	=
	\int_{\mathfrak{a}}
	e^{-\frac{1}{2}B_{t_n}(u, u)} \abs{\bpi_\Psi(u)}^2 {\: \rm d} u
\]
where
\[
	c_\Psi = \int_{\mathfrak{a}} e^{-\frac{1}{2} \norm{u}^2} \abs{\bpi_\Psi(u)}^2 {\: \rm d} u.
\]
Analogously, in the generic case, by Corollary \ref{cor:2}, we obtain
\[
	\calF_n(\omega_n) = \big(\det B_{s_n}\big)^{-\frac{1}{2}} n^{-\frac{r}{2}} e^{-n\phi(\delta_n)}
	\big(a_0(s_n) + E_n(s_n)\big),
\]
where
\[
	\abs{a_0(s_n)} \leq C, \qquad\qquad
	\abs{E_n(s_n)} \leq C n^{-1} \dist(\delta_n, \partial \calM)^{-2 \eta}.
\]

The final task is to identify the function
\begin{equation}
	\label{eq:46}
	\mathfrak{a}^\perp_\Psi \ni t \mapsto \sum_{\gamma \in \Gamma_\Psi} \eta^\gamma a_\gamma(t).
\end{equation}
To do so, we perform analysis resembling a proof of the local limit theorem towards the wall of 
$\mathfrak{a}_+$. Fix $\omega, \eta \in P^+$ such that $\sprod{\omega}{\alpha} = 0$ for all $\alpha \in \Psi$. 
There is $m$ such that
$V_\eta(O), V_\omega(O) \subseteq p(n; \cdot)$, for all $n \geq m$. By increasing $m$, we may assume that
$\delta = m^{-1} \omega$ belongs to $\calM$. Let $(\omega_n : n \geq j)$ be a
sequence of co-weights such that for $(k+1)m \leq n < (k+2)m$, $k \in \NN \cup \{0\}$, 
\[
	\omega_n = k \omega + \eta.
\]
We set $\delta_n = n^{-1} \omega_n$, $s_n = s(\delta_n)$ and $t = s(\delta)$. By \eqref{eq:72}, we have
\begin{equation}
	\label{eq:73}
	\lim_{n \to \infty} 
	n^{\frac{r}{2}+\abs{\Psi^{++}}} e^{n\phi(\delta_n)} \calF_n(\omega_n)
	=
	\big(\det B_t \big)^{-\frac{1}{2}}
	\bpi_{\Psi}\big(B_t^{-1} \rho\big) \sum_{\gamma \in \Gamma_\Psi} \eta^\gamma a_\gamma(t).
\end{equation}
The limit \eqref{eq:73} can be also computed by different method. By Claim \ref{clm:1},
\[
	\calF_n(\omega_n) =  \kappa(t)^n e^{-\sprod{t}{\omega_n}} 
	\int_U \bigg(\frac{\kappa(t+i\theta)}{\kappa(t)}\bigg)^n 
	e^{-i\sprod{\theta}{\omega_n}} \frac{{\rm d}\theta}{\bfc(t + i \theta)},
\]
thus,
\begin{equation}
	\label{eq:32}
	\calF_n(\omega_n) =  \kappa(t)^n e^{-\sprod{t}{\omega_n}} \bigg(
	\int_{U_\epsilon} e^{n\varphi(t, \theta)} e^{-i\sprod{\theta}{\omega_n - n \delta}} 
	\frac{{\rm d} \theta}{\bfc(t+i\theta)}+E_n(\delta_n)\bigg)
\end{equation}
where
\[
	\abs{E_n(\delta_n)} \leq C e^{-C' n}.
\]
We first show that
\begin{equation}
	\label{eq:75}
	\lim_{n \to \infty} n \big(\phi(\delta_n) - \log \kappa(t) + \sprod{t}{\delta_n}\big) = 0.
\end{equation}
By writing Taylor's polynomial for $\log \kappa$ centered at $t$, we get
\[
	\big| \log \kappa(s_n) - \log \kappa(t) - \sprod{s_n - t}{\nabla \log \kappa(t)} \big|
	\leq C \norm{s_n - t}^2.
\]
Since $\delta = \nabla \log \kappa(t)$, we have
\[
	\big| \sprod{s_n - t}{\nabla \log \kappa(t)} - \sprod{s_n - t}{\delta_n} \big|
	\leq
	\norm{s_n - t} \cdot \norm{\delta_n - \delta}.
\]
Hence, 
\begin{align*}
	\big| \phi(\delta_n) - \log \kappa(t) + \sprod{t}{\delta_n} \big| &=
	\big| \log \kappa(s_n) - \log \kappa(t) - \sprod{s_n - t}{\delta_n} \big| \\
	&\leq
	C \norm{s_n - t}^2 + \norm{s_n-t} \cdot \norm{\delta_n -\delta} \\
	&\leq
	C' \norm{\delta_n - \delta}^2,
\end{align*}
which proves \eqref{eq:75}, because $n \norm{\delta_n - \delta} \leq \norm{\eta} + 2 \norm{\omega}$.

We next deal with the integral over $U_\epsilon$. 
\begin{claim}
	\label{clm:6}
	\[
		\lim_{n \to \infty} n^{\frac{r}{2}+\abs{\Psi^{++}}} 
		\int_{U_\epsilon} 
		e^{n\varphi(t, \theta)} e^{-i\sprod{\theta}{\omega_n-n\delta}} 
		\frac{{\rm d}\theta}{\bfc(t+i\theta)}
		=
		(2\pi)^r
		\calQ_\Psi(t)  
		\frac{G_0(\eta)}{\abs{\bfb_\Psi(0)}^2}
	\]
	where
	\[
		G_0(\eta) = \lim_{\theta \to 0}
		\frac{1}{\abs{W_0(\Psi)}}
		\sum_{w \in W_0(\Psi)}
		e^{-i\sprod{w \cdot \theta}{\eta}} \bfc_\Psi(-i w \cdot \theta).
	\]
\end{claim}
For the proof, we consider a sequence of functions on $U_\epsilon$ defined by
\[
	f_n(\theta) = e^{n\psi(t, \theta)}
	e^{-i\sprod{\theta}{k \omega + (I-T_\Psi)\eta- n \delta}}
	\frac{\bfc_\Psi(i\theta)}{\bfc(t + i \theta)}
	\cdot
	\frac{1}{\big|\bpi_\Psi(\theta) \bfc_\Psi(i \theta) \big|^2}.
\]
Note that a simple reflection $r_j$ for $j \in J$, sends $\alpha_j$ to $-\alpha_j$ and permutes elements in 
$\Phi^+\setminus \Psi^+$. Consequently, $f_n$ is $W_0(\Psi)$-invariant. Since
\[
	\frac{1}{ \big|\bpi_\Psi(\theta) \bfc_\Psi(i \theta) \big|}
	= \prod_{\alpha \in \Psi^{++}} 
	\bigg| \frac{1 - e^{-i\sprod{\theta}{\alpha\spcheck}}}{\sprod{\theta}{\alpha\spcheck}}\bigg| 
	\cdot
	\Big|1-\tau_{2\alpha}^{-1} \tau_\alpha^{-1/2} e^{-\frac{i}{2}\sprod{\theta}{\alpha\spcheck}}\Big|^{-1}
	\cdot 
	\Big|1 + \tau_{\alpha}^{-1/2} e^{-\frac{i}{2} \sprod{\theta}{\alpha\spcheck}}\Big|^{-1}
	\leq C,
\]
and
\[
	\bigg|\frac{\bfc_\Psi(i\theta)}{\bfc(t+i\theta)} \bigg|
	=
	\prod_{\alpha \in \Phi^+\setminus\Psi^+}
	\left|\frac{1 -\tau_{\alpha/2}^{-1/2} e^{-\sprod{t+i\theta}{\alpha\spcheck}}}
	{1 - \tau_{\alpha}^{-1} \tau_{\alpha/2}^{-1/2} e^{-\sprod{t+i\theta}{\alpha\spcheck}}}
	\right|
	\leq
	C,
\]
by \eqref{eq:23}, for $u \in U_{\sqrt{n} \epsilon}$ we get
\[
	\abs{f_n(n^{-1/2} u)} \leq e^{\frac{1}{4} B_t(u, u)}.
\]
Moreover, we have
\[
	\big|\big\langle n^{-1/2} u, k \omega + (I-T_\Psi)\eta - n \delta \big\rangle \big| \leq n^{-1/2} \norm{u} 
	\big(\norm{\eta} + 2 \norm{\omega}\big),
\]
thus, by \eqref{eq:42}, we obtain
\[
	\lim_{n \to \infty} f_n\big(n^{-1/2} u\big) =
	\frac{1}{\abs{b_\Psi(0)}^2}
	\cdot
	\prod_{\alpha \in \Phi^+\setminus\Psi^+} 
	\frac{1 - \tau_{\alpha/2}^{-1/2} e^{-\sprod{t}{\alpha\spcheck}}} 
	{1 - \tau_{\alpha}^{-1} \tau_{\alpha/2}^{-1/2} e^{-\sprod{t}{\alpha\spcheck}}}.
\]
We now use $W_0(\Psi)$-invariance of $f_n$ to write
\[
	\int_{U_\epsilon} e^{n \varphi(t, \theta)} e^{-i\sprod{\theta}{\omega_n - n \delta}} 
	\frac{{\rm d}\theta}{\bfc(t + i \theta)}
	=
	\int_{U_\epsilon} e^{-\frac{n}{2} B_t(\theta, \theta)} f_n(\theta) g(\theta)
	\abs{\bpi_\Psi(\theta)}^2 {\: \rm d}\theta
\]
where
\[
	g(\theta) = \frac{1}{W_0(\Psi)} \sum_{w \in W_0(\Psi)}
	e^{-i\sprod{w \cdot \theta}{T_\Psi \eta}} \bfc_\Psi(-i w \cdot \theta).
\]
Because the function
\[
	\theta \mapsto
	\sum_{w \in W_0(\Psi)} (-1)^{\ell(w)} e^{-i\sprod{w \cdot \theta}{T_\Psi \eta+\rho_\Psi}} \bfb_\Psi(-iw\cdot \theta)
\]
is an anti-invariant exponential polynomial, it is divisible by the Weyl denominator
\[
	\Delta_\Psi(i\theta) = 
	\prod_{\alpha \in \Psi^{++}} 
	\Big(e^{i\sprod{\theta}{\alpha\spcheck}/2} - e^{-i\sprod{\theta}{\alpha\spcheck}/2}\Big).
\]
Hence, for $\theta \in U_\epsilon$,
\[
	\abs{g(\theta)}
	=
	\bigg|
	\frac{1}{\Delta_\Psi(i\theta)}
	\sum_{w \in W_0(\Psi)} (-1)^{\ell(w)} 
	e^{-i\sprod{w \cdot \theta}{T_\Psi \eta + \rho_\Psi}} \bfb_\Psi(-i w \cdot \theta)
	\bigg| \leq C.
\]
Finally, using the dominated convergence we can evaluate the limit
\begin{align*}
	&\lim_{n \to \infty} n^{\frac{r}{2} + \abs{\Psi^{++}}} 
	\int_{U_\epsilon} e^{n \varphi(t, \theta)} e^{-i\sprod{\theta}{\omega_n - n \delta_n}}
	\frac{{\rm d}\theta}{\bfc(t + i \theta)} \\
	&\qquad\qquad=
	\lim_{n \to \infty}
	\int_{U_{\sqrt{n} \epsilon}} e^{-\frac{1}{2} B_t(u, u)}
	f_n\big(n^{-1/2} u\big) g\big(n^{-1/2} u\big) \abs{\bpi_\Psi(u)}^2 {\: \rm d} u \\
	&\qquad\qquad=
	(2\pi)^r
	\calQ_\Psi(t) 
	\frac{G_0(\eta)}{\abs{\bfb_\Psi(0)}^2},
\end{align*}
proving the claim.

We now apply Claim \ref{clm:6} together with \eqref{eq:75} to the formula \eqref{eq:32} to get
\begin{equation}
	\label{eq:30}
	\lim_{n \to \infty} n^{\frac{r}{2}+\abs{\Psi^{++}}} e^{n\phi(\delta_n)} \calF_n(\omega_n) =
	(2\pi)^r
	\calQ_\Psi(t)
	\frac{G_0(\eta)}{\abs{\bfb_\Psi(0)}^2}.
\end{equation}
In view of Theorem \ref{thm:3}, the function \eqref{eq:46} is continuous, thus comparison \eqref{eq:30} with \eqref{eq:73}
gives
\[
	\sum_{\gamma \in \Gamma_\Psi} \eta^\gamma a_\gamma(t) 
	=
	c_\Psi
	\bigg(
	\prod_{\alpha \in \Phi^+\setminus \Psi^+}
    \frac{1 - e^{-\sprod{t}{\alpha\spcheck}}}{1 - q_\alpha^{-1} e^{-\sprod{t}{\alpha\spcheck}}}
    \bigg)
    \frac{G_0(\eta)}{\abs{\bfb_\Psi(0)}^2}.
\]
Our final step is to show how to control the error term in \eqref{eq:72}. Let us observe that for all $x, y, u \in
\mathfrak{a}$,
\begin{align*}
	\big|D_u \log \kappa(x) - D_u \log \kappa(y) \big| 
	&\leq
	\sup_{0 \leq t \leq 1} \abs{B_{x + t(y-x)}(u, x-y)} \\
	&\leq
	\sqrt{B_0(u, u) B_0(x-y, x-y)} \\
	&\leq
	C \norm{u} \cdot \norm{x - y}.
\end{align*}
Since for $\alpha \in \Phi^+$,
\[
	D_\alpha \log \kappa(s_n)
	=
	-D_{r_\alpha \alpha} \log \kappa(s_n)
	=-D_\alpha \log \kappa\big(s_n - \sprod{s_n}{\alpha\spcheck}\alpha\big),
\]
we obtain
\begin{align}
	\nonumber
	2 \sprod{\delta_n}{\alpha} &= D_\alpha \log \kappa(s_n) - 
	D_\alpha \log \kappa\big(s_n - \sprod{s_n}{\alpha\spcheck}\alpha\big) \\
	\label{eq:44}
	&\leq 
	C \sprod{s_n}{\alpha\spcheck}.
\end{align}
Therefore, \eqref{eq:70b} and \eqref{eq:16} imply that
\[
	\sprod{t_n}{\alpha\spcheck} \geq \sprod{s_n}{\alpha\spcheck} - \norm{h_n}
	\geq C^{-1} \xi, \qquad\text{for all}\qquad \alpha \in \Phi^+ \setminus \Psi^+.
\]
In particular, there is $C > 0$ such that
\[
    \prod_{\alpha \in \Phi^+\setminus \Psi^+}
    \frac{1 - \tau_{\alpha/2}^{-1/2} e^{-\sprod{t_n}{\alpha\spcheck}}}
	{1 - \tau_\alpha^{-1} \tau_{\alpha/2}^{-1/2} e^{-\sprod{t_n}{\alpha\spcheck}}}
    \geq C,
\]
and since (see \cite{a1, ascht})
\[ 
	C^{-1} \bpi_\Psi(\omega_n+\rho) \leq G_0(\omega_n),
\]
given $\gamma \in \Gamma_\Psi$, we can estimate
\[
	\omega_n^\gamma 
	\leq \bpi_\Psi(\omega_n + \rho) 
	\leq C \left(
    \prod_{\alpha \in \Phi^+\setminus \Psi^+}
	\frac{1 - \tau_{\alpha/2}^{-1/2} e^{-\sprod{t_n}{\alpha\spcheck}}}
	{1 - \tau_{\alpha}^{-1}\tau_{\alpha/2}^{-1/2} e^{-\sprod{t_n}{\alpha\spcheck}}}
    \right)
    \frac{G_0(\omega_n)}{\abs{\bfb_\Psi(0)}^2}.
\]
Hence,
\[
	\calF_n(\omega_n) = (2\pi)^r n^{-\frac{r}{2}-\abs{\Psi^{++}}} e^{-n \phi(\delta_n)} 
	\calQ_{\Psi}(t_n) \frac{G_0(\omega_n)}{\abs{\bfb_\Psi(0)}^2}
	\big(1 + E_n(\delta_n)\big),
\]
which completes the proof of Theorem \ref{thm:4}.
\end{proof}

The asymptotic in Theorem \ref{thm:4} is uniform on a large region with respect to $n$ and $v$, but it depends on the
implicit function $\delta \mapsto s(\delta)$. In most applications, one needs the asymptotic of the heat kernel in
the region where $\omega_n = o(n)$ accompanied by global upper estimates. For this reason we state the following corollary
which is a direct consequence of Theorem \ref{thm:4}.
\begin{corollary}
	\label{cor:1}
	Let $(\omega_n : n \in \NN)$ be a sequence of co-weights such that $V_{\omega_n}(O) \subseteq \supp p(n; \,\cdot\,)$.
	We assume that $\delta_n = n^{-1} \omega_n$ satisfies
	\[
		\lim_{n \to \infty} \sprod{\delta_n}{\alpha} = 0, \qquad\text{for all}\qquad \alpha \in \Phi.
	\]
	Then for any sequence of good vertices $(v_n : n \in \NN)$ such that $v_n \in V_{\omega_n}(O)$,
	\begin{equation}
		\label{eq:40}
		p(n; v_n) = n^{-\frac{r}{2}-\abs{\Phi^{++}}} P_{\omega_n}(0) \varrho^n e^{-n\phi(\delta_n)}
		\Big(C_0 + \calO\big(\norm{\delta_n}\big) + \calO\big(n^{-1}\big) \Big)
	\end{equation}
	where
	\[
		C_0 = W_0\big(q^{-1}\big) \frac{1}{\abs{\bfb_\Phi(0)}^2}
		\bigg(\frac{1}{2\pi}\bigg)^r \int_{\mathfrak{a}} e^{-\frac{1}{2} B_0(u, u)} \abs{\bpi_\Phi(u)}^2 {\: \rm d}u.
	\]
	The implied constants in \eqref{eq:40} are absolute.
\end{corollary}

\begin{remark}
	It is not possible to replace $\phi(\delta_n)$ by $\frac{1}{2} B_0^{-1}(\delta_n, \delta_n)$ without introducing an
	error term of a very different nature. Indeed, by \eqref{eq:29}, 
	\[
		\exp\big\{-n\phi(\delta_n)\big\}
		= 
		\exp\Big\{-\tfrac{n}{2}B_0^{-1}(\delta_n, \delta_n)\Big\} 
		\exp\big\{\calO\big(n \norm{\delta_n}^3\big)\big\}.
	\]
	If $\delta_n$ approaches $\partial \calM$ then $n \norm{\delta_n}^3$ cannot be small. Note that, the third power may
	be replaced by higher degree whenever the random walk has vanishing moments.
\end{remark}

\begin{remark}
	\label{rem:1}
	It is relatively easy to obtain a global upper bounds on $p(n; v)$, namely, by Claim \ref{clm:1}, for any
	$u \in \mathfrak{b}$ and $v \in V_\omega(O)$, we have
	\[
		p(n; v) = \bigg(\frac{1}{2\pi} \bigg)^r \chi_0(\omega)^{-\frac{1}{2}}
		\varrho^n \kappa(u)^n e^{-\sprod{u}{\omega}}
		\int_U \bigg(\frac{\kappa(u+i\theta)}{\kappa(u)}\bigg)^n e^{-i\sprod{\theta}{\omega}} 
		\frac{{\rm d}\theta}{\bfc(u+i\theta)}.
	\]
	Thus, 
	\begin{align}
		\nonumber
		p(n; v)
		&\leq
		C \chi_0(\omega)^{-\frac{1}{2}}
		\varrho^n \Big(\min\big\{\kappa(u) e^{-\sprod{u}{\delta}} : u \in \mathfrak{b} \big\}\Big)^n \\
		\label{eq:15}
		&=C \chi_0(\omega)^{-\frac{1}{2}}
		\varrho^n e^{-n \phi(\delta)}.
	\end{align}
\end{remark}

\subsection{Green functions}
\label{subsec:4.6}
In this section we prove the asymptotic formula for Green function of the random walk with the transition
probability $p$. Let us recall that Green function $G_\zeta$ is defined for $\zeta \in (0, \varrho^{-1}]$,
and $x, y \in V_P$ by the formula
\[
	G_\zeta(x, y) = \sum_{n \geq 0} \zeta^n p(n; x, y).
\]
We set $G_\zeta(x) = G_\zeta(O, x)$. 

We first treat the case $\zeta \in (0, \varrho^{-1})$. Let
\[
	\mathcal{C} = \big\{x \in \mathfrak{a} : \kappa(x) = (\zeta \varrho)^{-1} \big\}.
\]
For $u \in S^{r-1}$, the unit sphere in $\mathfrak{a}$ centered at the origin, there is the unique point
$s_u \in \mathcal{C}$ such that
\[
	\nabla \kappa(s_u) = \norm{ \nabla \kappa(s_u) } u.
\]
We have
\begin{theorem}
	\label{thm:6}
	Let $\Psi \subsetneq \Phi$. Suppose that $u = \norm{\omega}^{-1} \omega$ for $\omega \in P^+$
	satisfies
	\begin{subequations}
	\begin{equation}
		\label{eq:71a}
		\lim_{\norm{\omega} \to \infty} \sprod{u}{\alpha} = 0, \quad\text{for all}\quad \alpha \in \Psi,
	\end{equation}
	\begin{equation}
		\label{eq:71b}
		\sprod{u}{\alpha} \geq \xi, \quad\text{for all}\quad \alpha \in \Phi^+ \setminus \Psi^+,
	\end{equation}
	\end{subequations}
	for some $\xi > 0$. Then for all $x \in V_\omega(O)$,
	\begin{equation}
		\label{eq:76}
		G_\zeta(x) = 
		\norm{\omega}^{-\frac{r-1}{2}-\abs{\Psi^{++}}}
		\calP_\Psi(\omega) \calR_\Psi(u) e^{-\sprod{s_u}{\omega}}
		\big(1 + o(1)\big),
	\end{equation}
	as $\norm{\omega}$ tends to infinity where
	\[
		\calR_\Psi(u) = \sqrt{2 \pi} \norm{\nabla \log \kappa(s_u)}^{\frac{r-3}{2} +\abs{\Psi^{++}}} 
		\big(B_{s_u}^{-1}(u, u)\big)^{-\frac{1}{2}} \calQ_\Psi(s_u).
	\]
\end{theorem}
\begin{proof}
	Fix $u \in S^{r-1}$ and let $t_0 = \min \{t > 0: t^{-1} u \in \calM\}$. For $t > t_0$ we have $t^{-1} u \in \calM$,
	thus we may define $s_t = s\big(t^{-1} u\big)$. Consider a function on $(t_0, \infty)$ given by the formula
	\[
		\psi_u(t) = t \big(\log (\zeta \varrho) - \phi(t^{-1} u)\big).
	\]
	A simple calculation leads to
	\[
		\psi_u'(t) 
		= \log(\zeta \varrho) + \log \kappa(s_t), \qquad\text{and}\qquad
		\psi_u''(t) = -\frac{1}{t^3} B_{s_t}^{-1} (u,u).
	\]
	Hence, $\psi_u$ is concave in $(t_0, \infty)$. Since
	\[
		\lim_{t \to t_0} \kappa(s_t) = +\infty, \qquad\text{and}\qquad \lim_{t \to +\infty} \kappa(s_t) = 1,
	\]
	there is the unique maximum attained at $t_u > t_0$ satisfying
	\[
		0 = \psi_u'(t_u) = \log(\zeta \varrho) + \log \kappa(s_{t_u}).
	\]
	Because $\nabla \log \kappa(s_{t_u}) = t_u^{-1} u$, we conclude that $s_u = s_{t_u}$ and
	\begin{equation}
		\label{eq:27}
		\norm{\omega} \cdot \psi_u(t_u) = -\sprod{s_u}{\omega}.
	\end{equation}
	By compactness of $S^{r-1}$, there is $\delta > 0$ such that for all $u \in S^{r-1}$,
	\[
		\dist\big(t_u^{-1} u, \partial \calM\big) \geq 2 \delta.
	\]
	Hence, for all $t \in I_\delta$ where
	\[
		I_\delta = \Big\{t \in \RR : \big|t^{-1} - t_u^{-1} \big| \leq \delta \Big\},
	\]
	we have
	\[
		\dist\big(t^{-1} u, \partial \calM\big) \geq \delta,
	\]
	which entails that the mapping $I_\delta \ni t \mapsto s_t$ and all its derivatives are bounded independent of
	$u \in S^{r-1}$. Therefore, there is $C > 0$ such that for all $t \in I_\delta$ and $u \in S^{r-1}$,
 	\begin{equation}
		\label{eq:37}
		\bigg| \psi_u(t) - \psi_u(t_u) + \frac{1}{2 t_u^3} B_{s_u}^{-1}(u, u) \abs{t - t_u}^2 \bigg|
		\leq
		C \abs{t-t_u}^3.
	\end{equation}
	Moreover, $\psi_u$ is concave thus there is $c > 0$ such that for all $t > t_0$ and $u \in S^{r-1}$,
	\begin{align}
		\label{eq:74}
		\psi_u(t) - \psi_u(t_u) &\leq -\frac{1}{2 t_u^3}B_{s_u}^{-1}(u, u) \abs{t - t_u}^2 \\
		&\leq
		\label{eq:26}
		-2 c \abs{t -t_u}^2.
	\end{align}
	By a straightforward computation one can check that the function $\calQ_\Psi$ and all its derivatives are bounded on
	compact sets. Therefore, for all $t \in I_\delta$ and $u \in S^{r-1}$ satisfying \eqref{eq:71b} we can estimate
	\begin{align}
		\nonumber
		\Big|\calQ_\Psi\big((I-T_\Psi) s_t\big) - \calQ_\Psi(s_u) \Big|
		&\leq
		C\big( \norm{s_t - s_u} + \norm{T_\Psi s_u}\big) \\
		\label{eq:25}
		&\leq
		C\Big(\abs{t^{-1} - t_u^{-1}} + \sum_{\alpha \in \Psi^+} \sprod{u}{\alpha}\Big)
	\end{align}
	where in the last inequality we have used \eqref{eq:16}. Finally, by \eqref{eq:44}, there is $C > 0$ such that
	for all $u \in S^{r-1}$ satisfying \eqref{eq:71b},
	\[
		C^{-1} \leq \calQ_\Psi(s_u) \leq C.
	\]
	We are now ready to deal with Green function $G_\zeta(x)$. We write
	\[
		G_\zeta(x) = \sum_{n \in A} \zeta^n p(n; x) + \sum_{n \in B} \zeta^n p(n; x)
	\]
	where
	\[
		B=\Big\{n \in \mathbb{N}: \big|n - \norm{\omega} t_u \big| \geq \norm{\omega}^{\frac{3}{5}}, 
		\text{ and } n \geq \norm{\omega} t_0 \Big\},
	\]
	and
	\[
		A = \Big\{n \in \mathbb{N}: \big|n - t_u \norm{\omega} \big| < \norm{\omega}^\frac{3}{5} \Big\}.
	\]
	We can assume that
	\[
		\norm{\omega} \geq 32 \max\big\{t_u^5 : u \in S^{r-1}\big\} + \delta^{-\frac{5}{2}}.
	\]
	Let us first estimate the sum over $B$. Since (see \cite{a1, ascht})
	\[
		C \bigg(\prod_{\alpha \in \Phi^+} q_\alpha^{-\frac{1}{2}\sprod{\alpha}{\omega}} \bigg) 
		\leq \calP_\Psi(\omega),
	\]
	by \eqref{eq:15}, \eqref{eq:27} and \eqref{eq:26}, we get
	\begin{align*}
		\sum_{n \in B} \zeta^n p(n; x) 
		&\leq 
		C \calP_\Psi(\omega) \sum_{n \in B} e^{\norm{\omega} \cdot \psi_u(\norm{\omega}^{-1} n)} \\
		&\leq
		C \calP_\Psi(\omega) e^{-\sprod{s_u}{\omega}} 
		e^{-c \norm{\omega}^\frac{1}{5}} 
		\sum_{n \in \ZZ}
		\exp\left\{-c\norm{\omega} \cdot \big|\norm{\omega}^{-1} n - t_u \big|^2 \right\}. 
	\end{align*}
	In the light of
	\[
		\sum_{n \in \ZZ} 
        \exp\Big\{-c\norm{\omega} \cdot \big|\norm{\omega}^{-1} n - t_u \big|^2 \Big\}
		\leq
		2 \norm{\omega} \int_\RR \exp\big\{-c \norm{\omega} u^2 \big\} {\: \rm d} u
		\leq
		C \norm{\omega}^{\frac{1}{2}},
	\]
	we obtain
	\[
		\sum_{n \in B} \zeta^n p(n; x) \leq C \calP_\Psi(\omega) 
		e^{-\sprod{s_u}{\omega}} e^{-c \norm{\omega}^\frac{1}{5}}
		\norm{\omega}^{\frac{1}{2}}.
	\]
	To deal with the sum over $A$, we notice that for $n \in A$,
	\[
		\dist\big(n^{-1} \omega, \partial \calM\big) \geq \delta,
	\]
	which justifies the application of Theorem \ref{thm:4}. Hence,
	\[
		\sum_{n \in A} \zeta^n p(n; x) = \calP_{\Psi}(\omega) \sum_{n \in A} n^{-\frac{r}{2} - \abs{\Psi^{++}}}
		\calQ_\Psi\Big((I-T_\Psi) s\big(n^{-1} \omega\big) \Big)
		e^{\norm{\omega} \cdot \psi_u(\norm{\omega}^{-1} n)} 
		\Big(1 + E_n\big(n^{-1} \omega\big)\Big)
	\]
	where
	\[
		\big| E_n\big(n^{-1} \omega \big) \big| \leq C \sum_{\alpha \in \Psi^+} \frac{1}{n} \sprod{\omega+\rho}{\alpha}.
	\]
	Let us consider the following sum
	\[
		S(\omega) = \sum_{n \in A} n^{-\frac{r}{2} - \abs{\Psi^{++}}}
        e^{\norm{\omega} \cdot \psi_u(\norm{\omega}^{-1} n)}.
	\]
	For $n \in A$, we have
	\[
		\bigg| \frac{\omega}{n} - \frac{u}{t_u} \bigg|
		=
		\bigg| \frac{\norm{\omega}}{n} - \frac{1}{t_u} \bigg| \leq \norm{\omega}^{-\frac{2}{5}},
	\]
	thus, by \eqref{eq:25},
	\begin{align*}
		&
		\Big|
		\sum_{n\in A}
		n^{-\frac{r}{2} - \abs{\Psi^{++}}}
        \calQ_\Psi\Big((I-T_\Psi) s\big(n^{-1} \omega \big)\Big)
        e^{\norm{\omega} \cdot \psi_u(\norm{\omega}^{-1} n)}
		-
		\calQ_\Psi(s_u) S(\omega) \Big| \\
		&\qquad\qquad\qquad
		\leq
		C\Big(\norm{\omega}^{-\frac{2}{5}} + \sum_{\alpha \in \Psi^+} \sprod{u}{\alpha}\Big) S(\omega).
	\end{align*}
	Furthermore, we have
	\[
		\Big|
		\sum_{n\in A}
        n^{-\frac{r}{2} - \abs{\Psi^{++}}}
        \calQ_\Psi\Big((I-T_\Psi) s\big(n^{-1} \omega\big)\Big)
        e^{\norm{\omega} \cdot \psi_u(\norm{\omega}^{-1} n)}
		E_n\big(n^{-1} \omega\big)
		\Big|
		\leq
		C
		\Big(\norm{\omega}^{-1} + \sum_{\alpha \in \Psi^+} \sprod{u}{\alpha}\Big) S(\omega),
	\]
	because for $n \in A$,
	\[
        \big| E_n\big(n^{-1} \omega \big) \big| \leq
        C \Big(\norm{\omega}^{-1} + \sum_{\alpha \in \Psi^+} \sprod{u}{\alpha}\Big).
    \]
	Consequently, the problem reduces to establishing the asymptotic behavior of $S(\omega)$. To do so, let us introduce
	\[
		S_0(\omega) =
		\sum_{n \in A} 
		\exp\Big\{\tfrac{1}{2} \norm{\omega} \cdot \psi_u''(t_u) \cdot \big|t_u - \norm{\omega}^{-1} n \big|^2\Big\}.
	\]
	By \eqref{eq:37} and \eqref{eq:74}, we have
	\begin{align*}
		&
		\bigg|
		\exp\Big\{\norm{\omega} \cdot \psi_u\big(\norm{\omega}^{-1} n \big) \Big\}
		-
		\exp\Big\{\norm{\omega} 
		\Big(\psi_u(t_u) + \tfrac{1}{2} \psi_u''(t_u) \cdot \big|\norm{\omega}^{-1} n - t_u\big|^2\Big)\Big\}
		\bigg|\\
		&\qquad\qquad\leq
		C
		\norm{\omega}^{-\frac{1}{5}}
		\exp\Big\{\norm{\omega}
        \Big(\psi_u(t_u) + \tfrac{1}{2} \psi_u''(t_u) \cdot \big|\norm{\omega}^{-1} n - t_u\big|^2\Big)\Big\}.
	\end{align*}
	Furthermore, by the mean value theorem, we can estimate
	\[
		\Big|n^{-\frac{r}{2} - \abs{\Psi^{++}}} - \big(t_u \norm{\omega}\big)^{-\frac{r}{2}-\abs{\Psi^{++}}} \Big|
		\leq
		C \big(t_u \norm{\omega}\big)^{-\frac{r}{2}-\abs{\Psi^{++}}} \norm{\omega}^{-\frac{2}{5}},
	\]
	because $\norm{\omega}^{-\frac{1}{5}} \leq \frac{1}{2} t_u$. Hence,
	\[
		\Big|
		S(\omega) - e^{-\sprod{s_u}{\omega}} \big(t_u \norm{\omega}\big)^{-\frac{r}{2} -\abs{\Psi^{++}}} S_0(\omega)
		\Big|
		\leq
		C \norm{\omega}^{-\frac{2}{5}}  
		e^{-\sprod{s_u}{\omega}} \big(t_u \norm{\omega}\big)^{-\frac{r}{2} -\abs{\Psi^{++}}} S_0(\omega).
	\]
	Lastly, we replace the sum $S_0(\omega)$ by the corresponding integral, that is
	\[
		I(\omega) = \int_\RR
		\exp\Big\{\tfrac{1}{2} \norm{\omega} \cdot \psi_u''(t_u) \cdot \big|t_u - \norm{\omega}^{-1} t \big|^2\Big\}
        {\:\rm d}t
		=
		\sqrt{2\pi} \norm{\omega}^{\frac{1}{2}} \big(-\psi_u''(t_u)\big)^{-\frac{1}{2}}.
	\]
	For $n \in A$ such that $n < \norm{\omega} t_u$ and $n \leq t \leq \min\{n+1, \norm{\omega} t_u\}$, we have
	\begin{align*}
		&\Big|
		\exp\Big\{\tfrac{1}{2} \norm{\omega} \cdot \psi_u''(t_u) \cdot \big|t_u - \norm{\omega}^{-1} n \big|^2\Big\}
		-
		\exp\Big\{\tfrac{1}{2} \norm{\omega} \cdot \psi_u''(t_u) \cdot \big|t_u - \norm{\omega}^{-1} t \big|^2\Big\}
		\Big| \\
		&\qquad\qquad\leq
		C \norm{\omega}^{-\frac{2}{5}}
		\exp\Big\{\tfrac{1}{2} \norm{\omega} \cdot \psi_u''(t_u) \cdot \big|t_u - \norm{\omega}^{-1} t \big|^2\Big\}.
	\end{align*}
	Analogously, for $n \in A$ such that $n \geq \norm{\omega} t_u$ and $n \leq t \leq n+1$,
	\begin{align*}
        &\Big|
        \exp\Big\{\tfrac{1}{2} \norm{\omega} \cdot \psi_u''(t_u) \cdot \big|t_u - \norm{\omega}^{-1} n \big|^2\Big\}
        -
        \exp\Big\{\tfrac{1}{2} \norm{\omega} \cdot \psi_u''(t_u) \cdot \big|t_u - \norm{\omega}^{-1} t \big|^2\Big\}
        \Big| \\
        &\qquad\qquad\leq
        C \norm{\omega}^{-\frac{2}{5}}
        \exp\Big\{\tfrac{1}{2} \norm{\omega} \cdot \psi_u''(t_u) \cdot \big|t_u - \norm{\omega}^{-1} n \big|^2\Big\}.
    \end{align*}
	Hence, we deduce that 
	\[
		\big|S_0(\omega) - I(\omega) \big|
		\leq
		C
		\norm{\omega}^{-\frac{2}{5}} I(\omega).
	\]
	By putting these estimates together, we obtain
	\[
		S(\omega) = \sqrt{2 \pi} \big(t_u \norm{\omega}\big)^{-\frac{r-1}{2}-\abs{\Psi^{++}}} 
		\big(t_u^{-2} B_{s_u}^{-1}(u, u)\big)^{-\frac{1}{2}} e^{-\sprod{s_u}{\omega}}
		\big(1+o(1)\big),
	\]
	which entails \eqref{eq:76} because $t_u = \norm{\nabla \log \kappa(s_u)}^{-1}$.
\end{proof}

We now turn to the case $\zeta = \varrho^{-1}$.
\begin{theorem}
	\label{thm:7}
	For all $x \in V_\omega(O)$,
	\[
		G_{\varrho^{-1}}(x) = 
		P_{\omega}(0) \big(B_0^{-1}(\omega, \omega)\big)^{-\frac{r}{2}-\abs{\Phi^{++}}+1}
		\big(D_0 + o(1) \big),
	\]
	as $\norm{\omega}$ tends to infinity where
	\[
		D_0 =
		2^{\frac{r}{2}+\abs{\Phi^{++}}-1}
		\Gamma\big(\tfrac{r}{2}+\abs{\Phi^{++}} - 1\big)
		W_0\big(q^{-1}\big) \frac{1}{\abs{\bfb_\Phi(0)}^2}
        \bigg(\frac{1}{2\pi}\bigg)^r \int_{\mathfrak{a}} e^{-\frac{1}{2} B_0(u, u)} \abs{\bpi_\Phi(u)}^2 {\: \rm d}u.
	\]
\end{theorem}
\begin{proof}
	Let $n_0 = \min\{n \in \NN : n^{-1} \omega \in \calM\}$. We write
	\[
		G_{\varrho^{-1}} (x) = \sum_{n \in A} \varrho^{-n} p(n; x) + \sum_{n \in B} \varrho^{-n} p(n; x)
	\]
	where
	\[
		A = \Big\{n \in \NN : n \geq \norm{\omega}^{\frac{7}{4}} \Big\},
	\]
	and
	\[
		B = \Big\{n \in \NN : \norm{\omega}^{\frac{7}{4}} > n \geq n_0\Big\}.
	\]
	We first treat the sum over $B$. By \eqref{eq:15} and \eqref{eq:33},
	\begin{align*}
		\sum_{n \in B} \varrho^{-n} p(n; x) &\leq C \sum_{n \in B} e^{-2 c n^{-1} \norm{\omega}^2} \\
		&\leq C e^{-2 c \norm{\omega}^{\frac{1}{4}}} \norm{\omega}^{\frac{7}{4}} \\
		&\leq C e^{-c\norm{\omega}^\frac{1}{4}}.
	\end{align*}
	To the sum over $A$, we apply Corollary \ref{cor:1}. To justify its use, we observe that for $n \in A$,
	\[
		\frac{\norm{\omega}}{n} \leq \norm{\omega}^{-\frac{3}{4}}.
	\]
	Hence,
	\[
		\sum_{n \in A} \varrho^{-n} p(n; x) = P_\omega(0) \sum_{n \in A} n^{-\frac{r}{2} - \abs{\Phi^{++}}} 
		e^{-n\phi(n^{-1} \omega)} 
		\Big(C_0 + E_n\big(n^{-1} \omega\big) \Big)
	\]
	where
	\[
		\abs{E_n\big(n^{-1} \omega\big)} \leq C n^{-1} \big(\norm{\omega} + 1\big).
	\]
	Since for $n \in A$, 
	\[
		\abs{E_n\big(n^{-1} \omega\big)} \leq 2 C \norm{\omega}^{-\frac{3}{4}},
	\]
	it is enough to find the asymptotic of the sum
	\[
		\sum_{n \in A} n^{-\frac{r}{2}-\abs{\Phi^{++}}} e^{-n \phi(n^{-1} \omega)}. 
	\]
	To do so, let us introduce
	\[
		S_0(\omega) = \sum_{n \in A} n^{-\frac{r}{2}-\abs{\Phi^{++}}}
		\exp\Big\{-\tfrac{1}{2 n} B_0^{-1}(\omega,\omega) \Big\}.
	\]
	Because for $n \in A$,
	\begin{align*}
		\Big|
		n \phi\big(n^{-1} \omega\big) - \tfrac{1}{2 n} B_0^{-1}(\omega, \omega)
		\Big|
		\leq
		C
		\frac{\norm{\omega}^3}{n^2}
		\leq
		C \norm{\omega}^{-\frac{1}{2}},
	\end{align*}
	we see that
	\[
		\Big|
		\sum_{n \in A} n^{-\frac{r}{2}-\abs{\Psi^{++}}} e^{-n \phi(n^{-1} \omega)} - S_0(\omega)
		\Big|
		\leq
		C \norm{\omega}^{-\frac{1}{2}} S_0(\omega).
	\]
	By taking $n \in A$ and $n \leq t \leq n+1$, we can estimate
	\[
		\Big|
		\exp\Big\{-\tfrac{1}{2 n} B_0^{-1}(\omega, \omega)\Big\}
		-
		\exp\Big\{-\tfrac{1}{2 t} B_0^{-1}(\omega, \omega)\Big\}
		\Big|
		\leq
		C
		\norm{\omega}^{-\frac{3}{2}} \exp\Big\{-\tfrac{1}{2 t} B_0^{-1}(\omega, \omega)\Big\},
	\]
	and
	\[
		\big| n^{-\frac{r}{2}-\abs{\Phi^{++}}} - t^{-\frac{r}{2} - \abs{\Phi^{++}}} \big|
		\leq
		C
		t^{-\frac{r}{2} - \abs{\Phi^{++}}} \norm{\omega}^{-\frac{7}{4}},
	\]
	thus
	\begin{align*}
		&\bigg|
		S_0(\omega) - 
		\int_{\norm{\omega}^{-\frac{7}{4}}}^{\infty}
		t^{-\frac{r}{2}-\abs{\Phi^{++}}}
		\exp\Big\{-\tfrac{1}{2 t} B_0^{-1}(\omega, \omega) \Big\}
		{\: \rm d} t
		\Big| \\
		&\qquad\qquad
		\leq
		C \norm{\omega}^{-\frac{3}{2}}
		\int_{\norm{\omega}^{-\frac{7}{4}}}^\infty
		t^{-\frac{r}{2}-\abs{\Phi^{++}}}
		\exp\Big\{-\tfrac{1}{2 t} B_0^{-1}(\omega, \omega) \Big\}
		{\: \rm d} t.
	\end{align*}
	Finally, a straightforward computation shows that
	\[
		\int_{\norm{\omega}^{-\frac{7}{4}}}^\infty
		t^{-\frac{r}{2}-\abs{\Phi^{++}}}
		\exp\Big\{-\tfrac{1}{2 t} B_0^{-1}(\omega, \omega) \Big\}
        {\: \rm d} t
		=
		\big(B_0^{-1}(\omega, \omega)\big)^{-\frac{r}{2}-\abs{\Phi^{++}} + 1}
		\big(c_0 + o(1)\big)
	\]
	where
	\[
		c_0 = 2^{\frac{r}{2}+\abs{\Phi^{++}}-1} \Gamma\big(\tfrac{r}{2} + \abs{\Phi^{++}} - 1\big),
	\]
	which completes the proof.
\end{proof}

\appendix
\section{Asymptotic in the exceptional case}
\label{app:a}
In the appendix we indicate the necessary changes to the proof of Theorem \ref{thm:4} in the exceptional case, 
that is when $\tau_\alpha < 1$ for some $\alpha \in \Phi$. Then the root system $\Phi$ is $\text{BC}_r$ and
$q_r < q_0$. In view of the inversion formula \eqref{eq:81}, for $v_n \in V_{\omega_n}(O)$,
\begin{align*}
	p(n; v_n) 
	&= \bigg(\frac{1}{2\pi} \bigg)^r \frac{W_0(q^{-1})}{|W_0|} 
	\int_{U_0} \big(h_{i\theta}(A) \big)^n \overline{P_{\omega_n}(i\theta)} \frac{{\rm d} \theta}{|\bfc(i\theta)|^2} \\
	&\phantom{=}+
	\bigg(\frac{1}{2\pi}\bigg)^{r-1} \frac{W_0(q^{-1})}{|W_0'|}
	\int_{U_1} \big(h_{i\theta}(A) \big)^n \overline{P_{\omega_n}(i\theta)} \frac{{\rm d} \theta}{\phi_1(i\theta)}.
\end{align*}
Using $W_0$-invariance of the integrand and the definition of $P_{\omega_n}$ we can write
\begin{align*}
	p(n; v_n)
	&=\chi_0(\omega_n)^{-\frac{1}{2}} \bigg(\frac{1}{2\pi}\bigg)^r \int_{U_0} (h_{i\theta}(A))^n
	e^{-i\sprod{\theta}{\omega_n}} \frac{{\rm d}\theta}{\bfc{(i\theta)}} \\
	&\phantom{=}+
	\chi_0(\omega_n)^{-\frac{1}{2}}\bigg(\frac{1}{2\pi}\bigg)^{r-1} 
	\sum_{j = 1}^r \int_{U_j} (h_{i\theta}(A))^n e^{-\sprod{\theta}{\omega_n}} 
	\frac{{\rm d} \theta}{\tilde{\bfc}_j(i\theta)}
\end{align*}
where
\[
	\tilde{\bfc}_j(z_1, \ldots, z_{j-1}, -v, z_{j+1}, \ldots, z_r) = 
	\lim_{z_j \to -v} \frac{\bfc(z_1 e_1 + \ldots + z_r e_r)}{1 + b^{-1} e^{-z_j}}.
\]
For $u > 0$, if $u \neq -\log b$, we denote by $\gamma_u$ the line segment $u + i [-\pi/2, 3\pi/2]$, otherwise 
\[
	\gamma_u(t) = -\log b+
	\begin{cases}
		i t & \text{if } t \in [-\pi/2, \pi-\tau], \\
		i \pi - 
		\tau \exp\left( i\pi \left(1 - \frac{t-\pi}{2\tau}\right)\right) & \text{if } t \in (\pi-\tau, \pi + \tau],\\
		i t & \text{if } t \in (\pi+t, 3 \pi/2]
	\end{cases}
\]
where $0 < 4 \tau < \min\{-\log b, \pi\}$. Let
\[
	H = \left\{z \in \CC : \Re z > \min\{-\log q_1, -\log a\} \right\}.
\]
We set $s_n = s(\delta_n)$. Let $j \in \{1, \ldots, r\}$ and fix $z_1, \ldots, z_{j-1}, z_{j+1}, \ldots, z_r \in 
\{z \in H : |z + v| \geq \tau\}$. We consider the function
\begin{equation}
	\label{eq:82}
	H \ni z_j \mapsto (h_z(A))^n e^{-\sprod{z}{\omega_n}} \frac{1}{\bfc(z)}
\end{equation}
where $z = (z_1, \ldots, z_r)$. Since $b < 1$, the mapping \eqref{eq:82} is meromorphic in $H$ with a pole at $-v$. 
Moreover, it is $2\pi i$-periodic. Therefore, if $0 \leq s_{n; j} \leq -\log b$ then
\begin{align*}
	\int_{\gamma_0} (h_{z}(A))^n
	e^{-\sprod{z}{\omega_n}} 
	\frac{{\rm d} z_j}{\bfc(z)}
	=
	\int_{\gamma_{s_{n; j}}} (h_{z}(A))^n
	e^{-\sprod{z}{\omega_n}} \frac{{\rm d} z_j}{\bfc(z)},
\end{align*}
otherwise
\begin{equation}
	\label{eq:83}
	\begin{aligned}
	\int_{\gamma_0} (h_{z}(A))^n
	e^{-\sprod{z}{\omega_n}}
	\frac{{\rm d} z_j}{\bfc(z)}
	&=
	\int_{\gamma_{s_{n; j}}} (h_{z}(A))^n
	e^{-\sprod{z}{\omega_n}} \frac{{\rm d} z_j}{\bfc(z)} \\
	&\phantom{=}
	- 2\pi i
	\lim_{z_j \to -v} 
	(h_{z}(A))^n
	e^{-\sprod{z}{\omega_n}} \frac{1 + b^{-1} e^{-z_j}} {\bfc(z)}.
	\end{aligned}
\end{equation}
Since $q_1 > 1$, we must have $q_1 b\geq 1$, see \cite[Lemma 5.6]{park2} for details. Therefore, the second term
in \eqref{eq:82} as a function of $z_k$ is holomorphic function in $H$ for $k = 1, \ldots, j-1, j+1, \ldots, r$.
Hence, by repeated change of the contour of integration, we get
\begin{equation}
	\label{eq:84}
	\begin{aligned}
	p(n; v_n)
	&=
	\chi_0(\omega_n)^{-\frac{1}{2}} \bigg(\frac{1}{2\pi i}\bigg)^r
	\int_{\gamma_{s_{n; 1}}} \cdots \int_{\gamma_{s_{n; r}}} (h_z(A))^n
	e^{-\sprod{z}{\omega_n}}
	\frac{{\rm d} z}{\bfc(z)} \\
	&\phantom{=}+
	\chi_0(\omega_n)^{-\frac{1}{2}} \bigg(\frac{1}{2\pi}\bigg)^{r-1}
	\sum_{j: s_{n; j} \leq -\log b}
	\int_{U_j} (h_{\hat{s}_{n; j}+i\theta}(A))^n e^{-\sprod{\hat{s}_{n; j} + i\theta}{\omega_n}} 
	\frac{{\rm d} \theta}{\tilde{\bfc}_j(\hat{s}_{n; j} + i\theta)}
	\end{aligned}
\end{equation}
where $\hat{s}_{n; j}=(s_{n; 1}, \ldots, s_{n; j-1}, 0, s_{n; j+1}, \ldots, s_{n; r})$. 

Let us consider the first integral in \eqref{eq:84}. Select $\epsilon$ satisfying \eqref{eq:52} and
\eqref{eq:38} and let $U_\epsilon = [-\epsilon, \epsilon]^r$. Every $z \in \gamma_{s_{n; 1}} \times \ldots \times
\gamma_{s_{n; r}}$ can be written as $z = s_n + x + i \theta$ with $\norm{x} \leq \tau$ and $\theta \in U_0$. Hence, if
$\theta \in U_0 \setminus U_\epsilon$, by \eqref{eq:20} and Claim \ref{clm:2},
\begin{align*}
	1 - \bigg|\frac{\kappa(s_n+x+i\theta)}{\kappa(s_n)} \bigg|^2
	&=
	2 \sum_{v, v' \in \calV} \scoef{s_n+x}{v} \scoef{s_n+x}{v'} 
	\bigg(\sin\bigg\langle\frac{\theta}{2}, v-v'\bigg\rangle\bigg)^2 \\
	&\geq
	2 \xi
	\sum_{v, v' \in \calV}  \scoef{s_n+x}{v} \scoef{s_n+x}{v'}
	\geq
	C \cdot \scoef{s_n}{v'}.
\end{align*}
Thus, by Theorem \ref{thm:2},
\[
	\bigg|\frac{\kappa(s_n+x+i\theta)}{\kappa(s_n)} \bigg|^2
	\leq
	e^{-C \dist(\delta_n, \partial \calM)^{\eta}},
\]
and so
\[
	\frac{1}{i^r}
    \int_{\gamma_{s_{n; 1}}} \cdots \int_{\gamma_{s_{n; r}}} (h_z(A))^n
	e^{-\sprod{z}{\omega_n}}
	\frac{{\rm d} z}{\bfc(z)}
	=
	e^{-n \phi(\delta_n)} \big( F_n(s_n) + E_n(\delta_n)\big)
\]
where $F_n$ is given by the formula \eqref{eq:85}, and
\[
	|E_n(\delta_n)| \leq C \exp\left\{-C'n \dist(\delta_n, \partial \calM)^{\eta} \right\}.
\]
We can now repeat the reasoning from Theorem \ref{thm:4} to obtain the asymptotic of $F_n(s_n)$. Hence, it remains to show
that the second term in \eqref{eq:84} is negligible, that is
\begin{equation}
	\label{eq:86}
	\bigg|
	\int_{U_j}
	(h_{\hat{s}_{n; j}+i\theta}(A))^n e^{-\sprod{\hat{s}_{n; j} + i\theta}{\omega_n}}
	\frac{{\rm d} \theta}{\tilde{\bfc}_j(\hat{s}_{n; j} + i\theta)}
	\bigg| 
	\leq
	C \varrho^n e^{-n \phi(\delta_n)} \exp{-C' \dist(\delta_n, \partial \calM)^\eta}
\end{equation}
provided that $s_{n; j} \leq -\log b$. To do so we need the estimate on $P_\lambda(z)$ if $z_1 = v$. We
start with the following theorem.
\begin{theorem}
	\label{thm:5}
	Suppose that $b = \sqrt{q_r/q_0} < 1$. Then for each $\lambda \in P^+$ and $u+i\theta \in \CC^{r-1}$,
	\begin{equation}
		\label{eq:87}
		\big|
		P_{\lambda}(u+i\theta, v)
		\big|
		\leq
		P_{\lambda}(u, 0).
	\end{equation}
	where $v = \log b - i \pi$.
\end{theorem}
\begin{proof}
	Let us consider the simples case $r = 1$, that is a semi-homogeneous tree. Then for $\lambda = k \lambda_1$,
	$k \in \NN$,
	\[
		P_{\lambda}(z) = \frac{\chi_0(\lambda)^{-\frac{1}{2}}}{1+q_1^{-1}} 
		\big(
		e^{k z} c(z) + e^{-k z} c(-z)
		\big)
	\]
	where
	\[
		c(z) = \frac{(1-a^{-1} e^{-z})(1 + b^{-1} e^{-z})}{1 - e^{-2 z}}
	\]
	and $a = \sqrt{q_0 q_1}$, $b = \sqrt{q_1 / q_0}$. Hence,
	\[
		P_{\lambda}(v) = \frac{\chi_0(\lambda)^{-\frac{1}{2}}}{1+q^{-1}} (-b)^k \big(1+a^{-1} b^{-1}\big).
	\]
	Since
	\[
		P_{\lambda}(z) = \frac{\chi_0(\lambda)^{-\frac{1}{2}}}{1+q^{-1}} 
		\bigg(\frac{e^{(k+1)z}-e^{-(k+1)z}}{e^z - e^{-z}} 
		+ (b^{-1} - a^{-1}) \frac{e^{k z}-e^{-k z}}{e^z - e^{-z}} 
		+ \frac{e^{(k-1)z}-e^{-(k-1)z}}{e^z - e^{-z}}
		\bigg),
	\]
	we easily get
	\[
		P_{\lambda}(0) = \frac{\chi_0(\lambda)^{-\frac{1}{2}}}{1+q^{-1}} k \big(1 + b^{-1} - a^{-1}\big).
	\]
	Thus
	\[
		\big|P_{\lambda}(v)\big| 
		\leq 
		\frac{\chi_0(\lambda)^{-\frac{1}{2}}}{1+q^{-1}} \big(1 + a^{-1} b^{-1}\big) \\
		\leq 
		P_{\lambda}(0).
	\]
	For $r \geq 2$, we use the integral representation of Macdonald spherical functions. Namely, there is a measure $\nu$
	on the maximal boundary $\Omega$ of the affine building $\mathscr{X}$ such that for any $x \in V_{\lambda}(O)$,
	\begin{equation}
		\label{eq:89}
		P_\lambda(z) = \int_\Omega \chi_0\big(h(O, x; \omega)\big)^{\frac{1}{2}} e^{\sprod{z}{h(O, x; \omega)}} 
		\nu({\rm d} \omega)
	\end{equation}
	where $h(O, x; \omega)$ is the horocycle function, see \cite[Section 3]{park2}. Furthermore, in view of
	\cite[Section 4]{mz1}, we can decompose $\Omega$ as a disjoint union
	\[
		\Omega = \bigsqcup_{\eta \in \Omega_r} \partial \mathbb{T}_\eta
	\]
	where each $\partial \mathbb{T}_\eta$ denotes the maximal boundary of a semi-homogeneous tree with parameters 
	$(q_0, q_r)$. On $\Omega_r$ there is a probability measure $\mu_r$ such that
	\begin{align*}
		P_{\lambda}(z) = 
		\int_{\Omega_r} \int_{\partial \mathbf{T}_\eta} 
		\chi_0\big(h(O, x; \omega)\big)^{\frac{1}{2}} e^{\sprod{z}{h(O, x; \omega)}}
		\nu_\eta({\rm d} \omega) 
		\mu_r({\rm d} \eta)
	\end{align*}
	where $\nu_\eta$ is the probability measure on $\mathbb{T}_\eta$ determined by \eqref{eq:89}. For a fixed
	$\eta \in \Omega_r$ the mapping
	\[
		\partial \mathbb{T}_\eta \ni \omega \mapsto Q\big(h(O, x; \omega)\big) 
		= \sum_{j = 1}^{r-1} \sprod{h(O, x; \omega)} {e_j}
	\]
	is constant (see \cite[Proposition 4.13]{mz1}). Moreover,
	\[
		 \partial \mathbb{T}_\eta \ni \omega \mapsto \sprod{h(O, x; \omega)}{e_r}
	\]
	is the horocycle function between the projections of $O$ and $x$ onto $\mathbb{T}(\eta)$
	(see \cite[Proposition 4.13]{mz1}). Then by the first part of the proof we have
	\[
		\bigg|
		\int_{\partial \mathbb{T}_\eta} \chi_{0, r}\big(\sprod{h(O, x; \omega)}{e_r}\big)^{\frac{1}{2}} 
		(-b)^{\sprod{h(O, x; \omega)}{e_r}}
		\nu_\eta({\rm d} \omega)
		\bigg|
		\leq
		\int_{\partial \mathbb{T}_\eta} 
		\chi_{0, r}\big(\sprod{h(O, x; \omega)}{e_r}\big)^{\frac{1}{2}} \nu_\eta({\rm d} \omega)
	\]
	where we have set
	\[
		\chi_{0, r}(k) = \tau_{\alpha_r}^k \tau_{2\alpha_r}^{2k}, \qquad k \in \ZZ.
	\]
	Hence, 
	\[
		\big|P_{\lambda}(u+i\theta, v)\big|
		\leq
		\int_{\Omega_r} \int_{\partial \mathbb{T}_\eta} 
		\chi_0\big(h(O, x; \omega)\big)^{\frac{1}{2}} 
		e^{\sprod{(x, 0)}{h(O, x; \omega)}}
		\nu_\eta({\rm d} \omega) \: \mu_r({\rm d} \eta)
		=
		P_{\lambda}(u, 0)
	\]
	and the theorem follows.
\end{proof}
In the next step we improve the estimate \eqref{eq:87}.
\begin{theorem}
	\label{thm:8}
	Suppose that $b = \sqrt{q_r/q_0} < 1$. Then for each $\lambda \in P^+$ there is
	$c_\lambda > 0$ so that for all $u+i\theta \in \CC^{r-1}$,
	\[
		\big|P_\lambda(u+i\theta, v)\big| \leq P_\lambda(u, 0) - \frac{c_\lambda}{P_{\lambda}(u, 0)}.
	\]
	where $v = \log b - i \pi$.
\end{theorem}
\begin{proof}
	Let us first show that
	\begin{equation}
		\label{eq:88}
		\big|P_{\lambda_1}(u+ i\theta, v)\big| \leq P_{\lambda_1}(u, 0) - c_{\lambda_1}
	\end{equation}
	for some $c_{\lambda_1} > 0$. Indeed, by \cite[Lemma B.3.2]{park}
	\[
		P_{\lambda_1}(z) = \frac{1}{N_{\lambda_1}}
		\Big(a_1 + a_2 \sum_{j = 1}^r \big(e^{z_j} + e^{-z_j}\big)\Big)
	\]
	where $a_1 = (q_0 - 1)(1+q_1+\ldots + q_1^{r-1})$, $a_2 = \sqrt{q_0 q_r} q_1^{r-1}$. Thus
	\[
		P_{\lambda_1}(u+i\theta, v) = \frac{1}{N_{\lambda_1}}
		\bigg(
		a_1 - a_2 \frac{q_0 + q_r}{\sqrt{q_0 q_r}} + a_2 \sum_{j = 1}^{r-1} 
		\big(e^{u_j+i\theta_j} + e^{-u_j-i\theta_j}\big)
		\bigg).
	\]
	If
	\[
		a_1 - a_2 \frac{q_0+q_r}{\sqrt{q_0+q_r}} \geq 0, 
	\]
	then
	\begin{align*}
		\big|P_{\lambda_1}(u+i\theta, v)\big| 
		&\leq 
		\frac{1}{N_{\lambda_1}}
		\bigg(a_1 - a_2 \frac{q_0 + q_r}{\sqrt{q_0 q_r}} + a_2 \sum_{j=1}^{r-1}
		\big(e^{u_j} + e^{-u_j}\big)\bigg)\\ 
		&= 
		P_{\lambda_1}(u, 0) - 2 a_2 - a_2 \frac{q_0 + q_r}{\sqrt{q_0 q_r}}.
	\end{align*}
	Otherwise, by \cite[Theorem B.3.3]{park},
	\[
		a_1 - a_2\frac{q_0+q_r}{\sqrt{q_0+q_r}} > -a_1 -2a_2,
	\]
	thus
	\begin{align*}
		\big|P_{\lambda_1}(u+i\theta, v)\big| 
		&\leq \frac{1}{N_{\lambda_1}}\bigg(a_2 \frac{q_0+q_r}{\sqrt{q_0+q_r}} - a_1 + a_2 \sum_{j = 1}^{r-1} 
		\big(e^{u_j} + e^{-u_j}\big)\bigg)\\
		&\leq P_{\lambda_1}(u, 0) - 2a_1 - 2 a_2 + a_2 \frac{q_0+q_r}{\sqrt{q_0+q_r}}.
	\end{align*}
	Now, by the triangularity condition for Macdonald spherical functions, and estimates \eqref{eq:87} and Theorem \ref{thm:5},
	we get
	\begin{align*}
		\big|P_\lambda(u+i\theta, v)\big|^2 
		&\leq \sum_{\mu \in P^+} a_{\lambda, \lambda; \mu} \big|P_\mu(u+ i\theta, v) \big| \\
		&\leq \sum_{\mu \in P^+} a_{\lambda, \lambda; \mu} P_\mu(u, 0) - a_{\lambda, \lambda; \lambda_1} c_{\lambda_1} \\
		&=(P_\lambda(u, 0))^2 - a_{\lambda, \lambda; \lambda_1} c_{\lambda_1},
	\end{align*}
	which completes the proof because by \cite[Lemma B.3.4]{park} we have $a_{\lambda, \lambda; \lambda_1} > 0$.
\end{proof}

We return to proving \eqref{eq:86}. Since the random walk has finite range, by Theorem \ref{thm:8}, we easily get
\[
	\big|h_{(v, x+i\theta)}(A) \big| \leq h_{(0, x)}(A) - \frac{c}{h_{(0, x)}(A)}.
\]
Using $W_0$-invariance, for any $t \in \RR$ we have
\[
	h_{(t, x)}(A) = \tfrac{1}{2} 
	\sum_{\mu \in \calV} c_\mu \big(e^{t \sprod{e_1}{\mu}}+e^{-t\sprod{e_1}{\mu}}\big) e^{\sprod{x}{\mu}}
	\geq 
	\sum_{\mu \in \calV} c_\mu e^{\sprod{x}{\mu}} = h_{(0, x)}(A).
\]
Therefore, for $\theta \in U_j$,
\[
	\big|h_{(\hat{s}_{n; j} +i \theta)}(A) \big| \leq h_{s_n}(A)\bigg(1 - \frac{c}{h_{s_n}^2(A)}\bigg),
\]
thus, by Theorem \ref{thm:2},
\begin{align*}
	\big|h_{(\hat{s}_{n; j} +i \theta)}(A) \big|^n
	&\leq (h_{s_n}(A))^n \exp\left(-c \frac{n}{h_{s_n}^2(A)} \right) \\
	&\leq (h_{s_n}(A))^n \exp\left\{-C' n \dist(\delta_n, \partial M)^{2\eta}\right\}.
\end{align*}
Since
\[
	(h_{s_n}(A))^n e^{-\sprod{\hat{s}_{n; j}}{\omega_n}}
	\leq \varrho^n e^{-n \phi(\delta_n)} e^{\log b \sprod{e_j}{\omega_n}},
\]
we obtain \eqref{eq:86}. This completes the proof of Theorem \ref{thm:3} in the exceptional case.

\begin{bibliography}{buildings}
\bibliographystyle{plain}
\end{bibliography}
\end{document}